\documentclass[reqno, 11pt, a4paper]{amsart}
\usepackage{latexsym,ifthen,xspace}

\usepackage{dsfont}

\usepackage{latexsym,ifthen,xspace}
\usepackage{amsmath,amssymb,amsthm}
\usepackage{dsfont}
\usepackage{mathtools}
\usepackage{enumerate, calc, bbm}
\usepackage[usenames,dvipsnames]{xcolor}
\usepackage[colorlinks=true, pdfstartview=FitV, linkcolor=blue,
citecolor=blue, urlcolor=blue, pagebackref=false]{hyperref}
\usepackage{orcidlink}
\usetikzlibrary{calc}
\usepackage[text={33pc,605pt},centering]{geometry}    

\usepackage{subcaption}
\usepackage{graphicx}
\usepackage{array}
\usepackage{xcolor}
\newtheorem{theorem}{Theorem}[section]
\newtheorem{lemma}[theorem]{Lemma}
\newtheorem{proposition}[theorem]{Proposition}

\theoremstyle{definition}
\newtheorem{definition}[theorem]{Definition}
\newtheorem{assumption}[theorem]{Assumption}
\newtheorem{example}[theorem]{Example}

\theoremstyle{remark}
\newtheorem{remark}[theorem]{Remark}

\numberwithin{equation}{section}

\DeclareMathAlphabet{\mathsl}{OT1}{cmss}{m}{sl}
\SetMathAlphabet{\mathsl}{bold}{OT1}{cmss}{bx}{sl}

\DeclarePairedDelimiter{\abs}{\lvert}{\rvert}
\DeclarePairedDelimiter{\norm}{\lVert}{\rVert}
\DeclarePairedDelimiter{\scpr}{\langle}{\rangle}

\newcommand{\me}{\ensuremath{\mathrm{e}}}

\DeclareMathOperator{\capacity}{cap}
\DeclareMathOperator{\capacitym}{\mathbf{cap}}

\DeclareMathOperator{\mean}{\mathbb{E}}
\DeclareMathOperator{\Mean}{\mathrm{E}}
\DeclareMathOperator{\prob}{\mathbb{P}} 
\DeclareMathOperator{\Prob}{\mathrm{P}} 
\DeclareMathOperator{\var}{\mathbb{V}}

\newcommand{\ldef}{\ensuremath{\mathrel{\mathop:}=}}
\newcommand{\rdef}{\ensuremath{=\mathrel{\mathop:}}}

\newcommand{\overbar}[1]{\mkern 2mu\overline{\mkern-3mu#1\mkern-1mu}\mkern 2mu}

\newcommand{\indicator}{\mathbbm{1}}

\usepackage{mathrsfs}

\newtheorem{corollary}{Corollary}[section]
\usepackage{pgf,tikz}
\usetikzlibrary{arrows}
\definecolor{ttffqq}{rgb}{0.2,1,0}
\definecolor{ttttff}{rgb}{0.2,0.2,1}
\definecolor{eeeeee}{rgb}{0.93,0.93,0.93}
\definecolor{fftttt}{rgb}{1,0.2,0.2}
\definecolor{cqcqcq}{rgb}{0.75,0.75,0.75}

\begin{document}

\title[Metastability for the disordered Curie--Weiss--Potts model]{Metastability for the Curie--Weiss--Potts model with unbounded random interactions}

\author[J. Dubbeldam]{Johan L. A. Dubbeldam\,
\orcidlink{0000-0002-5891-998X}}
\address{Delft University of Technology}
\curraddr{Delft Institute of Applied Mathematics (DIAM), Mekelweg 4, 2628CD  Delft}
\email{J.L.A.Dubbeldam@tudelft.nl\\https://orcid.org/0000-0002-5891-998X}

\author[V. Lenz Burnier]{Vicente Lenz Burnier\,
\orcidlink{0009-0009-9497-444X}}
\address{Delft University of Technology}
\curraddr{Delft Institute of Applied Mathematics (DIAM), Mekelweg 4, 2628CD  Delft}
\email{V.Lenz@tudelft.nl\\https://orcid.org/0009-0009-9497-444X}

\author[E. Pulvirenti]{Elena Pulvirenti\,\orcidlink{0000-0003-4210-5954}}
\address{Delft University of Technology}
\curraddr{Delft Institute of Applied Mathematics (DIAM), Mekelweg 4, 2628CD  Delft}
\email{E.Pulvirenti@tudelft.nl\\https://orcid.org/0000-0003-4210-5954}

\author[M. Slowik]{Martin Slowik\,\orcidlink{0000-0001-5373-5754}}
\address{University of Mannheim}
\curraddr{Mathematical Institute, B6, 26, 68159 Mannheim}
\email{slowik@math.uni-mannheim.de\\ https://orcid.org/0000-0001-5373-5754}
\thanks{}




\date{\today}

\dedicatory{}

\begin{abstract}
  We analyse the metastable behaviour of the disordered Curie--Weiss--Potts (DCWP) model subject to a Glauber dynamics. The model is a randomly disordered version of the mean-field $q$-spin Potts model (CWP), where the interaction coefficients between spins are general independent random variables. These random variables are chosen to have fixed mean (for simplicity taken to be $1$) and well defined cumulant generating function, with a fixed distribution not depending on the number of particles. The system evolves as a discrete-time Markov chain with single spin flip Metropolis dynamics at finite inverse temperature $\beta$. We provide a comparison of the metastable behaviour of the CWP and DCWP models, when $N \to \infty$. First, we establish the metastability of the CWP model and, using this result, prove metastability for the DCWP model (with high probability). We then determine the ratio between the metastable transition time for the DCWP model and the corresponding time for the CWP model. Specifically, we derive the asymptotic tail behavior and moments of this ratio. Our proof combines the potential-theoretic approach to metastability with concentration of measure techniques, the latter adapted to our specific context.  
\end{abstract}

\maketitle

\tableofcontents



\tableofcontents

\section{Introduction}
Over the past 50 years, the mathematical study of statical and dynamical aspects of disordered mean-field spin systems has attracted considerable interest. In this paper, we continue the analysis of metastable behaviour of these systems, as initiated in \cite{BMP21}, \cite{BdHM22}, \cite{dHJ21} and \cite{BdHMPS24}, by examining the disordered Curie--Weiss--Potts (DCWP) model with unbounded interactions. This model generalises the disordered mean-field Ising model to $q \geq 2$ spins. Here, ``disordered'' refers to the fact that spin interactions are independent and identically distributed random variables. These random variables are chosen to have a fixed mean and a well-defined cumulant generating function, with a distribution independent of the number of spins. In particular, this model also encompasses the Potts model on homogeneous dense random graphs. Specific examples that fit within our framework include the Potts model on Erd\H{o}s--R{\'e}nyi random graphs, the Potts model on multi-edge random graphs, and the Potts model with Gaussian noise. 

As a first result, we prove metastability in the sense of \cite{SS19} for the Curie--Weiss--Potts (CWP) model at fixed temperature in large volumes. Further, we show that metastability of the CWP model implies metastability of the DCWP model with respect to the same metastable sets, for almost all realisations of the random interactions. After identifying specific metastable sets for both models, we derive estimates for the ratio of mean metastable transition times in the DCWP and CWP models in the regime of large-volumes and fixed temperatures. These estimates are of two types: the first one provides insight into the tail behaviour, showing that, asymptotically in $N$, this ratio behaves like a random variable of order an exponential of a sub-Gaussian random variable. Moreover, we derive moment estimates for this random ratio, again in large volumes and at fixed temperatures. Our strategy is based on techniques developed in \cite{BdHMPS24} and \cite{SS19}. Where adaptations are necessary, references are provided to the specific results in these publications.

Our methods are based on the potential-theoretic approach to metastability. This method was initiated in \cite{BEGK01} and allows us to express mean metastable exit times in terms of capacities and weighted sums of the equilibrium potential, the latter frequently referred to as harmonic sums (for a general overview of this method, we refer to \cite{BdH15}). Estimates on the former can be obtained with the help of well-known variational principles, while estimates on the latter are generally more involved and, in this manuscript, rely on a new definition of metastability given by \cite{SS19}. This definition differs slightly from the standard one given in \cite{BdH15}, yet it provides crucial insights, particularly regarding the localisation of the harmonic sums around the initial metastable set. Additionally, our proof offers a strategy for verifying metastability in other similar mean-field models. Similar to previous works on metastability in disordered models, the use of concentration of measure is pivotal in the comparison of the disordered and mean-field model. However, in contrast to these studies, we allow for potentially unbounded random interactions. To handle this, we develop concentration inequalities using Chernoff-type bounds tailored to our setting, inspired by results from \cite{Ko14}. 
Furthermore, the presence of multiple critical temperatures, unlike the single critical temperature of the Curie--Weiss model, necessitates a careful analysis of the free energy landscape and its phase transition structure (we refer to \cite{Le22} for a complete description of the free energy landscape for the CWP model). This in turn means that different temperature regimes are linked to different properties of the critical points of the free energy landscape and therefore need to be treated in different manners. 

Another class of disordered models is the random-field Curie--Weiss--Potts (RFCWP) model, in which the magnetic field is given by i.i.d. random variables. To our knowledge, the metastability of the RFCWP model has only been studied in \cite{Sl12}. The techniques developed in this paper are also applicable to the study of the metastable behaviour of the disordered Curie--Weiss--Potts model with an additional random field, under a suitable assumption on the structure of the free-energy landscape of the underlying RFCWP model.

\subsection{The model}\label{subsection:DCWP_def}
The \emph{disordered Curie--Weiss--Potts (DCWP) model} is a generalisation of the disordered mean-field Ising model to $q$ components. For any $N \in \mathbb{N}$, consider an enumeration of the vertex set consisting of $N$ elements. To each vertex $i \in \{1, \ldots, N\}$ we associate a spin variable $\sigma_{i}$ taking values in $\{1, \ldots, q\}$, $q \geq 2$, the so-called set of colours. We write $\mathcal{S}_{N} = \{1, \ldots, q\}^{N}$ to denote the corresponding state space. Elements of $\mathcal{S}_{N}$ are denoted by Greek letters $\sigma, \eta$, and will be called \emph{configurations}. 

Let $(\Omega, \mathcal{F}, \prob)$ be an abstract probability space and let $\mean$ and $\var$ denote expectation and variance with respect to $\prob$. Let $J \equiv (J_{ij})_{1 \leq i < j \leq \infty}$ be a triangular array of real random variables on $(\Omega, \mathcal{F}, \prob)$ whose law satisfies the following assumption.
\begin{assumption}
  \label{ass:law:J}
  For some $v \in (0, \infty)$ assume that the triangular array $(J_{ij})_{1 \leq i < j \leq \infty}$ consists of i.i.d.\ random variables with
  \begin{enumerate}[(i)]
  \item
    $\mean\bigl[J_{12}\bigr] = 1$ and $\var[J_{12}]=v$,

  \item
    the set $\mathscr{D} = \left\{\lambda \in \mathbb{R} \colon \mean\bigl[\exp(\lambda J_{12})\bigr] < \infty \right\}$ has non-empty interior containing~$0$.
  \end{enumerate}
\end{assumption}
Given a realisation of $J$ and $N \in \mathbb{N} \setminus \{1\}$, we consider the following \emph{random} Hamiltonian, $H_{N}\colon \mathcal{S}_{N} \to \mathbb{R}$, given by
\begin{align}
  \label{eq:H_Jij_form}
  H_{N}(\sigma)
  \;\ldef\;
  -\frac{1}{N} \sum_{1 \leq i < j \leq N} J_{ij}\,
  \indicator_{\{\sigma_{i}=\sigma_{j}\}}.
\end{align}
The corresponding random Gibbs measure, $\mu_{N}$, at inverse temperature $\beta \geq 0$ is defined by
\begin{align}
  \mu_{N}(\sigma)
  \;\equiv\;
  \mu_{N, \beta}(\sigma)
  \;\ldef\;
  \frac {\me^{-\beta H_N(\sigma)}}{Z_{N}},
\end{align}
where $Z_{N} \equiv Z_{N, \beta}$ denotes the partition function. In view of Assumption~\ref{ass:law:J}-(ii), the expected value of the partition function is finite, for all values of $\beta$ and provided that $N$ is chosen large enough. Notice that for $q=2$, the model becomes the disordered Curie--Weiss model. 

The spin configuration evolves as a discrete-time Markov chain $\Sigma^{N} \equiv ( \Sigma^{N}(t) )_{t \geq 0}$, with state space $\mathcal{S}_{N}$ and Glauber-Metropolis transition probabilities given for any $\sigma, \eta \in \mathcal{S}_{N}$ by
\begin{align}\label{eq:rates_def}
  \pi_{N}(\sigma, \eta)
  \;\equiv\;
  \pi_{N, \beta}(\sigma, \eta)
  \;=\;
  \begin{cases}
    \,(N q)^{-1} \me^{-\beta\left[H_N(\eta) - H_N(\sigma)\right]_+},
    &\text{if } d_{H}(\sigma, \eta) = 1,
    \\
    1-\sum_{\eta\neq \sigma} \pi_N(\sigma,\eta)
    &\text{when } \sigma=\eta,
    \\
    0
    &\text{otherwise},
  \end{cases}
\end{align}
where $[a]_{+} \ldef \max\{0, a\}$ and $d_{H}(\sigma, \eta)$ denotes the Hamming distance between configurations $\sigma$ and $\eta$. To lighten notation we will also write $\sigma \sim \eta$, if $d_{H}(\sigma, \eta) = 1$. The Markov chain, $\Sigma^{N}$, defined by $\pi_{N}$ is irreducible and reversible with respect to the Gibbs measure $\mu_{N}$.  The associated (discrete) generator $\mathcal{L}_{N}$ acts on bounded functions $f\colon \mathcal{S}_{N} \to \mathbb{R}$ as
\begin{align}\label{eq:generator_def}
  \bigl(\mathcal{L}_N f\bigr)(\sigma)
  \;\coloneqq\;
  \sum_{\eta \in \mathcal{S}_N} \pi_{N}(\sigma, \eta) \bigl( f(\eta) - f(\sigma) \bigr).
\end{align}
For any $N \in \mathbb{N}$, we write $\Prob_{\!\nu}^{N}$ to denote the law of $\Sigma^{N}$ starting from an initial distribution $\nu$ in $\mathcal{S}_{N}$, and $\Mean_{\nu}^{N}$ to denote the corresponding expectation. Furthermore, for $\mathcal{A} \subset \mathcal{S}_N$, we define the \emph{first return time} $\tau_{\mathcal{A}}^{N}$ to be the following
\begin{align}\label{def:hitting_time}
  \tau_{\mathcal{A}}^{N}
  \;\equiv\;
  \tau_{\mathcal{A}}^{N}\bigl( \Sigma^{N} \bigr)
  \;\coloneqq\;
  \inf\bigl\{
    t > 0 \colon \Sigma^{N}(t) \in \mathcal{A}
  \bigr\}.
\end{align}
The goal of the present paper is to compare the metastable behaviour of the DCWP model with that of the standard mean-field Curie--Weiss--Potts (CWP) model. The latter is the model with Hamiltonian
\begin{align}
  \widetilde{H}_{N}(\sigma)
  \;\coloneqq\;
  -\frac{1}{N} \sum_{1 \leq i < j \leq N} \indicator_{\{ \sigma_{i} = \sigma_{j} \}}
  \;=\;
  \mean\bigl[ H_{N}(\sigma) \bigr].
\end{align}
Quantities {such as} $\widetilde{Z}_{N}, (\widetilde{\Sigma}^{N}(t))_{t \geq 0}, \widetilde{\pi}_{N}, \widetilde{\mathcal{L}}_{N}$ and any other one with the $\sim$ superscript are defined analogously, taking $\widetilde{H}_N$ instead of $H_N$.  With an abuse of terminology and in accordance with the literature (see e.g.~\cite{BMP21} and \cite{BdHMPS24}), we sometimes refer to the models defined in terms of $H_{N}$ and $\widetilde{H}_{N}$ as the \emph{quenched} and the \emph{annealed} model, respectively.

A particular feature of the CWP model is that its Hamiltonian can be expressed in terms of the \emph{empirical measure}, $L_{N}$, encoding the relative frequencies of the different colours.  For this purpose, define $L_{N}\colon \mathcal{S}_{N} \to \mathcal{P}_{N}$ by
\begin{align}\label{eq:def_rho}
  \sigma
  \longmapsto
  \bigl( L_{N}(\sigma)[\{1\}], \dots, L_{N}(\sigma)[\{q\}] \bigr)
  \quad \text{with} \quad
  L_{N}(\sigma)[\{k\}]
  \;\ldef\;
  \frac{1}{N} \sum_{i=1}^{N} \indicator_{\{ \sigma_{i}=k \}},
\end{align}
where $\mathcal{P} \ldef \{x \in [0, 1]^{q} : \sum_{k=1}^{q} x_{k} = 1\}$ and $\mathcal{P}_{N} \ldef \frac{1}{N} \mathbb{N}_{0}^{q} \cap \mathcal{P}$. Then, the Hamiltonian of the CWP model can be rewritten as
\begin{align}\label{eq:H_meso_form}
  \widetilde{H}_{N}(\sigma) \;=\; -\frac{N}{2} \norm{L_{N}(\sigma)}_{2}^{2} + \frac{1}{2}.
\end{align}
Since the transition probabilities $\widetilde{\pi}_{N}$ depend only on the energy difference of two adjacent configurations, we have that the process is lumpable, that is, $\bigl(L_{N}(\widetilde{\Sigma}^{N}(t))\bigr)_{t\geq 0}$ is also a reversible, discrete-time Markov process, see e.g.~\cite[Proposition~2.1]{Le22}.

Our choice of the quenched model is very general and we now illustrate it with three specific examples. The first two pertain to the Potts model on two different types of random graphs, while the third example outlines a Potts model incorporating Gaussian noise.
\begin{example}[Potts model on the Erd\H{o}s--R{\'e}nyi random graph]
  \label{sec:examples}
  By choosing $J_{12}$ distri\-bu\-ted as $p^{-1}\operatorname{Ber}(p)$ with $p \in (0, 1]$, $H_{N}$ in \eqref{eq:H_Jij_form} becomes the Hamiltonian of the CWP model on the Erd\H{o}s--R{\'e}nyi random graph in which edges are present with probability $p$.
\end{example} 
\begin{example}[Potts model on the multi-edge random graph]
  Let $K \sim \text{Pois}(p\binom{N}{2})$ and let $(I_k^N)_{k \in \{1, \ldots, K\}}$ be a sequence of i.i.d.\ uniform random variables in $\{\{i, j\} : 1 \leq i < j\leq N\} $. These define the so-called multi-edge random graph with edge set $E$, also known as Norros--Reittu model (see \cite{NR06}). The CWP model on the multi-edge random graph is therefore defined by the Hamiltonian 
  \begin{align}
    H_N(\sigma)
    \;=\;
    -\frac{1}{Np}\sum_{\{i,j\} \in E}\indicator_{\{\sigma_{i} = \sigma_{j}\}},
  \end{align}
  that is, we sum over the edges present in the random graph and set the interaction identically equal to $1$. We obtain the same model by defining the random variable
  \begin{align}
    J_{ij}
    \;\ldef\;
    \frac{1}{p}\sum_{k=1}^{K} \indicator_{\{ I^{N}_{k} = \{i, j\} \}}
  \end{align}
  and replacing it in the Hamiltonian \eqref{eq:H_Jij_form}. This is the same as choosing $J_{12}$ distributed as $p^{-1}\operatorname{Pois}(p)$ with $p \in (0,1]$ in \eqref{eq:H_Jij_form}. Notice that these random variables are not sub-Gaussian.
\end{example}
\begin{example}[Potts model with Gaussian noise]
  By letting $J_{ij} \sim \mathcal{N}(1, v)$ results in an only \emph{partially} ferromagnetic model, as the random variables are allowed negative values. However, our results show that, for fixed $v$ and for $N$ going to infinity, it behaves as the ferromagnetic mean-field model. In addition, the form of the cumulant generating function simplifies the expression of some results, for instance dropping the error term from lemma $\ref{lemma:xi_comp}$ and it's consequences.
\end{example}

\subsection{Main results}\label{subsec:main_results}
Our main objective is to compare the metastable transition times of the CWP model with the ones of the DCWP model. For this purpose, let us first recall the definition of metastable Markov chains and metastable sets following \cite[Definition~1.1]{SS19}.
\begin{definition}[$\rho_{N}$-Metastability]\label{def:rhometa}
  For $\rho_{N} > 0$ and $K \in \mathbb{N}$, let $\bigl\{\mathcal{M}_{1, N}, \dots, \mathcal{M}_{K, N}\bigr\}$ be a collection of disjoint subsets of $S_{N}$ and set $\mathcal{M}_{N} \ldef \bigcup_{i=1}^{K} \mathcal{M}_{i, N}$. The Markov chain $(\Sigma^{N}(t))_{t \geq 0}$ is $\rho_{N}$-metastable with respect to $\bigl\{\mathcal{M}_{1, N}, \dots, \mathcal{M}_{K, N}\bigr\}$ when
  \begin{align}\label{eq:rhometa}
    K
    \frac{%
      \max_{j \in \{1, \dots, K\}}
      \Prob^{N}_{\mu_{N} \rvert \mathcal{M}_{j, N}}\bigl[
        \tau^{N}_{\mathcal{M}_{N} \setminus \mathcal{M}_{j, N}} < \tau^{N}_{\mathcal{M}_{j, N}}
      \bigr]
    }{%
      \min_{\mathcal{X} \subset \mathcal{S}_N \setminus \mathcal{M}_N}
      \Prob^{N}_{\mu_{N} \rvert \mathcal{X}}\bigl[ \tau^{N}_{\mathcal{M}_N} < \tau^{N}_{\mathcal{X}} \bigr]
    }
    \;\leq\;
    \rho_N
    \;\ll\;
    1,
  \end{align}
  where, for a non-empty set $\mathcal{X} \subset \mathcal{S}_{N}$, $\mu_{N} \rvert \mathcal{X}$ denotes the invariant measure $\mu_{N}$ conditioned on the set $\mathcal{X}$.
\end{definition}
\begin{remark}
  This definition of metastability covers both \emph{metastable transitions} and \emph{tunneling transitions}. In the former, the system evolves towards states of lower energy, whereas in the latter the system moves between states with the same energy. Due to the symmetry of the CWP model, we see both types of states.
\end{remark}
In general, identifying suitable candidates for metastable sets can be a challenging task that is highly dependent on the specific model being considered. However, for mean-field spin systems, it is well established, cf.~\cite{BdH15,OV05}, that metastable sets correspond to the local minima of the free energy landscape. Specifically, in the context of the Curie--Weiss--Potts model, it is known (see, for example, \cite{EW90}) that for any $\beta \in (0, \infty)$, the limiting free energy $\widetilde{F}_{\beta, q}\colon \mathcal{P} \to \mathbb{R}$ is given by
\begin{align}
  \lim_{N \to \infty} -\frac{1}{\beta N} \ln \widetilde{Z}_{N}
  \;=\;
  \inf_{\boldsymbol{x} \in \mathcal{P}} \widetilde{F}_{\beta, q}(\boldsymbol{x}),
\end{align}
where
\begin{align}\label{eq:freeen2}
  \widetilde{F}_{\beta, q}(\boldsymbol{x})
  \;=\;
  -\frac{1}{2}\, \norm{\boldsymbol{x}}_{2}^{2}
  + \frac{1}{\beta}\sum_{i=1}^{q} \boldsymbol{x}_{i} \ln \boldsymbol{x}_{i}.
\end{align}
While the phase diagram of the Curie--Weiss--Potts model -- specifically, the dependence of the \emph{global} minima of $\widetilde{F}_{\beta, q}$ on $\beta$ -- is well-established and thoroughly described in \cite{Wu82, EW90, CET05}, a comprehensive characterisation of the metastable states given by the \emph{local} minima of $\widetilde{F}_{\beta, q}$ \emph{and} the relevant connecting saddle points has recently been studied in \cite{KM20} for $q = 3$ and \cite{Le22} for $q \geq 3$. While the Curie--Weiss model has only one critical value $\beta_{c} = 1$, the CWP model exhibits at least three (critical) temperatures, $0 < \beta_{1}(q) < \beta_{2}(q) < q$ at which the free energy landscape (and therefore the metastable behaviour of  the model) change drastically depending on the chosen temperature regime\footnote{Note that in \cite{Le22}, four relevant temperatures $\beta_1<\beta_2<\beta_3\leq\beta_4$ are introduced and discussed, where $\beta_1=\beta_1(q),\beta_2=\beta_2(q)$ and $\beta_4=q$. The additional temperature $\beta_3\in(\beta_2,\beta_4]$ appearing only for $q\geq 5$ describes the relevant saddle point for metastable transition between local minima, see also Section \ref{app}.}. The local minima of $\widetilde{F}_{\beta, q}$ can be characterised as follows:  Set $\boldsymbol{m}_{0} \equiv \boldsymbol{m}_{0}(q) \ldef (1/q, \ldots, 1/q) \in \mathcal{P}$ and, for any $i \in \{1, \ldots, q\}$, $\boldsymbol{m}_{i} \equiv \boldsymbol{m}_{i}(\beta, q)  = (\boldsymbol{m}_{i, 1}, \ldots, \boldsymbol{m}_{i, q}) \in \mathcal{P}$, where
\begin{align*}
  \boldsymbol{m}_{i, k}
  \;\ldef\;
  \begin{cases}
    (1-s)/q, &k \neq i
    \\[.5ex]
    (1+(q-1)s)/q, &k = i
  \end{cases},
\end{align*}
with $s$ being the largest solution of the equation $\ln(1+(q-1)s) - \ln(1-s) = \beta s$. For $\beta \leq \beta_{1}(q)$, $\boldsymbol{m}_{0}$ is the unique global minimum. For $\beta_1(q)<\beta\leq \beta_2(q)$ $\boldsymbol{m}_{0}$ is a global minimum and $\{\boldsymbol{m}_{1}, \ldots, \boldsymbol{m}_{q}\}$ are local minima. For $\beta_2(q)\leq\beta< q$, $\boldsymbol{m}_{0}$ is a local minimum and $\{\boldsymbol{m}_{1}, \ldots, \boldsymbol{m}_{q}\}$ are global minima. Finally, for $\beta\geq q$, $\{\boldsymbol{m}_{1}, \ldots, \boldsymbol{m}_{q}\}$  are the global minima of $\widetilde{F}_{\beta, q}$. This is summarized in the following table. For a graphical illustration, see Figure~\ref{fig:free:energy}.
\begin{center}
  \begin{tabular}{c||c c c c c} 
    \hline
    $\beta\in$ & $(0,\beta_1(q)]$ &$(\beta_{1}(q), \beta_{2}(q))$& $\{\beta_2(q)\}$ & $(\beta_{2}(q), q)$ &$[q,\infty)$  \\ 
    $\boldsymbol{m}_0$ & global min. & global min. & global min. & local min. & - \\ 
    $\boldsymbol{m}_i$ & - & local min. & global min. & global min. & global min.\\
    \hline
  \end{tabular}
\end{center}
%
\begin{figure}[ht]
  \centering
  \begin{subfigure}{0.35\textwidth}
    \centering
    \includegraphics[width=\textwidth,height=2.5cm,keepaspectratio]{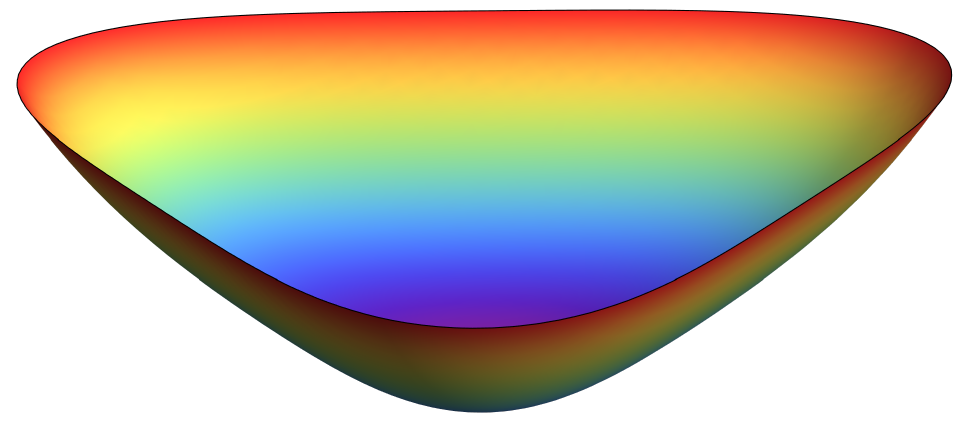}
    \caption{$\beta \leq \beta_{1}(q)$.}
  \end{subfigure}
  \begin{subfigure}{0.35\textwidth}
    \centering
    \includegraphics[width=\textwidth,height=2.5cm,keepaspectratio]{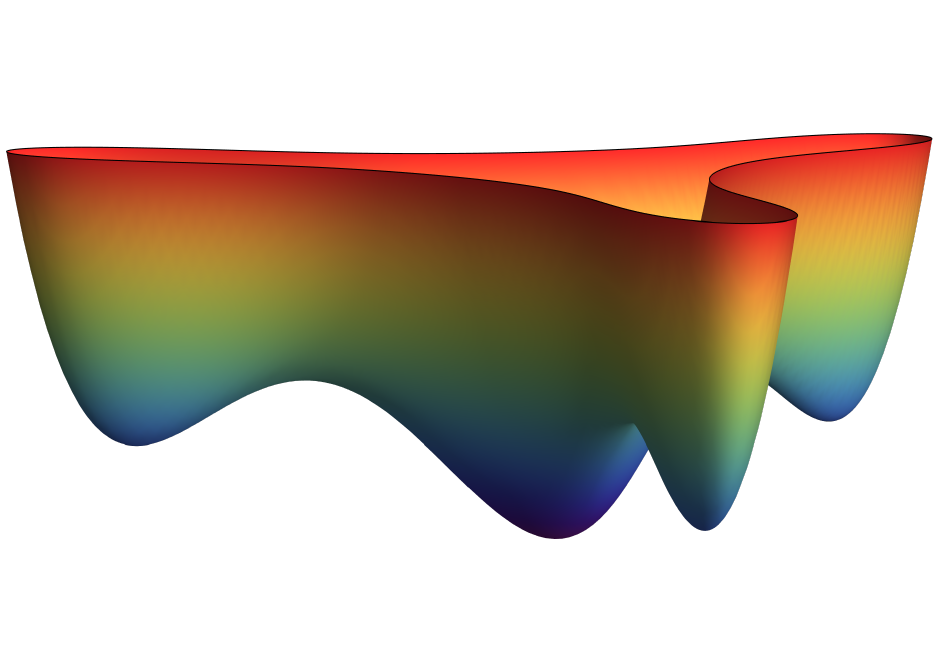}
    \caption{$\beta_{1}(q) < \beta < \beta_{2}(q)$.}
  \end{subfigure}
  \begin{subfigure}{0.35\textwidth}
    \centering
    \includegraphics[width=\textwidth,height=2.5cm,keepaspectratio]{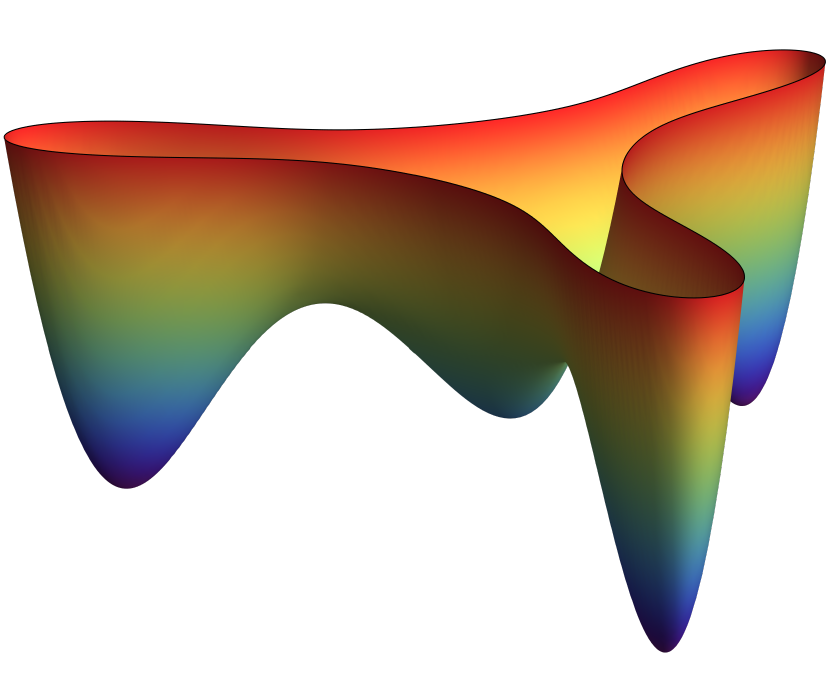}
    \caption{$\beta_{2}(q) < \beta < q$.}
  \end{subfigure}
  \begin{subfigure}{0.35\textwidth}
    \centering
    \includegraphics[width=\textwidth,height=2.5cm,keepaspectratio]{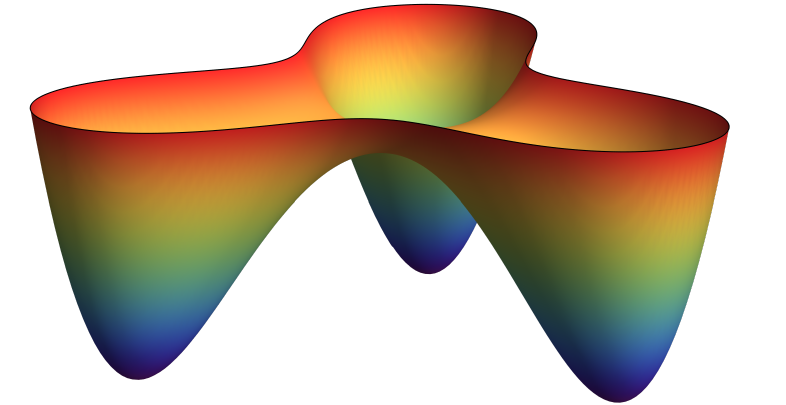}
    \caption{$\beta \geq q$.}
  \end{subfigure}
  \caption{%
    Illustrations of the graph of $\widetilde{F}_{\beta, q}$ for $q = 3$ and different values of $\beta$.
  }
  \label{fig:free:energy}
\end{figure}

For any $N \in \mathbb{N}$ and $i \in \{0, 1, \ldots, q\}$, let $\boldsymbol{m}_{i, N} \in \mathcal{P}_{N}$ be a closest lattice point approximations of $\boldsymbol{m}_{i}$, respectively, and set
\begin{align}
  \mathcal{I}_{\beta}
  \;=\;
  \begin{cases}
    \{0, \ldots, q\}, \quad \text{if } \beta_{1}(q) < \beta < q,
    \\ 
    \{1, \ldots, q\}, \quad \text{if } \beta \geq q.
  \end{cases}
\end{align}
Further, we define for any $\beta \in (\beta_{1}(q), \infty)$ and $i \in \mathcal{I}_{\beta}$ the sets $\mathcal{M}_{i, N} \subset \mathcal{S}_{N}$ as the (set-valued) pre-image of the empirical measure, $L_{N}$, of the points $\boldsymbol{m}_{i, N}$, that is,
\begin{align}\label{eq:metaset}
  \mathcal{M}_{i, N}
  \;\ldef\;
  L^{-1}_{N}(\boldsymbol{m}_{i, N}),
  \qquad i \in \mathcal{I}_{\beta}.
\end{align}
%
%
Our first result says that the Curie--Weiss--Potts model exhibits metastable behaviour in the sense of Definition~\ref{def:rhometa}.
\begin{theorem}[Metastability of the CWP model] \label{thm:meta:cwp}
  For every $\beta > \beta_{1}(q)$ there exists $k_{1} \equiv k_{1}(\beta, q)$ such that, for all $k < k_{1}$, there exists $N_{0} \in \mathbb{N}$ with the property that for every $N > N_0$ the Markov chain $(\widetilde{\Sigma}^{N}(t))_{t \geq 0}$ is $\me^{-k N}$-metastable with respect to the metastable sets $\{\mathcal{M}_{i, N} : i \in \mathcal{I}_{\beta}\}$ defined in \eqref{eq:metaset}.
\end{theorem}
\begin{remark}
  Although the above result is primarily used in the proof of the next theorem, that addresses the metastability of the disordered Curie--Weiss--Potts model, it is a novel development in its own right. The parameter $k_{1}$, whose explicit value is given in the proof of Theorem~\ref{thm:meta:cwp}, encodes the strength of metastability with respect to the chosen sets $\{\mathcal{M}_{i, N} : i \in \mathcal{I}_{\beta}\}$ and, consequently, the separation of time scales induced by the underlying free energy landscape. For a discussion of the optimality of $k_{1}$, see Remark~\ref{rem:optimal:k1:CWP}.
\end{remark}
The next theorem states that the disordered Curie--Weiss--Potts model is also $\rho$-metastable in the sense of Definition~\ref{def:rhometa} with the same metastability parameter and metastable set of the CWP model as described in Theorem~\ref{thm:meta:cwp}.
\begin{theorem}[Metastability of the DCWP model] \label{thm:meta:dcwp}
  For every $\beta > \beta_{1}(q)$ and any $k \in (0, k_{1})$, where $k_1$ is defined as in Theorem \ref{thm:meta:cwp}, the event
  \begin{align}\label{eq:def_Ometa}
    \Omega_{\textnormal{meta}}(N)
    \;\ldef\;
    \Bigl\{%
      ({\Sigma}^{N}(t))_{t \geq 0} \text{ is $\me^{-k N}$-metastable}
      \text{ w.r.t.}\, \{\mathcal{M}_{i,N} : i \in \mathcal{I}_{\beta}\}
    \Bigr\}
  \end{align}
  satisfies
  \begin{align}\label{eq:meta_thm}
    \prob\biggl[\liminf_{N \to \infty} \Omega_{\textnormal{meta}}(N)\biggr]
    \;=\;
    1.
  \end{align}
\end{theorem}
\begin{remark}
  Note that, $\prob$-almost surely, the additional disorder induced by the random variables $(J_{ij})_{1 \leq i < j \leq N}$  does not affect the strength of metastability of the Markov chain $\Sigma^{N}$ with respect to $\{\mathcal{M}_{i, N} : i \in \mathcal{I}_{\beta}\}$. Actually, as we argue in Remark~\ref{rem:tightk}, the constant $k_{1}$ is optimal (on exponential scale): for every $k > k_{1}$ there exists a random variable $N_{1}$ such that, $\prob$-almost surely, $N_{1} < \infty$ and for all $N \geq N_{1}$ the Markov chain $\Sigma^{N}$ is not $\me^{-k N}$- metastable.
\end{remark}

In our second set of results, we compare the mean transition times between specific disjoint subsets of the metastable sets of the Markov chain $(\Sigma^{N}(t))_{t \geq 0}$ with those of the corresponding Markov chain $(\widetilde{\Sigma}^{N}(t))_{t \geq 0}$. For this purpose, we distinguish between metastable and tunnelling transitions. In the former, we examine the mean hitting times of the metastable set, $\mathcal{B}_{N}$, associated with the global minima of $\widetilde{F}_{\beta, q}$ (stable states) when the corresponding Markov chain starts in the metastable set, $\mathcal{A}_{N}$, linked to local minima of $\widetilde{F}_{\beta, q}$ (metastable states). In contrast, the latter pertains to transitions between stable states. 
\begin{definition}\label{def:regimes}
  For metastable transitions we consider the following metastable and stable sets:
  \begin{enumerate}
  \item[(i)]
    $\mathcal{A}_{N} = \bigcup_{i=1}^{q} \mathcal{M}_{i, N}$ and $\mathcal{B}_{N} = \mathcal{M}_{0,N}$ if $\beta_{1}(q) < \beta \leq \beta_{2}(q)$,    
  \item[(ii)]
    $\mathcal{A}_{N} = \mathcal{M}_{0, N}$ and $\mathcal{B}_{N} = \bigcup_{i=1}^{q} \mathcal{M}_{i, N}$ if $\beta_{2}(q) < \beta < q$.
  \end{enumerate}
  For tunnelling transitions we consider the following stable sets:
  \begin{enumerate}
  \item[(iii)]
    $\mathcal{A}_{N} = \mathcal{M}_{1, N}$ and $\mathcal{B}_{N} = \bigcup_{i=2}^{q} \mathcal{M}_{i, N}$ if $\beta_{2}(q) < \beta$.
  \end{enumerate}
\end{definition}
\begin{remark}
  Notice that in case (iii) it is possible to define $\mathcal{A}_{N}$ as any of the sets $\mathcal{M}_{i, N}$, $i \in \{1, \dots, q\}$, and $\mathcal{B}_{N}$ as the union of the remaining ones. 
\end{remark} 
Moreover, we define, for non-empty disjoint sets $\mathcal{A}, \mathcal{B} \subset \mathcal{S}_{N}$, the so-called \emph{last-exit biased distribution} on $\mathcal{A}$ for the transition from $\mathcal{A}$ to $\mathcal{B}$ by
\begin{align}
  \label{eq:def_nuAB}
  \nu_{\mathcal{A}, \mathcal{B}}(\sigma)
  \;\equiv\;
  \nu_{\mathcal{A}, \mathcal{B}}^N(\sigma)
  \;=\;
  \frac{\mu_{N}(\sigma) \Prob_{\!\sigma}^{N}\bigl[\tau_{\mathcal{B}}^N < \tau_{\mathcal{A}}^N\bigr]}
  {%
    \sum_{\sigma \in \mathcal{A}} \mu_{N}(\sigma)
    \Prob_{\!\sigma}^{N}\bigl[\tau_{\mathcal{B}}^N < \tau_{\mathcal{A}}^N\bigr]
  },
  \qquad \sigma \in \mathcal{A}.
\end{align}
This distribution, weighting the underlying Gibbs measure $\mu_{N}$ with the probability of a successful escape from $\mathcal{A}$ to $\mathcal{B}$, plays an essential role in the potential-theoretic approach to metastability, as will be explained in Section~\ref{sec:pottheo}.

Our next two results describe the mean hitting time of the stable set $\mathcal{B}_{N}$ when starting the Markov chain, $(\Sigma^{N}(t))_{t \geq 0}$, with initial distribution $\nu_{\mathcal{A}_{N}, \mathcal{B}_{N}}$ and compare it to the corresponding quantity for the Markov chain $(\widetilde{\Sigma}^{N}(t))_{t \geq 0}$. Theorem~\ref{theo:ratio_conc} provides an estimate of the tail behaviour of the ratio of these hitting times, while Theorem~\ref{theo:ratio_mom} provides moment estimates. In particular, the following results show that the mean hitting time of the DCWP model is equal to the one of the CWP model times a random variable. This random variable is the exponential of a sub-Gaussian random variable whose tail distribution is described in Theorem~\ref{theo:ratio_conc} and whose moments are estimated in Theorem~\ref{theo:ratio_mom}.
\begin{theorem}[Tail estimates of the mean hitting time]\label{theo:ratio_conc}
  For any $\beta > \beta_{1}(q)$ let $\mathcal{A}_{N}$ and $\mathcal{B}_{N}$ be chosen as in Definition~\ref{def:regimes}. Then for any $t > 0$,
  \begin{align}
    \liminf_{N \to \infty}
    \prob\Biggl[
      \me^{-t - \beta^{2} v/4}
      \leq
      \frac{%
        \Mean^{N}_{\nu_{\mathcal{A}_{N}, \mathcal{B}_{N}}}\bigl[ \tau^{N}_{\mathcal{B}_N} \bigr]
      }{%
        \widetilde{\Mean}^{N}_{\tilde{\nu}_{\mathcal{A}_{N}, \mathcal{B}_{N}}}\bigl[\tilde{\tau}^N_{\mathcal{B}_N}\bigr]
      }
      \leq
      \me^{t + \beta^{2} v/2}
    \Biggr]
    \;\geq\;
    1 - 4 \me^{-t^{2}/(8 \beta^{2} v)}.
  \end{align}
\end{theorem}
\begin{theorem}[Moment estimates of the mean hitting time]\label{theo:ratio_mom}
  For any $\beta > \beta_{1}(q)$ let $\mathcal{A}_{N}$ and $\mathcal{B}_{N}$ be chosen as in Definition~\ref{def:regimes}. Then for any $k \geq 1$, 
  \begin{align}\label{eq:theo:ratio_mom}
    \me^{-\beta^{2} v/4}
    \;\leq\;
    \liminf_{N \to \infty}
    \frac{
      \mean\Bigl[\Mean^{N}_{\nu_{\mathcal{A}_{N}, \mathcal{B}_{N}}}\bigl[ \tau^{N}_{\mathcal{B}_N}\bigr]^k \Bigr]^{1/k}
    }{
      \widetilde{\Mean}^{N}_{\tilde{\nu}_{\mathcal{A}_{N}, \mathcal{B}_{N}}}\bigl[\tilde{\tau}^N_{\mathcal{B}_N}\bigr]
    }
    \;\leq\;
    \limsup_{N \to \infty}
    \frac{
      \mean\Bigl[\Mean^{N}_{\nu_{\mathcal{A}_{N}, \mathcal{B}_{N}}}\bigl[ \tau^{N}_{\mathcal{B}_N}\bigr]^k \Bigr]^{1/k}
    }{
      \widetilde{\Mean}^{N}_{\tilde{\nu}_{\mathcal{A}_{N}, \mathcal{B}_{N}}}\bigl[\tilde{\tau}^N_{\mathcal{B}_N}\bigr]
    }
    \;\leq\;
    \me^{\beta^{2} v k}.
  \end{align}
\end{theorem}
\begin{remark}
  We emphasize that this framework naturally extends to inhomogeneous settings by allowing random variables to be independent, but not necessarily identically distributed. However, for clarity and readability, we adopt the i.i.d.\ assumption in our analysis.
\end{remark}
Aspects of the metastable behaviour of the CWP model have also been studied in \cite{Le22} and \cite{LS16}, where under a suitable time rescaling, a limiting process has been derived for both reversible and non-reversible Glauber dynamics, in the spirit of the martingale approach to metastability {(see e.g. \cite{BL15})}. Moreover mixing times for Glauber and Swendsen-Wang type dynamics are estimated in \cite{CDLLPS12}, \cite{KL26} and \cite{GSV19}. Moreover, an explicit expression for the mean hitting time of the CWP model has been obtained in \cite[Theorem 4.7]{Le22} and \cite[Theorem 5.5]{Sl12}, as a special case of the random field CWP model.

In the context of disordered spin models, previous works have estimated the meta\-stable transition time for the Curie--Weiss model with either bond or site disorder, as in  \cite{BMP21}, \cite{dHJ21}, \cite{BdHMPS24} and \cite{BdHM22}. The same results have been obtained in \cite{BBI09} for the Curie--Weiss model with random magnetic field. Further results have also been obtained for the CWP with random magnetic field in \cite{Sl12}, where, under some assumptions on the free energy landscape, sharp bounds on the mean hitting times are given.

\subsection{Methods and outline}\label{sec:pottheo}
Our proofs crucially rely on the \emph{potential theoretic} approach to metastability, which links the probabilistic objects describing the metastable behaviour of the system to the solutions of certain boundary value problems. This approach was initiated by the paper \cite{BEGK01} and leads to precise asymptotics of the metastable transition time (for a general overview of this method we refer to \cite{BdH15}). Furthermore, we establish Chernoff-type concentration inequalities for arbitrary Lipschitz functions of the edge weights, which can be achieved under a condition pertaining to the existence of the cumulant generating function. This approach was inspired by \cite{Ko14}, where the concept of subgaussian random variables is generalised to encompass arbitrary metric probability spaces.

\subsubsection{Key notions of potential theory}\label{subsec:pot_theo}
For every $N \in \mathbb{N}$ and non-empty, disjoint subsets $\mathcal{A}, \mathcal{B}\subset \mathcal{S}_N$, the \emph{equilibrium potential} $h^{N}_{\mathcal{A}, \mathcal{B}}: \mathcal{S}_{N} \to [0,1]$ is the unique solution of the boundary value problem
\begin{align}\label{eq:dirichlet_prob}
  \left\{
    \begin{array}{rcll}
      \big(\mathcal{L}_{N} f\big)(\sigma)
      &\mspace{-5mu}=\mspace{-5mu}& 0,
      &\sigma \in \mathcal{S}_{N} \setminus (\mathcal{A} \cup \mathcal{B}),
      \\[.5ex] 
      f(\sigma)
      &\mspace{-5mu}=\mspace{-5mu}& \mathbbm{1}_{\mathcal{A}}(\sigma), \quad 
      &\sigma \in \mathcal{A} \cup \mathcal{B}.
    \end{array}
  \right.
\end{align}
Notice that $h^{N}_{\mathcal{A}, \mathcal{B}}$ has the following probabilistic interpretation: for any $\sigma \in \mathcal{S}_{N} \setminus (\mathcal{A} \cup \mathcal{B})$, we have $h^{N}_{\mathcal{A}, \mathcal{B}}(\sigma) = \Prob^{N}_{\!\sigma}\bigl[\tau^{N}_{\mathcal{A}} < \tau^{N}_{\mathcal{B}} \bigr]$. Another pivotal quantity in potential theory is the \emph{capacity} of the pair $(\mathcal{A}, \mathcal{B})$ that is defined by
\begin{align}\label{eq:cap_def}
  \capacity_{N}(\mathcal{A}, \mathcal{B})
  \;\ldef\;
  \sum_{\sigma \in \mathcal{A}} \mu_{N}(\sigma)\, \Prob^{N}_{\!\sigma}\bigl[\tau^{N}_{\mathcal{B}} < \tau^{N}_{\mathcal{A}} \bigr] 
  \;=\;
  \sum_{\sigma \in \mathcal{A}} \mu_{N}(\sigma)\bigl(-\mathcal{L}_{N} h^{N}_{\mathcal{A}, \mathcal{B}}\bigr)(\sigma).
\end{align}
Recalling that we write $\mu_{N} \rvert \mathcal{A}$ to denote the Gibbs measure $\mu_{N}$ conditioned on the set $\mathcal{A}$, we clearly have that
\begin{align}\label{eq:escape_prob}
  \Prob^{N}_{\mu_{N} \rvert \mathcal{A}}\bigl[\tau^{N}_{\mathcal{B}} < \tau^{N}_{\mathcal{A}} \bigr]
  \;=\;
  \frac{\capacity_{N}(\mathcal{A}, \mathcal{B})}{\mu_{N}[\mathcal{A}]}.
\end{align}
Furthermore, since $h^{N}_{\mathcal{A}, \mathcal{B}}(\sigma) + h^{N}_{\mathcal{B}, \mathcal{A}}(\sigma) = 1$, for any $\sigma \in \mathcal{S}$, and $\mathcal{L}_{N}$ applied to a constant function vanishes, the definition of metastability also implies $\capacity_{N}(\mathcal{A}, \mathcal{B}) = \capacity_{N}(\mathcal{B}, \mathcal{A})$. Moreover, for arbitrary sets $\mathcal{A}, \mathcal{B}, \mathcal{C} \subset \mathcal{S}_{N}$ with $\mathcal{A} \subset \mathcal{B}$ and $\mathcal{B} \cap \mathcal{C} = \emptyset$,
\begin{align}\label{eq:monotonic}
  \capacity_{N}(\mathcal{C},\mathcal{A})
  &\;=\;
  \sum_{\sigma \in \mathcal{C}} \mu_{N}(\sigma)\, \Prob^{N}_{\!\sigma}\bigl[ \tau^{N}_{\mathcal{A}} < \tau^{N}_{\mathcal{C}} \bigr]
  \nonumber\\[.5ex]
  &\;\leq\;
  \sum_{\sigma \in \mathcal{C}} \mu_{N}(\sigma)\, \Prob^{N}_{\!\sigma}\bigl[ \tau^{N}_{\mathcal{B}} < \tau^{N}_{\mathcal{C}}\bigr]
  \;=\;
  \capacity_{N}(\mathcal{C}, \mathcal{B}).
\end{align}
By using potential theory, see e.g.\ \cite[Corollary 7.11]{BdH15}, one obtains an expression for the mean hitting time between two sets, albeit with respect to the last-exit biased distribution defined in \eqref{eq:def_nuAB}
\begin{align}\label{eq:golden_formula}
  \Mean^{N}_{\nu_{\mathcal{A},\mathcal{B}}}\bigl[ \tau^{N}_{\mathcal{B}} \bigr]
  \;=\;
  \frac{1}{\capacity_{N}(\mathcal{A}, \mathcal{B})}\,
  \sum_{\sigma \in \mathcal{S}_{N}} \mu_{N}(\sigma)\, h^{N}_{\mathcal{A}, \mathcal{B}}(\sigma)
  \;=\;
  \frac{\|h^{N}_{\mathcal{A}, \mathcal{B}}\|_{\mu_{N}}}{\capacity_{N}(\mathcal{A}, \mathcal{B})},
\end{align}
where $\norm{\cdot}_{\mu_N}$ denotes the $\ell_{1}(\mu_{N})$-norm. The presence of the last-exit biased distribution in the formula is an artefact of the method. However, if the starting set is atomic it reduces to a Dirac measure. From Equations~\eqref{eq:escape_prob} and \eqref{eq:golden_formula} we see that, via the potential-theoretic approach, the definitions and quantities relevant to metastability admit both probabilistic and analytical representations. One difficulty in this route is estimating the $\ell_{1}$-norm of the harmonic function $h_{\mathcal{A},\mathcal{B}}^{N}$, for which we rely on results from \cite{SS19}. In order to estimate capacities effectively, we employ two dual variational principles. The \emph{Dirichlet principle} states that
\begin{align}
  \label{eq:Dirichlet}
  \capacity_{N}(\mathcal{A}, \mathcal{B})
  \;=\;
  \inf \bigl\{\mathcal{E}_{N}(f) \,:\, f \in \mathcal{H}^{N}_{\mathcal{A}, \mathcal{B}} \bigr\},
\end{align}
where $\mathcal{H}^{N}_{\mathcal{A}, \mathcal{B}} \ldef \bigl\{h\colon \mathcal{S}_{N} \to \mathbb{R} : 0 \leq h \leq 1,\, h\rvert_{\mathcal{A}} = 1,\, h\rvert_{\mathcal{B}} = 0 \bigr\}$
%
%
and
\begin{align}\label{def:D_form_E}
  \mathcal{E}_{N}(f)
  \;\ldef\;
  \frac{1}{2}\sum_{\sigma, \eta \in \mathcal{S}_{N}}\mu_{N}(\sigma)\, \pi_{N}(\sigma, \eta)\, \bigl( f(\sigma) - f(\eta) \bigr)^{2}
\end{align}
is the Dirichlet form. Analogously, the \emph{Thomson principle} states that
\begin{align}
  \label{eq:Thomson}
  \capacity_{N}(\mathcal{A}, \mathcal{B})
  \;=\;
  \sup\biggl\{\frac{1}{\mathcal{D}_{N}(\varphi)} : \varphi \in \mathcal{U}^{N}_{\mathcal{A}, \mathcal{B}} \biggr\}
  \;=\;
  \Bigl(
    \inf\bigl\{\mathcal{D}_{N}(\varphi) : \varphi \in \mathcal{U}^{N}_{\mathcal{A}, \mathcal{B}} \bigr\}
  \Bigr)^{-1},
\end{align}
where $\mathcal{U}^{N}_{\mathcal{A}, \mathcal{B}}$ denotes the space of all unit anti-symmetric $\mathcal{A},\mathcal{B}$-flows, while
\begin{align}\label{def:D_form_D}
  \mathcal{D}_{N}(\varphi)
  \;\ldef\;
  \frac{1}{2}\sum_{\sigma, \eta \in \mathcal{S}_{N}} \frac{1}{\mu_{N}(\sigma)\pi_{N}(\sigma, \eta)} \varphi(\sigma, \eta)^2.
\end{align}
These dual variational principles yield matching upper and lower bounds on the capacity when suitable test functions and flows are chosen. A common corollary of these principles is the Rayleigh shortcut method (see, e.g., \cite{DS84}). Let $\gamma\colon \{0, \dots, K\} \to \mathcal{S}_{N}$ be a cycle free path from $\mathcal{A}$ to $\mathcal{B}$ with $\gamma_{0} \in \mathcal{A}$, $\gamma_{K} \in \mathcal{B}$ and $\pi_{N}(\gamma_{n}, \gamma_{n+1}) > 0$ for all $n \in \{0, \dots, K-1\}$. From either variational principle we obtain
\begin{align}\label{eq:rayleigh_def}
  \capacity_{N}(\mathcal{A}, \mathcal{B})
  \;\geq\;
  \Biggl(
    \sum_{n=0}^{K-1}\frac{1}{\mu_{N}(\gamma_n)\pi_{N}(\gamma_n, \gamma_{n+1})}
  \Biggr)^{\!\!-1}
  \rdef\;
  \capacity_{\gamma, N}(\mathcal{A}, \mathcal{B})
\end{align}
We will denote by $\mathcal{E}_{N},\mathcal{D}_{N}, \widetilde{\mathcal{E}}_{N}$ and $\widetilde{\mathcal{D}}_{N}$ the forms defined for the specific cases of the quenched and annealed models respectively.
Sections~\ref{section:cap_mu} and \ref{sec:final} are devoted to controlling the denominator (the capacity) and the numerator (the harmonic sum), respectively.

\subsubsection*{Outline}
The paper is structured as follows: Section~\ref{app} focuses on proving Theorem~\ref{thm:meta:cwp}. It begins by describing the free energy of the CWP model and identifying its relevant critical points, as obtained from \cite{Le22}. Section~\ref{sec:meta-DCWP} introduces preliminary concentration inequalities for the comparison of both models and concludes with the proof of Theorem~\ref{thm:meta:dcwp}, as detailed in Section~\ref{sec:proofmeta_dcwp}. Section~\ref{section:cap_mu} provides annealed estimates and concentration inequalities for the capacities of the DCWP model. Section~\ref{app:concent} provides a derivation of the concentration inequalities used throughout the work. Section~\ref{sec:final} starts with estimates for the harmonic sum, both annealed and concentration estimates, which lead to the proof of our main Theorems~\ref{theo:ratio_conc} and~\ref{theo:ratio_mom}. 

\section{Metastability for the CWP model} \label{app}
In this section we study the metastability of the CWP model. We begin by characterising the critical points of the free energy in Proposition~\ref{prop:eland}, following mainly \cite{Le22}. We then introduce the lumped model, i.e.\ the model described by the mesoscopic order parameter representing the array of colours/spins frequencies. In Section~\ref{app:proofteo} we prove Proposition~\ref{prop:preem_cap_estimates}, which allows us to obtain rough estimates for the capacities of the annealed model. The same section contains the proof of Theorem~\ref{thm:meta:cwp}, which states that the CWP model is $\rho$-metastable.
%
%
%
%

%
%
%
%
We will now elaborate on the description of the free energy landscape of the CWP model, started in Section~\ref{subsec:main_results} and thoroughly explained in \cite[Section 3]{Le22}. In there, properties of critical points of the free energy landscape are described in terms of the relevant temperatures $0 < \beta_{1} < \beta_{2} < \beta_{3} \leq \beta_{4} = q$, where we simplify notation by omitting the $q$-dependence of the temperatures. The point $\boldsymbol{m}_{0} \ldef (1/q, \dots, 1/q)$ changes, as $\beta$ increases, from a global minimum to a local maximum of the free energy $\widetilde{F}_{\beta, q}$ defined in \eqref{eq:freeen2}. The points $\boldsymbol{m}_{1}, \dots, \boldsymbol{m}_{q}$ are the other local minima of $\widetilde{F}_{\beta, q}$, and the points $\boldsymbol{z}_{j,k}, j \neq k \in \{0, \dots, q \}$, are index-$1$ saddle points. All these properties are summarised in Proposition~\ref{prop:eland}. For its proof we refer to \cite{Le22}, where the points $\boldsymbol{m}_{i}$, $\boldsymbol{z}_{0, i}$, $\boldsymbol{z}_{j,k}$ are the solutions of Equations \cite[(3.2)--(3.4)]{Le22} denoted there by $\boldsymbol{u}_1^k$, $\boldsymbol{v}_{1}^{k}$ and $\boldsymbol{u}_{1}^{k,l}$ respectively. 

In the following proposition we will describe the energy landscape in terms of communication height and essential gates, as defined in \cite[Definition~10.2]{BdH15}. Let $\mathcal{G}_{j, k} \equiv \mathcal{G}(\boldsymbol{m}_{j}, \boldsymbol{m}_{k})$, $j \neq k$, be the set of essential gates between the local minima $\boldsymbol{m}_{j}$ and $\boldsymbol{m}_{k}$ as defined e.g.\ in \cite[Definition~10.2]{BdH15}, and let $c(\boldsymbol{m}_{j}, \boldsymbol{m}_{k})$ be the value that $\widetilde{F}_{\beta, q}$ takes at $\mathcal{G}(\boldsymbol{m}_{j}, \boldsymbol{m}_{k})$, also called \emph{communication height}.
\begin{definition}[Communication height]
  For $\boldsymbol{x}, \boldsymbol{y} \in \mathcal{P}$ we define the communication height $c(\boldsymbol{x}, \boldsymbol{y})$ by
  \begin{align}
    c(\boldsymbol{x},\boldsymbol{y})
    \;=\;
    \inf_{\substack{\gamma \in \mathcal{C}([0,1],\mathcal{P}) \\ \gamma(0) = \boldsymbol{x}, \gamma(1) = \boldsymbol{y}}}
    \max_{t \in [0, 1]} \widetilde{F}_{\beta, q}(\gamma(t))
  \end{align}
  That is, the infimum (over all paths from $\boldsymbol{x}$ to $\boldsymbol{y}$) of the maximal value attained along the path for the landscape $\widetilde{F}_{\beta, q}$.
\end{definition}
\begin{proposition}\label{prop:eland}
  Let $\beta_{1}, \dots, \beta_{4}$ be the ordered relevant temperatures of the CWP model. Then, for $i, j, k \neq 0$, $j \neq k$, the critical points of $\widetilde{F}_{\beta, q}$ are described by the following table:
  \vspace{0.2cm}
  \begin{center}
    \begin{tabular}{c||cccc}
      \hline
      $\beta \in $ & $ (0, \beta_{1}]$ & $(\beta_{1}, \beta_{2})$ & $\{\beta_{2}\}$ & $(\beta_{2},\beta_{3})$
      \\
      \hline
      $\boldsymbol{m}_{0}$ & global min. & global min. & global min. & local min.
      \\ 
      $\boldsymbol{m}_{i}$ & - & local min. & global min. & global min.
      \\   
      $\mathcal{G}_{0, i}$ & - & $\{\boldsymbol{z}_{0, i}\}$ & $\{\boldsymbol{z}_{0,i}\}$ & $\{\boldsymbol{z}_{0,i}\}$
      \\  
      $\mathcal{G}_{j, k}$ & - & $\{\boldsymbol{z}_{0, j}, \boldsymbol{z}_{0, k}\}$ & $\{\boldsymbol{z}_{0, j}, \boldsymbol{z}_{0, k}\}$ & $\{\boldsymbol{z}_{0, j}, \boldsymbol{z}_{0, k}\}$
      \\
      \hline
      $\beta \in$ & $\{\beta_{3}\}$ & $(\beta_{3}, \beta_{4})$ & $\{\beta_{4}\}$ & $(\beta_{4}, \infty)$
      \\
      \hline
      $\boldsymbol{m}_{0}$ & local min. & local min. & degenerate & local max.
      \\ 
      $\boldsymbol{m}_{i}$ & global min. & global min. & global min. & global min.
      \\  
      $\mathcal{G}_{0, i}$ & $\{\boldsymbol{z}_{0, i}\}$ & $\{\boldsymbol{z}_{0, i}\}$ & - & -
      \\  
      $\mathcal{G}_{j, k}$ & $\{\boldsymbol{z}_{0, j}, \boldsymbol{z}_{0, k}, \boldsymbol{z}_{j, k}\}$ & $\{\boldsymbol{z}_{j, k}\}$ & $\{\boldsymbol{z}_{j, k}\}$ & $\{\boldsymbol{z}_{j, k}\}$
      \\
      \hline
    \end{tabular}
  \end{center}
\end{proposition}
\begin{figure}[ht]
  \centering
  \begin{subfigure}{0.45\textwidth}
    \includegraphics[width=\textwidth]{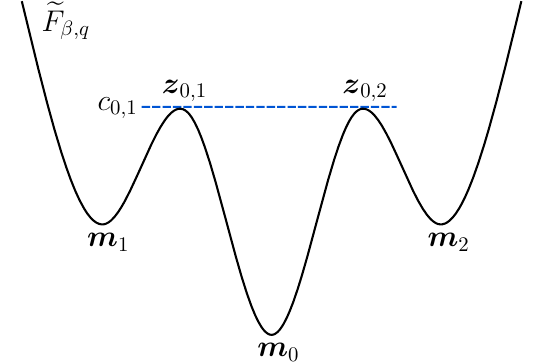}
    \caption{$\beta_{1} < \beta < \beta_{2}$}
  \end{subfigure}
  \begin{subfigure}{0.45\textwidth}
    \includegraphics[width=\textwidth]{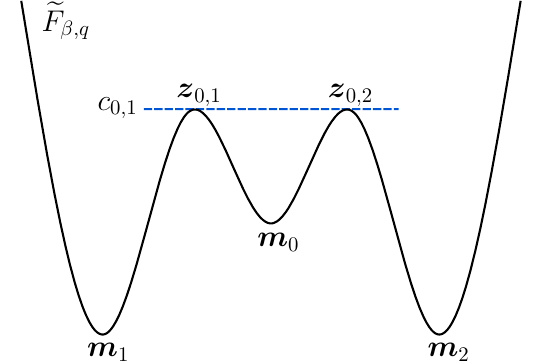}
    \caption{$\beta_{2} < \beta < \beta_{3}$}
  \end{subfigure}
  \begin{subfigure}{0.45\textwidth}
    \includegraphics[width=\textwidth]{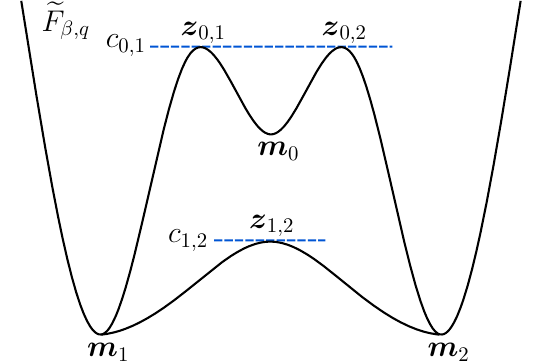}
    \caption{$\beta_{3} < \beta < \beta_{4}$}
  \end{subfigure}
  \begin{subfigure}{0.45\textwidth}
    \includegraphics[width=\textwidth]{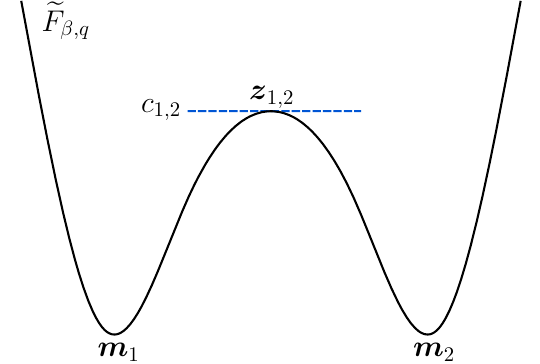}
    \caption{$\beta > \beta_{4}$}
  \end{subfigure}
  \caption{%
    Slices of the graph of $\widetilde{F}_{\beta,q}$, showing the local minima $\boldsymbol{m}_{i}$, the saddle points $\boldsymbol{z}_{i,j}$ and the communication heights $c_{i,j} \equiv c(\boldsymbol{m}_{i}, \boldsymbol{m_{j}})$. These are referenced in and used for the proofs of Proposition~\ref{prop:preem_cap_estimates} and the argument in Subsection~\ref{app:proofteo}. The landscape depicted in case (c) does not occur for $q=3$ or $q=4$, since $\beta_{3} = \beta_{4}$.
  }
  \label{fig_paths}
\end{figure}
Since the Markov process $(\widetilde{\Sigma}^{N}(t))_{t \geq 0}$ is lumpable, its metastable behavior can also be studied through the Markov process $\bigl(L_{N}(\widetilde{\Sigma}^{N}(t))\bigr)_{t \geq 0}$ defined via the empirical measure \eqref{eq:def_rho}, see e.g.~\cite[Section 9.3]{BdH15} for further properties of lumpable Markov chains.
Note that the Markov chain, $\bigl(L_{N}(\widetilde{\Sigma}^{N}(t))\bigr)_{t \geq 0}$, behaves like a weighted nearest neighbour random walk in the space $\mathcal{P}_{ N}$. Its (mesoscopic) transition probabilities can be computed from $\widetilde{\pi}_{N}$ as follows. For every $\boldsymbol{x}, \boldsymbol{y} \in \mathcal{P}_{N}$, let
\begin{align}
  \widetilde{p}_N(\boldsymbol{x}, \boldsymbol{y})
  \;=\;
  \frac{1}{\widetilde{Q}_{N}(\boldsymbol{x})}\,
  \sum_{\sigma \in L_{N}^{-1}(\boldsymbol{x})} \widetilde{\mu}_{N}(\sigma)
  \sum_{\eta \in L_{N}^{-1}(\boldsymbol{y})} \widetilde{\pi}_{N}(\sigma, \eta),
\end{align}
where $\widetilde{Q}_{N} \ldef \widetilde{\mu}_{N} \circ L_{N}^{-1}$ denotes the macroscopic equilibrium measure that can also be expressed as
\begin{align}\label{eq:Qn_expanded}
  \widetilde{Q}_N(\boldsymbol{x})
  \;=\;
  \binom{N}{N\boldsymbol{x}_{1}, \dots, N\boldsymbol{x}_{q}}\,
  \frac{\me^{\frac{N\beta}{2} (-\norm{\boldsymbol{x}}_2^2 - N^{-1})}}{\widetilde{Z}_{N}}
  \;=\; 
  \frac{\exp(-\beta N \widetilde{F}_{N}(\boldsymbol{x}))}{\widetilde{\boldsymbol{Z}}_{N}},
\end{align}
where $(2 \pi N)^{(q-1)/2} \widetilde{Z}_{N} = \widetilde{\boldsymbol{Z}}_{N}$ and
\begin{align}
  \widetilde{F}_{N}(\boldsymbol{x})
  \;=\;
  -\frac{1}{2}\, \norm{\boldsymbol{x}}_{2}^{2}
  + \frac{1}{2N} - \frac{1}{N \beta} \ln\binom{N}{N\boldsymbol{x}_{1}, \dots, N\boldsymbol{x}_{q}}
  - \frac{(q-1)}{2N\beta}\, \ln(2 \pi N).
\end{align}
By \cite[Lemma 2.2]{CS04} we have the following uniform bounds, valid for all $\boldsymbol{x} \in \mathcal{P}_{N}$,
\begin{align}\label{ineq: F-Fn_unif}
  -\frac{q-1}{N\beta}\, \ln\biggl(\frac{N+q-1}{q-1}\biggr)
  \;\leq\;
  \widetilde{F}_{\beta, q}(\boldsymbol{x}) - \widetilde{F}_{N}(\boldsymbol{x}) + \frac{1}{2N} - \frac{q-1}{2N\beta} \ln(2\pi N)
  \;\leq\;
  0,
\end{align}
where $\widetilde{F}_{\beta, q}$ is defined as in \eqref{eq:freeen2}. When restricting to compact subsets of the interior of $\mathcal{P}$, by means of the Stirling formula, the convergence speed is improved. More precisely, under these conditions
\begin{align}\label{eq:Fbeta_lim}
  \widetilde{F}_{N} \;=\; \widetilde{F}_{\beta, q} + O\bigl(1/N\bigr).
\end{align}
Moreover, for any $\boldsymbol{x}, \boldsymbol{y} \in \mathcal{P}_{N}$ such that $\boldsymbol{y} = \boldsymbol{x} + \hat{e}_{i} - \hat{e}_{j}$ for some $i, j \in \{1, \ldots, q\}$ with $i \ne j$, where $\hat{e}_{\ell} = \frac{1}{N} e_{\ell}$ denotes the rescaled unit vector in $\mathbb{R}^{q}$ in coordinate direction $\ell$, we obtain by an elementary computation that 
\begin{align}\label{eq:meso_rates}
  \widetilde{p}_{N}(\boldsymbol{x}, \boldsymbol{x} + \hat{e}_{i} - \hat{e}_{j})
  \;=\;
  \frac{\boldsymbol{x}_{j}}{q}\, \me^{-\frac{\beta N}{2}\left[-\norm{\boldsymbol{x} + \hat{e}_{i} - \hat{e}_{j}}_2^2 + \norm{\boldsymbol{x}}_{2}^{2}\right]_+}.
\end{align}
For any $N \in \mathbb{N}$ and non-empty, disjoint subsets $\boldsymbol{A}, \boldsymbol{B} \subset \mathcal{P}_{N}$ set $\mathcal{A} \ldef L_{N}^{-1}(\boldsymbol{A})$ and $\mathcal{B} = L_{N}^{-1}(\boldsymbol{B})$. Further, define \emph{macroscopic capacity} $\capacitym_{N}(\boldsymbol{A}, \boldsymbol{B})$ of the sets $\boldsymbol{A}$ and $\boldsymbol{B}$ by
\begin{align}
  \label{eq:meso:cap:variation}
  \capacitym_{N}(\boldsymbol{A}, \boldsymbol{B})
  \;=\;
  \inf_{g \in \boldsymbol{\mathcal{H}}_{\boldsymbol{A}, \boldsymbol{B}}^{N}}
  \biggl\{
    \frac{1}{2} \sum_{\boldsymbol{x}, \boldsymbol{y} \in \mathcal{P}_{N}}\mspace{-9mu}
    \widetilde{Q}_{N}(\boldsymbol{x})\,
    \widetilde{p}_{N}(\boldsymbol{x}, \boldsymbol{y})\,
    \bigl( g(\boldsymbol{x}) - g(\boldsymbol{y}) \bigr)^{2}
  \biggr\},
\end{align}
where $\boldsymbol{\mathcal{H}}_{\boldsymbol{A}, \boldsymbol{B}}^{N}$ denotes the set of all functions $g\colon \mathcal{P}_{N} \to [0, 1]$ with $g|_{\boldsymbol{A}} = 1$ and $g|_{\boldsymbol{B}} = 0$. Then, in view of the lumpability of the Markov chain $(\widetilde{\Sigma}^{N}(t))_{t \geq 0}$, cf.\ \cite[Theorem 9.7]{BdH15}, the \emph{microscopic capacity} $\widetilde{\capacity}(\mathcal{A}, \mathcal{B})$ coincides with the \emph{macroscopic capacity} $\capacitym_{N}(\boldsymbol{A}, \boldsymbol{B})$, that is,
\begin{align}\label{mesocap}
  \widetilde{\capacity}_{N}(\mathcal{A}, \mathcal{B})
  \;=\;
  \capacitym_{N}(\boldsymbol{A}, \boldsymbol{B}).
\end{align}
In particular, for $g_{\boldsymbol{A}, \boldsymbol{B}}^{N} \in \boldsymbol{\mathcal{H}}_{\boldsymbol{A}, \boldsymbol{B}}^{N}$ being the minimizer of the right-hand side of \eqref{eq:meso:cap:variation} it follows from \eqref{mesocap} that, for any $\boldsymbol{x} \in \mathcal{P}_{N} \setminus (\boldsymbol{A} \cup \boldsymbol{B})$ and $\sigma, \eta \in L_{N}^{-1}(\boldsymbol{x})$,
\begin{align}
  \label{eq:lumpprop}
  \widetilde{\Prob}^{N}_{\!\sigma}\bigl[ \tau_{\mathcal{A}}^{N} < \tau_{\mathcal{B}}^{N} \bigr]
  \;=\;
  g_{\boldsymbol{A}, \boldsymbol{B}}^{N}(L_{N}(\sigma))
  \;=\;
  g_{\boldsymbol{A}, \boldsymbol{B}}^{N}(\boldsymbol{x})
  \;=\;
  g_{\boldsymbol{A}, \boldsymbol{B}}^{N}(L_{N}(\eta))
  \;=\;
  \widetilde{\Prob}^{N}_{\!\eta}\bigl[ \tau_{\mathcal{A}}^{N} < \tau_{\mathcal{B}}^{N} \bigr].
\end{align}

\subsection{Proof of Theorem~\ref{thm:meta:cwp}} \label{app:proofteo}
The proof relies on the following proposition  providing rough estimates for the macroscopic capacity of the CWP model. 
\begin{proposition}\label{prop:preem_cap_estimates}
  Let $\beta > \beta_{1}$.
  \begin{enumerate}[(i)]
  \item
    Let $(\boldsymbol{x}_{N})_{N \in \mathbb{N}}$ and $(\boldsymbol{y}_{N})_{N \in \mathbb{N}}$ be two sequences in $\mathcal{P}$ with $\boldsymbol{x}_{N}, \boldsymbol{y}_{N} \in \mathcal{P}_{N}$ for every $N \in \mathbb{N}$. Suppose $\lim_{N \to \infty} \boldsymbol{x}_{N} = \boldsymbol{x}$ and $\lim_{N \to \infty} \boldsymbol{y}_{N} = \boldsymbol{y}$, where $\boldsymbol{x}, \boldsymbol{y}$ lie in a compact subset of the interior of $\mathcal{P}$ and are separated by a communication height $c(\boldsymbol{x}, \boldsymbol{y}) > \widetilde{F}(\boldsymbol{x}) \vee \widetilde{F}(\boldsymbol{y})$. Then there exist $\ell_{1} \equiv \ell_{1}(q)$, $\ell_{2} \equiv \ell_{2}(q) \in (0, \infty)$ and $N_{0}(\beta, q) \in \mathbb{N}$ such that, for any $N \geq N_{0}(\beta, q)$,
    \begin{align}\label{upper_cap_final}
      \frac{1}{N^{\ell_1}}
      \;\leq\;
      \frac{\capacitym_{N}(\boldsymbol{x}_{N}, \boldsymbol{y}_{N})}{\widetilde{Q}_{N}(\boldsymbol{x}_{N})}\,
      \exp\bigl(
        \beta N (c(\boldsymbol{x}, \boldsymbol{y}) - \widetilde{F}_{\beta, q}(\boldsymbol{x}_{N}))
      \bigr)
      \;\leq\;
      N^{\ell_{2}}\, .
    \end{align}
    
  \item
    Set $\boldsymbol{M}_{\!N} \ldef \{\boldsymbol{m}_{i, N} : i \in \mathcal{I}_{\beta}\} \subset \mathcal{P}_{N}$ where $\boldsymbol{m}_{i, N}$ denotes a closest lattice point approximation of $\boldsymbol{m}_{i}$ for any $i \in \mathcal{I}_{\beta}$. Then, there exist $\ell_{3} \equiv \ell_{3}(\beta, q) \in [0, \infty)$ and $N_{0}(\beta) \in \mathbb{N}$ such that, for any $N \geq N_{0}(\beta)$ and $\boldsymbol{z} \in \mathcal{P}_{N}$,
    \begin{align}\label{lower_cap_final}
      \frac{\capacitym_{N}(\boldsymbol{z}, \boldsymbol{M}_{\!N})}{\widetilde{Q}_{N}(\boldsymbol{z})}
      \;\geq\;
      N^{-\ell_3}.
    \end{align}
  \end{enumerate}
\end{proposition}
\begin{proof}
  First, let us introduce some notation that we will use throughout the proof. For $\varepsilon \in (0, 1)$, define the $\varepsilon$-interior of $\mathcal{P}$ by
  \begin{align*}
    \mathcal{P}^{\varepsilon}
    \;\ldef\;
    \bigl\{
      \boldsymbol{x} \in \mathcal{P}
      :
      \boldsymbol{x}_i > \epsilon,\, \forall\, i \in \{1, \dots, q\}
    \bigr\}.
  \end{align*}
  For any $\ell \in \{1, \ldots, q\}$ write $\partial \mathcal{P}^{\varepsilon}_{\ell}$ for the lateral boundary of $\mathcal{P}^{\varepsilon}$ and $\boldsymbol{v}^{\ell} \in \mathbb{R}^{q}$ for the corresponding (outward) normal direction that are given through
  \begin{align*}
    \begin{split}
      \partial \mathcal{P}^{\varepsilon}_{\ell}
      \;\ldef\;&
      \bigl\{
        \boldsymbol{x} \in \mathcal{P}
        :
        \boldsymbol{x}_{\ell} = \varepsilon,\,
        \boldsymbol{x}_{k} \geq \varepsilon,\,
        \forall\, k \in \{1, \ldots, q\} \setminus \{\ell\}
      \bigr\}
      \\[.5ex]
      \boldsymbol{v}^{\ell}
      \;=\;&
      (\boldsymbol{v}_{1}^{\ell}, \ldots, \boldsymbol{v}_{q}^{\ell})
      \qquad \text{with} \qquad
      \boldsymbol{v}_{k}^{\ell}
      \;\ldef\;
      \indicator_{k = \ell} - \frac{1}{q}.
    \end{split}
  \end{align*}
  For any $s \in \mathbb{R}$, let $V(s)$ be the sub-level set of the limiting free energy $F_{\beta, q}$, that is,
  \begin{align*}
    V(s)
    \;\ldef\;
    \bigl\{
      \boldsymbol{x} \in \mathcal{P}
      :
      \widetilde{F}_{\beta, q}(\boldsymbol{x}) < s
    \bigr\}.
  \end{align*}
  
  Now, let $\beta > \beta_{1}$ and set $\varepsilon_{0} \equiv \varepsilon_{0}(\beta, q) \ldef 1/(\me^{\beta} + q - 1)$. Since all critical points of $\widetilde{F}_{\beta, q}$ lie in the interior of $\mathcal{P}$, we may choose $\varepsilon \in (0, \varepsilon_{0})$ sufficiently small so that
  \begin{align}\label{eq:cond:eps:(i)}
    \bigl\{
      \boldsymbol{m}_{i} : i \in \mathcal{I}_{\beta}
    \bigr\}
    \qquad\text{and}\qquad
    \bigl\{
      \boldsymbol{z}_{i, j}
      :
      i, j \in \mathcal{I}_{\beta},\, i \neq j
    \bigr\}
    \;\subset\;
    \mathcal{P}^{\varepsilon}.
  \end{align}
  \begin{figure}[t]
    \centering
    \includegraphics[width=7cm,keepaspectratio]{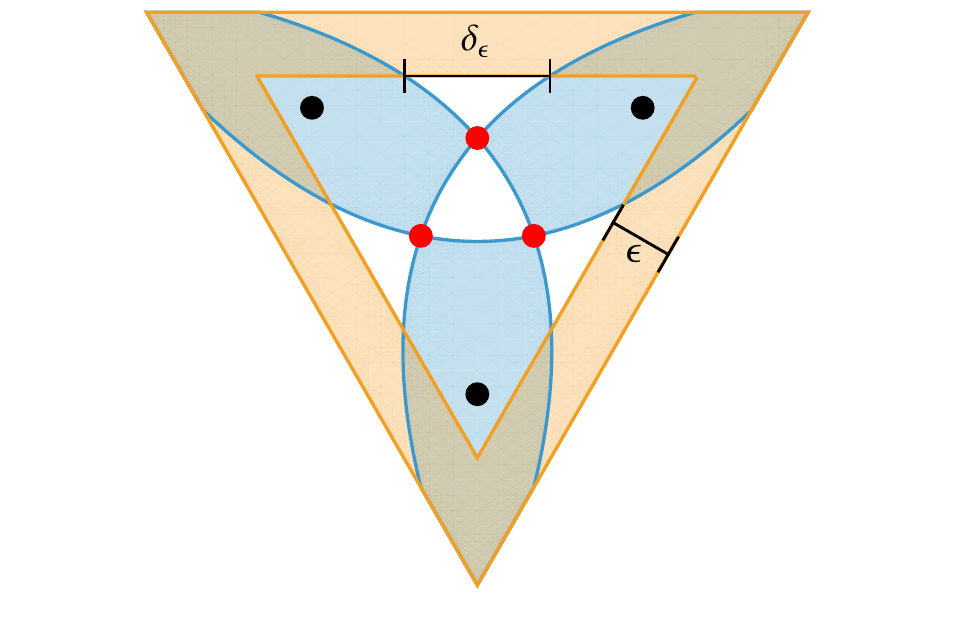}
    \begin{tikzpicture}[xscale=2.5, yscale=2.5, >=stealth]

      \def\rt{1.73205080757} 
      \def\shifty{0.8}       
      \def\trimy{1.6}        

      \coordinate (V_out) at (0,0);
      \coordinate (A_out_R) at ({1.1/\rt}, 1.1);
      \coordinate (A_out_L) at ({-1.1/\rt}, 1.1);
      \coordinate (V_in) at (0, \shifty);
      
      \pgfmathsetmacro{\dotsY}{\trimy - 0.1}
      \coordinate (End_out_R) at ({\dotsY/\rt}, \dotsY);
      \coordinate (End_out_L) at ({-\dotsY/\rt}, \dotsY);
      \coordinate (End_in_R) at ({(\dotsY-\shifty)/\rt}, \dotsY);
      \coordinate (End_in_L) at ({-(\dotsY-\shifty)/\rt}, \dotsY);

      \def\vfinaly{1.1}
      \pgfmathsetmacro{\vfinalx}{(\vfinaly - \shifty)/\rt-0.05}

      \def\stepx{0.05}
      \pgfmathsetmacro{\stepy}{\stepx*\rt}
      
      \coordinate (glob_min) at (0,0.9*\rt);
      \coordinate (V6) at (\vfinalx, \vfinaly); 
      \coordinate (V5) at (\vfinalx+\stepx, \vfinaly-\stepy);
      \coordinate (V4) at (\vfinalx, \vfinaly-2*\stepy);
      \coordinate (V3) at (\vfinalx+\stepx, \vfinaly-3*\stepy);
      \coordinate (V2) at (\vfinalx+2*\stepx, \vfinaly-4*\stepy);
      \coordinate (V1) at (\vfinalx+1*\stepx, \vfinaly-5*\stepy);

      \tikzset{
        boundary/.style={thick, orange}, 
        pathgamma_discr/.style={draw=blue, thick, solid, every node/.style={circle, fill=blue, inner sep=1.2pt}},
        pathgamma_smooth/.style={draw=black, thick, dashed,every node/.style={circle, fill=blue, inner sep=1.2pt}}
      }

      \draw[boundary] (End_out_L) -- (V_out) -- (End_out_R);
      \draw[boundary, opacity=0.6] (End_in_L) -- (V_in) -- (End_in_R);

      \node[orange, rotate=60] at ($(End_out_R)+(0.02,0.035)$) {$\dots$};
      \node[orange, rotate=-60] at ($(End_out_L)+(-0.02,0.035)$) {$\dots$};
      \node[orange, opacity=0.6, rotate=60] at ($(End_in_R)+(0.02,0.035)$) {$\dots$};
      \node[orange, opacity=0.6, rotate=-60] at ($(End_in_L)+(-0.02,0.035)$) {$\dots$};

      \node[circle, fill=black, minimum size=3.2pt, inner sep=0pt] at (glob_min) {};
      \node[font=\small] at ($(glob_min)+(0.15,0.05)$) {$\boldsymbol{m}^\ast$};

      \draw[pathgamma_discr] (V1) node {} -- (V2) node {} -- (V3) node {} -- (V4) node {} -- (V5) node {} -- (V6) node {};

      \draw[pathgamma_smooth] (V6) .. controls ($ (V6) + (-0.2, 0.3) $)  .. (glob_min) ;
      \node[anchor=west, font=\small] at ($(V6) + (-0.15, 0.28) $) {$\Gamma^N$};
      \node[anchor=west, blue, font=\small] at ($ (V2) - (0.2, 0.2) $) {$\gamma^N$};
      \node[anchor=east, blue, font=\small] at ($ (V6)  $) {$\boldsymbol{z}^\ast$};

      \clip (-0.5, -0.1) rectangle (0.6, 1.1);

    \end{tikzpicture}
    \caption{%
      Illustration of $\mathcal{P}$ for $q=3$: (Left) $V(c(\boldsymbol{x},\boldsymbol{y}))$ is shown in blue and $\mathcal{P} \setminus \mathcal{P}^{\varepsilon}$ in orange. Local minima are indicated by black dots, and essential gates are highlighted in red. (Right) Illustration of the path $(\boldsymbol{\gamma}_{0}^{N}, \ldots, \boldsymbol{\gamma}_{k_{N}}^{N})$ from $\boldsymbol{z}$ to some $\boldsymbol{z}^{\ast} \in \mathcal{P}^{\varepsilon} \cap \mathcal{P}_{N}$ used in the proof of Proposition \ref{prop:preem_cap_estimates}, part (ii).
    }
    \label{fig:level_sets}
  \end{figure}
  Consequently, there exists $\delta_{\varepsilon} \in (0, 1)$ such that, for any two distinct connected components $\boldsymbol{V}_{\!1}$ and $\boldsymbol{V}_{\!2}$ of the sub-level set $\boldsymbol{V} \ldef V(c(\boldsymbol{x}, \boldsymbol{y}))$, the Euclidean distance between $\boldsymbol{V}_{\!1} \cap (\mathcal{P} \setminus \mathcal{P}^{\varepsilon})$ and $\boldsymbol{V}_{\!2} \cap (\mathcal{P} \setminus \mathcal{P}^{\varepsilon})$ exceeds $\delta_{\varepsilon}$. Indeed, suppose to the contrary that this distance is zero for some distinct connected component $\boldsymbol{V}_{\!1}$ and $\boldsymbol{V}_{\!2}$ of $\boldsymbol{V}$. Then the closures of $\boldsymbol{V}_{\!1} \cap (\mathcal{P} \setminus \mathcal{P}^{\varepsilon})$ and $\boldsymbol{V}_{\!1} \cap (\mathcal{P} \setminus \mathcal{P}^{\varepsilon})$ intersect, so any point in $\overbar{\boldsymbol{V}_{\!1} \cap (\mathcal{P} \setminus \mathcal{P}^{\varepsilon})} \cap \overbar{\boldsymbol{V}_{\!2} \cap (\mathcal{P} \setminus \mathcal{P}^{\varepsilon})} \subset \mathcal{P} \setminus \mathcal{P}^{\varepsilon}$ is an essential gate between $\boldsymbol{x}$ and $\boldsymbol{y}$, contradicting \eqref{eq:cond:eps:(i)}.
  
  Finally, since $\widetilde{F}_{\beta, q}$ is locally Lipschitz continuous and in view of \eqref{eq:Fbeta_lim}, there exist constants $K_{\varepsilon, 1}, K_{\varepsilon, 2} \in (0, \infty)$ such that for every $\boldsymbol{z}, \boldsymbol{z}' \in \mathcal{P}^{\varepsilon/2}$ and $N \in \mathbb{N}$,
  \begin{align}\label{eq:Lcont+unif}
    \bigl|
      \widetilde{F}_{\beta, q}(\boldsymbol{z})
      - \widetilde{F}_{\beta, q}(\boldsymbol{z}')
    \bigr|
    \;\leq\;
    K_{\varepsilon, 1}\, \|\boldsymbol{z} - \boldsymbol{z'}\|_{2}
    \qquad \text{and} \qquad
    \bigl|
      \widetilde{F}_{N}(\boldsymbol{z})
      - \widetilde{F}_{\beta, q}(\boldsymbol{z})
    \bigr|
    \;\leq\;
    \frac{K_{\varepsilon, 2}}{N}.
  \end{align}
  
  \textit{(i)}
  Let $(\boldsymbol{x}_{N})_{N \in \mathbb{N}}$ and $(\boldsymbol{y}_{N})_{N \in \mathbb{N}} \subset \mathcal{P}_{N}$ be two sequences with $\boldsymbol{x}_{N}, \boldsymbol{y}_{N} \in \mathcal{P}_{N}$ for every $N \in \mathbb{N}$. Suppose that $\boldsymbol{x}_{N} \to \boldsymbol{x}$ and $\boldsymbol{y}_{N} \to \boldsymbol{y}$ as $N \to \infty$, where $\boldsymbol{x}, \boldsymbol{y} \in \operatorname{int}(\mathcal{P})$ and the communication height between $\boldsymbol{x}$ and $\boldsymbol{y}$ satisfies $\eta \ldef c(\boldsymbol{x}, \boldsymbol{y}) - \max\{\widetilde{F}_{\beta, q}(\boldsymbol{x}), \widetilde{F}_{\beta, q}(\boldsymbol{y})\} > 0$. Write $\boldsymbol{V}_{\!\boldsymbol{x}}$ to denote the connected component of $\boldsymbol{V}$ containing $\boldsymbol{x}$, and set $\boldsymbol{V}_{\!\boldsymbol{x}, N} \ldef \boldsymbol{V}_{\!\boldsymbol{x}} \cap \mathcal{P}_{N}$. Since $\widetilde{F}_{\beta, q}$ is uniformly continuous, there exists $\delta_{\eta} > 0$ such that $|\widetilde{F}_{\beta, q}(\boldsymbol{z}) - \widetilde{F}_{\beta, q}(\boldsymbol{z}')| < \eta/2$ for all $\boldsymbol{z}, \boldsymbol{z}' \in \mathcal{P}$ with $\|\boldsymbol{z} - \boldsymbol{z}'\| < \delta_{\eta}$. Hence, there exists $\bar{N} \equiv \bar{N}(\delta_{\eta}, \delta_{\varepsilon}) \in \mathbb{N}$ such that, for all $N \geq \bar{N}$, we have $\|\boldsymbol{x}_{N} - \boldsymbol{x}\| < \delta_{\eta}$, $\|\boldsymbol{y}_{N} - \boldsymbol{y}\| < \delta_{\eta}$, and any two distinct points $\boldsymbol{z}, \boldsymbol{z}' \in \mathcal{P}_{N}$ with $\widetilde{p}_{N}(\boldsymbol{z}, \boldsymbol{z}') > 0$ satisfy $\|\boldsymbol{z} - \boldsymbol{z}'\|_{2} < \min\{\delta_{\varepsilon}, \varepsilon\}/2$. In particular, for every $N \geq \bar{N}$, $\boldsymbol{x}_{N} \in \boldsymbol{V}_{\!\boldsymbol{x}}$, $\boldsymbol{y}_{N} \in \boldsymbol{V} \setminus \boldsymbol{V}_{\!\boldsymbol{x}}$, and the function $g\colon \mathcal{P}_{N} \to [0, 1]$, $\boldsymbol{z} \mapsto g(\boldsymbol{z}) \ldef \indicator_{\boldsymbol{V}_{\!\boldsymbol{x}}}(\boldsymbol{z})$ belongs to $\boldsymbol{\mathcal{H}}_{\boldsymbol{x}_{N}, \boldsymbol{y}_{N}}^{N}$. Therefore, by applying the Dirichlet principle,
  \begin{align}\label{eq:cap:xN:yN:ub}
    \capacitym_{N}(\boldsymbol{x}_{N}, \boldsymbol{y}_{N})
    \overset{\eqref{mesocap}}{\;\leq\;}
    \sum_{\boldsymbol{z} \in \boldsymbol{V}_{\!\boldsymbol{x},N}}
    \sum_{\boldsymbol{z}' \notin \boldsymbol{V}_{\!\boldsymbol{x},N}}
    \widetilde{Q}_{N}(\boldsymbol{z})\,
    \widetilde{p}_{N}(\boldsymbol{z}, \boldsymbol{z}').
  \end{align}
  To lighten notations, we write $\partial \boldsymbol{V}_{\boldsymbol{x}, N} \ldef \{(\boldsymbol{z}, \boldsymbol{z}') : \boldsymbol{z} \in \boldsymbol{V}_{\!\boldsymbol{x}, N},\, \boldsymbol{z}' \not\in \boldsymbol{V}_{\!\boldsymbol{x}, N},\, \widetilde{p}_{N}(\boldsymbol{z}, \boldsymbol{z}') > 0\}$. By distinguishing two different cases it follows from the continuity of $\widetilde{F}_{\beta, q}$ that, for any $N \geq \bar{N}$ and  $(\boldsymbol{z}, \boldsymbol{z}') \in \partial \boldsymbol{V}_{\!\boldsymbol{x}, N}$ with $\boldsymbol{z} \in \boldsymbol{V}_{\!\boldsymbol{x}} \cap \mathcal{P}^{\varepsilon}$, there exists $\bar{\boldsymbol{z}}$ on the line segment $\{\boldsymbol{z} + t (\boldsymbol{z}' - \boldsymbol{z}) : t \in [0, 1]\}  \subset \mathcal{P}^{\varepsilon/2}$ such that $\widetilde{F}_{\beta, q}(\bar{\boldsymbol{z}}) = c(\boldsymbol{x}, \boldsymbol{y})$, and
  \begin{align*}
    \bigl|
      \widetilde{F}_{N}(\boldsymbol{z}) - c(\boldsymbol{x}, \boldsymbol{y})
    \bigr|
    \;\leq\;
    \bigl|
      \widetilde{F}_{N}(\boldsymbol{z})
      - \widetilde{F}_{\beta, q}(\boldsymbol{z})
    \bigr|
    +
    \bigl|
      \widetilde{F}_{\beta, q}(\boldsymbol{z})
      - \widetilde{F}_{\beta, q}(\bar{\boldsymbol{z}})
    \bigr|
    \overset{\eqref{eq:Lcont+unif}}{\;\leq\;}
    \frac{\sqrt{2} K_{\varepsilon, 1} + K_{\varepsilon, 2}}{N}
    .
  \end{align*}
  Hence, by setting $c_{\beta, \varepsilon} \ldef \me^{\beta (\sqrt{2} K_{\varepsilon, 1} + K_{\varepsilon, 2})}$, 
  \begin{align}\label{eq:measure:xN:yN:1}
    \widetilde{Q}_{N}(\boldsymbol{z})
    \;\leq\;
    c_{\beta, \varepsilon}\,
    \frac{\exp\bigl(-\beta N c(\boldsymbol{x}, \boldsymbol{y})\bigr)}{\widetilde{\boldsymbol{Z}}_{N}},
    \qquad
    \forall\, (\boldsymbol{z}, \boldsymbol{z}') \in \partial \boldsymbol{V}_{\!\boldsymbol{x}, N}
    \text{ with } \boldsymbol{z} \in \boldsymbol{V}_{\!\boldsymbol{x}} \cap \mathcal{P}^{\varepsilon}.
  \end{align}
  Next, consider $(\boldsymbol{z}, \boldsymbol{z}') \in \partial \boldsymbol{V}_{\!\boldsymbol{x}, N}$ with $\boldsymbol{z} \in \boldsymbol{V}_{\!\boldsymbol{x}} \cap (\mathcal{P} \setminus \mathcal{P}^{\varepsilon})$. By the choice of $\delta_{\varepsilon}$ and $\bar{N}$ it follows that, for any $N \geq \bar{N}$, $\|\boldsymbol{z} - \boldsymbol{z}'\| = \sqrt{2}/N < \delta_{\varepsilon}/2$ and, hence, $\boldsymbol{z}' \not\in \boldsymbol{V}$. Therefore, $\widetilde{F}_{\beta, q}(\boldsymbol{z}') \geq c(\boldsymbol{x}, \boldsymbol{y})$ and
  \begin{align*}
    \widetilde{F}_{N}(\boldsymbol{z}')
    \overset{\eqref{ineq: F-Fn_unif}}{\;\geq\;}
    \widetilde{F}_{\beta, q}(\boldsymbol{z}')
    + \frac{1}{2N} - \frac{q-1}{2 N \beta}\, \ln(2 \pi N)
    \;\geq\;
    c(\boldsymbol{x}, \boldsymbol{y})
    - \frac{q-1}{2 N \beta}\, \ln(2 \pi N).
  \end{align*}
  Consequently, for all $N \geq \bar{N}$,
  \begin{align}\label{eq:measure:xN:yN:2}
    \widetilde{Q}_{N}(\boldsymbol{z}')
    \;\leq\;
    (2\pi N)^{(q-1)/2}\,
    \frac{\exp\bigl(-\beta N c(\boldsymbol{x}, \boldsymbol{y})\bigr)}{\widetilde{\boldsymbol{Z}}_{N}},
    \quad
    \forall\, (\boldsymbol{z}, \boldsymbol{z}') \in \partial \boldsymbol{V}_{\!\boldsymbol{x}, N}
    \text{ with } \boldsymbol{z} \in \boldsymbol{V}_{\!\boldsymbol{x}} \cap (\mathcal{P} \setminus \mathcal{P}^{\varepsilon}).
  \end{align}
  Substituting \eqref{eq:measure:xN:yN:1} and \eqref{eq:measure:xN:yN:2}, after an application of the detailed-balance condition, that reads $\widetilde{Q}_{N}(\boldsymbol{z}) \widetilde{p}_{N}(\boldsymbol{z}, \boldsymbol{z}') = \widetilde{Q}_{N}(\boldsymbol{z}') \widetilde{p}_{N}(\boldsymbol{z}', \boldsymbol{z})$, into \eqref{eq:cap:xN:yN:ub}, and using the fact that the cardinality of the set $\boldsymbol{V}_{\!\boldsymbol{x}, N}$ is bounded by $N^{q-1}$, it follows that for every $N \geq \bar{N}$,
  \begin{align}\label{eq:cap:xN:yN:ub:final}
    \capacitym_{N}(\boldsymbol{x}_{N}, \boldsymbol{y}_{N})
    \;\leq\;
    \max\bigl\{
      c_{\beta, \varepsilon}, (2\pi)^{(q-1)/2}\, N^{(q-1)/2}
    \bigr\}\, N^{q-1}\,
    \frac{
      \exp(-\beta N c(\boldsymbol{x}, \boldsymbol{y}))
    }{\widetilde{\boldsymbol{Z}}_{N}}.
  \end{align}

  The lower bound on $\capacitym_{N}(\boldsymbol{x}_{N}, \boldsymbol{y}_{N})$ is proved using Rayleigh's shortcut method. For any $N \geq \bar{N}$, we have $c(\boldsymbol{x}, \boldsymbol{y}) = c(\boldsymbol{x}_{N}, \boldsymbol{y}_{N})$. Hence, there exists a continuous path $\Gamma^{N}\colon [0, 1] \to \operatorname{int}(\mathcal{P})$ with $\Gamma^{N}(0) = \boldsymbol{x}_{N}$, $\Gamma^{N}(1) = \boldsymbol{y}_{N}$, $\max_{t \in [0, 1]} \widetilde{F}_{\beta, q}(\Gamma^{N}(t)) = c(\boldsymbol{x}, \boldsymbol{y})$ and $\Gamma^{N}(t) \in \operatorname{int}(\boldsymbol{V})$ for all except of finitely many $t \in [0, 1]$. Let $\boldsymbol{\gamma}^{N} = (\boldsymbol{\gamma}^{N}_{0}, \dots, \boldsymbol{\gamma}^{N}_{K_N})$ be a vertex-disjoint, lattice approximation of $\Gamma^{N}$ with $\boldsymbol{\gamma}_{i}^{N} \in \mathcal{P}_{N}$ and $\min_{t \in [0, 1]} \|\boldsymbol{\gamma}_{i}^{N} - \Gamma^{N}(t)\|_{2} \leq \sqrt{q} / N$ for all $i \in \{0, \ldots, K_{N}\}$ and  $\widetilde{p}_{N}(\boldsymbol{\gamma}_{i}^{N}, \boldsymbol{\gamma}_{i+1}^{N}) > 0$ for all $i \in \{0, \ldots, K_{N}-1\}$. Moreover, if $\Gamma^{N}$ contains line segments lying in $\mathcal{P} \setminus \mathcal{P}^{\varepsilon}$, the corresponding part of the lattice path $\boldsymbol{\gamma}^{N}$ is contained in $\boldsymbol{V} \cap \mathcal{P}_{N}$. Then, by applying Rayleigh's shortcut method,
  \begin{align}\label{eq:cap:xN:yN:lb}
    \capacitym_{N}(\boldsymbol{x}_N, \boldsymbol{y}_{N})
    \overset{\eqref{eq:rayleigh_def}}{\;\geq\;}
    \Biggl(
      \sum_{i=0}^{K_{N}-1}
      \frac{1}
      {%
        \widetilde{Q}_{N}(\boldsymbol{\gamma}^{N}_{i})
        \widetilde{p}_{N}(\boldsymbol{\gamma}^{N}_{i}, \boldsymbol{\gamma}^{N}_{i+1})
      }
    \Biggr)^{\!\!-1}
    .
  \end{align}
  Notice that, for any $i \in \{0, \ldots, K_{N}\}$ such that $\boldsymbol{\gamma}_{i}^{N} \in \boldsymbol{V} \cap \mathcal{P}_{N}$ we have $\widetilde{F}_{\beta, q}(\boldsymbol{\gamma}_{i}^{N}) < c(\boldsymbol{x}, \boldsymbol{y})$. Hence, for any $N \geq \bar{N}$,
  \begin{align*}
    \widetilde{F}_{N}(\boldsymbol{\gamma}_{i}^{N})
    \overset{\eqref{ineq: F-Fn_unif}}{\;\leq\;}
    \widetilde{F}_{\beta, q}(\boldsymbol{\gamma}_{i}^{N})
    + \frac{1}{2N}
    + \frac{q-1}{N \beta}\, \ln\Bigl( \frac{N+q-1}{q-1}\Bigr)
    \;\leq\;
    c(\boldsymbol{x}, \boldsymbol{y})
    + \frac{1}{2N}
    + \frac{q-1}{N \beta}\, \ln N.
  \end{align*}
  Moreover, for each $i \in \{0, \ldots, K_{N}\}$ such that $\boldsymbol{\gamma}_{i}^{N} \not\in \boldsymbol{V}$ there exists $t_{i} \in [0, 1]$ with $\min_{t \in [0, 1]} \|\boldsymbol{\gamma}_{i}^{N} - \Gamma^{N}(t)\|_{2} = \|\boldsymbol{\gamma}_{i}^{N} - \Gamma^{N}(t_{i})\|_{2}$. In particular, for any $N \geq \bar{N}$, we have that both $\boldsymbol{\gamma}_{i}^{N}$ and $\Gamma^{N}(t_{i})$ lie in $\mathcal{P}^{\varepsilon/2}$. Hence, by \eqref{eq:Lcont+unif},
  \begin{align*}
    \widetilde{F}_{N}(\boldsymbol{\gamma}_{i}^{N})
    &\;\leq\;
    \widetilde{F}_{\beta, q}(\Gamma^{N}(t_{i}))
    +
    \bigl|
      \widetilde{F}_{\beta, q}(\boldsymbol{\gamma}_{i}^{N})
      - \widetilde{F}_{\beta, q}(\Gamma^{N}(t_{i}))
    \bigr|
    +
    \bigl|
      \widetilde{F}_{N}(\boldsymbol{\gamma}_{i}^{N})
      - \widetilde{F}_{\beta, q}(\boldsymbol{\gamma}_{i}^{N})
    \bigr|
    \\[.5ex]
    &\;\leq\;
    c(\boldsymbol{x}, \boldsymbol{y})
    + K_{\varepsilon, 1} \|\boldsymbol{\gamma}^{N}_{i} - \Gamma^{N}(t)\|_{2}
    + \frac{K_{\varepsilon, 2}}{N}
    \\
    &\;\leq\;
    c(\boldsymbol{x}, \boldsymbol{y})
    + \frac{\sqrt{q} K_{\varepsilon, 1} + K_{\varepsilon, 2}}{N}.
  \end{align*}
  Therefore, by setting $c_{\beta, \varepsilon}' \ldef \me^{-\beta(\sqrt{q} K_{\varepsilon, 1} + K_{\varepsilon, 2})}$ and using the estimates above, we obtain for all $N \geq \bar{N}$,
  \begin{align}\label{eq:measure:xN:yN:3}
    \widetilde{Q}_{N}(\boldsymbol{\gamma}_{i}^{N})
    \;\geq\;
    \min\bigl\{
      c_{\beta, \varepsilon}',
      \me^{-\beta/2} N^{-(q-1)}
    \bigr\}\,
    \frac{
      \exp\bigl(-\beta N c(\boldsymbol{x}, \boldsymbol{y})\bigr)
    }{\widetilde{\boldsymbol{Z}}_{N}}
    \qquad
    \forall\, i \in \{0, \ldots, K_{N}\}.
  \end{align}
  Substituting \eqref{eq:measure:xN:yN:3} into \eqref{eq:cap:xN:yN:lb} and using the facts that $\widetilde{p}_{N}(\boldsymbol{\gamma}_{i}^{N}, \boldsymbol{\gamma}_{i+1}^{N}) \geq \me^{-\beta}/(N q)$ and $K_{N} \leq N^{q-1}$, it follows that for every $N \geq \bar{N}$,
  \begin{align}\label{eq:cap:xN:yN:lb:final}
    \capacitym_{N}(\boldsymbol{x}_{N}, \boldsymbol{y}_{N})
    \;\geq\;
    \min\bigl\{
      c_{\beta, \varepsilon}',
      \me^{-\beta/2} N^{-(q-1)}
    \bigr\}\,
    \frac{\me^{-\beta}}{q}\, N^{-q}\,
    \frac{
      \exp\bigl(-\beta N c(\boldsymbol{x}, \boldsymbol{y})\bigr)
    }{\widetilde{\boldsymbol{Z}}_{N}}.
  \end{align}
  
  Finally, in view of \eqref{ineq: F-Fn_unif}, we have that for any $N \geq \bar{N}$,
  \begin{align*}
    \me^{-\beta/2}\, N^{-(q-1)}\,
    \frac{
      \exp\bigl(
        -\beta N \widetilde{F}_{\beta, q}(\boldsymbol{x}_{N})
      \bigr)
    }{\widetilde{\boldsymbol{Z}}_{N}}
    \;\leq\;
    \widetilde{Q}_{N}(\boldsymbol{x}_{N})
    \;\leq\;
    (2 \pi N)^{(q-1)/2}\,
    \frac{
      \exp\bigl(
        -\beta N \widetilde{F}_{\beta, q}(\boldsymbol{x}_{N})
      \bigr)
    }{\widetilde{\boldsymbol{Z}}_{N}}.
  \end{align*}
  By combining the above estimate with \eqref{eq:cap:xN:yN:ub:final} and \eqref{eq:cap:xN:yN:lb:final} the assertion \eqref{upper_cap_final} follows.
  \smallskip
  
  \textit{(ii)}
  Set $\bar{N} \equiv \bar{N}(\varepsilon) \ldef \max\{1/(\me^{-\beta} - \varepsilon), 2\sqrt{q}/\varepsilon\}$. For $N \geq \bar{N}$ fix some $\boldsymbol{z} \in \mathcal{P}_{N}$ and let $\boldsymbol{M}_{\!N} \ldef \{\boldsymbol{m}_{i, N} : i \in \mathcal{I}_{\beta}\}$ be the set of nearest lattice approximation of the different local minima of $\widetilde{F}_{\beta, q}$. Again, the proof of a lower bound on $\capacitym_{N}(\boldsymbol{z}, \boldsymbol{M}_{\!N})$ uses Rayleigh's shortcut method. However, due to the divergence of $\nabla \widetilde{F}_{\beta, q}$ at the boundary of $\mathcal{P}$, extra care is required when constructing a vertex-disjoint lattice path which starts at $\boldsymbol{z}$ and ends in $\boldsymbol{M}_{N}$. 
  Therefore, we treat the cases $\boldsymbol{z} \in \mathcal{P}^{\varepsilon}$ and $\boldsymbol{z} \in \mathcal{P} \setminus \mathcal{P}^{\varepsilon}$ separately.
  
  For $\boldsymbol{z} \in \mathcal{P}^{\varepsilon}$ let $\Gamma\colon [0, 1] \to \operatorname{int}(\mathcal{P})$ be the time-reparametrised solution of the initial-value problem $\dot{x}(t) = -\nabla\widetilde{F}_{\beta, q}(x(t))$ with $x(0) = \boldsymbol{z}$, chosen so that $\Gamma(1) \in \{\boldsymbol{m}_{i} : i \in \mathcal{I}_{\beta}\}$. Notice that for any $\ell \in \{1, \ldots, q\}$ and $\boldsymbol{x} \in \partial \mathcal{P}^{\varepsilon}_{\ell}$,
  \begin{align*}
    \nabla \widetilde{F}_{\beta, q}(\boldsymbol{x}) \cdot \boldsymbol{v}^{\ell}
    &\;=\;
    -\boldsymbol{x}_{l} + \frac{1}{q}
    + \frac{1}{q \beta} \ln\biggl( \frac{(\boldsymbol{x}_{l})^{q}}{\prod_{i=1}^{q} \boldsymbol{x}_{i}} \biggr)
    \\[.5ex]
    &\;\leq\;
    -\frac{1}{q\beta}\ln\biggl( \frac{1}{\varepsilon} - q + 1 \biggr)
    + \frac{1}{q} - \varepsilon
    \;\leq\;
    -\varepsilon
    \;<\;
    0,
  \end{align*}
  where the first inequality follows from the fact that the resulting expression is maximised when all coordinates except one equal $\varepsilon$, and the second uses $\varepsilon < \varepsilon_{0}$. Hence, the vector field points strictly into $\mathcal{P}^{\varepsilon}$, so $\Gamma(t) \in \mathcal{P}^{\varepsilon}$ for all $t \in [0, 1]$.
  
  Let $\boldsymbol{\gamma}^{N} = (\boldsymbol{\gamma}_{0}^{N}, \ldots, \boldsymbol{\gamma}_{K_{N}}^{N})$ be a vertex-disjoint, lattice approximation of $\Gamma$ with $\boldsymbol{\gamma}_{0}^{N} = \boldsymbol{z}$, $\boldsymbol{\gamma}_{K_{N}}^{N} \in \boldsymbol{M}_{\!N}$ such that $\boldsymbol{\gamma}_{i}^{N} \in \mathcal{P}_{N}$ and $\min_{t \in [0, 1]} \|\boldsymbol{\gamma}_{i}^{N} - \Gamma(t)\|_{2} \leq \sqrt{q} / N$ for all $i \in \{0, \ldots, K_{N}\}$, and $\widetilde{p}_{N}(\boldsymbol{\gamma}_{i}^{N}, \boldsymbol{\gamma}_{i+1}^{N}) > 0$ for every $i \in \{0, \ldots, K_{N}-1\}$. Then, by applying Rayleigh's shortcut method,
  \begin{align*}
    \frac{
      \capacitym_{N}(\boldsymbol{z}, \boldsymbol{M}_{N})
    }{\widetilde{Q}_{N}(\boldsymbol{z})}
    \overset{\eqref{eq:rayleigh_def}}{\;\geq\;}
    \Biggl(
      \sum_{i=0}^{K_{N}-1}
      \frac{
        \widetilde{Q}_{N}(\boldsymbol{z})
      }
      {%
        \widetilde{Q}_{N}(\boldsymbol{\gamma}^{N}_{i})
        \widetilde{p}_{N}(\boldsymbol{\gamma}^{N}_{i}, \boldsymbol{\gamma}^{N}_{i+1})
      }
    \Biggr)^{\!\!-1}.
  \end{align*}
  By construction, $\Gamma(t) \in \mathcal{P}^{\varepsilon}$ for every $t \in [0, 1]$ and, for each $i \in \{0, \ldots, K_{N}\}$, we can find $t_{i} \in [0, 1]$ so that $\min_{t \in [0, 1]} \|\boldsymbol{\gamma}_{i}^{N} - \Gamma(t)\|_{2} = \|\boldsymbol{\gamma}_{i}^{N} - \Gamma(t_{i})\|_{2}$. This implies that, for any $N \geq \bar{N}$, the lattice path $\boldsymbol{\gamma}^{N}$ lies in $\mathcal{P}^{\varepsilon/2}$. Therefore, by applying \eqref{eq:Lcont+unif} and using the fact that  $\widetilde{F}_{\beta, q}(\boldsymbol{z}) = \widetilde{F}_{\beta, q}(\Gamma(t_{0})) \geq \widetilde{F}_{\beta, q}(\Gamma(t_{i}))$ for any $i \in \{0, \ldots, K_{N}\}$,
  \begin{align*}
    \widetilde{F}_{N}(\boldsymbol{z}) - \widetilde{F}_{N}(\boldsymbol{\gamma}_{i}^{N})
    &\;\geq\;
    \widetilde{F}_{\beta, q}(\boldsymbol{z}) - \widetilde{F}_{\beta, q}(\Gamma(t_{i}))
    -
    \bigl|
      \widetilde{F}_{\beta, q}(\Gamma(t_{i})) - \widetilde{F}_{\beta, q}(\boldsymbol{\gamma}_{i}^{N})
    \bigr|
    \\
    &\mspace{36mu}-
    \bigl|
      \widetilde{F}_{N}(\boldsymbol{z}) - \widetilde{F}_{\beta, q}(\boldsymbol{z})
    \bigr|
    -
    \bigl|
      \widetilde{F}_{\beta, q}(\boldsymbol{\gamma}_{i}^{N}) - \widetilde{F}_{N}(\boldsymbol{\gamma}_{i}^{N})
    \bigr|
    \\[.5ex]
    &\;\geq\;
    -\frac{\sqrt{q} K_{\varepsilon, 1} + 2 K_{\varepsilon, 2}}{N}.
  \end{align*}
  Hence,
  \begin{align}
    \frac{
      \capacitym_{N}(\boldsymbol{z}, \boldsymbol{M}_{N})
    }{\widetilde{Q}_{N}(\boldsymbol{z})}
    \;\geq\;
    \me^{-\beta(\sqrt{q} K_{\varepsilon, 1} + 2 K_{\varepsilon, 2})}\,
    \frac{\me^{-\beta}}{q}\, N^{-q},
  \end{align}
  where we again used that $\widetilde{p}_{N}(\boldsymbol{\gamma}_{i}^{N}, \boldsymbol{\gamma}_{i+1}^{N}) \geq \me^{-\beta} / (q N)$ and $K_{N} \leq N^{q-1}$.

  It remains to consider the case that $\boldsymbol{z} \in \mathcal{P} \setminus \mathcal{P}^{\varepsilon}$. We first construct explicitly a vertex-disjoint lattice path $(\boldsymbol{\gamma}_{0}^{N}, \ldots, \boldsymbol{\gamma}_{k_{N}}^{N})$ from $\boldsymbol{z}$ to some $\boldsymbol{z}^{\ast} \in \mathcal{P}^{\varepsilon} \cap \mathcal{P}_{N}$ along which $\widetilde{F}_{N}$ is non-increasing, see Figure \ref{fig:level_sets} (Right). The remaining lattice path $(\boldsymbol{\gamma}_{k_{N}+1}^{N}, \ldots, \boldsymbol{\gamma}_{K_{N}}^{N})$ from $\boldsymbol{z}^{\ast}$ to $\boldsymbol{M}_{N}$ is then obtained, as above, as a vertex-disjoint lattice approximation of the time-reparametrised solution $\Gamma\colon [0, 1] \to \mathcal{P}^{\varepsilon}$ of the initial-value problem $\dot{x}(t) = -\nabla\widetilde{F}_{\beta, q}(x(t))$ with $x(0) = \boldsymbol{z}^{\ast}$, chosen so that $\Gamma(1) \in \{\boldsymbol{m}_{i} : i \in \mathcal{I}_{\beta}\}$. We are now going to construct $(\boldsymbol{\gamma}_{0}^{N}, \ldots, \boldsymbol{\gamma}_{k_{N}}^{N})$ inductively. Set $\boldsymbol{\gamma}_{0}^{N} \ldef \boldsymbol{z}$, and suppose that for some $i \in \mathbb{N}$ the vertex-disjoint path $(\boldsymbol{\gamma}_{0}^{N}, \ldots, \boldsymbol{\gamma}_{i}^{N})$ has already been constructed such that $\boldsymbol{\gamma}_{j}^{N} \in \mathcal{P}_{N} \setminus \mathcal{P}^{\varepsilon}$ for all $j \in \{0, \ldots, i\}$, $\widetilde{p}_{N}(\boldsymbol{\gamma}_{j}^{N}, \boldsymbol{\gamma}_{j+1}^{N}) > 0$ and $\widetilde{F}_{N}(\boldsymbol{\gamma}_{j}^{N}) \geq \widetilde{F}_{N}(\boldsymbol{\gamma}_{j+1}^{N})$ for every $j \in \{0, \ldots, i-1\}$. By the pigeonhole principle and the choice $\varepsilon < \varepsilon_{0}$, it follows that, for any $N \geq \bar{N}$, there exist $k, \ell \in \{1, \ldots, q\}$ such that $(\boldsymbol{\gamma}_{i}^{N})_{k} < \varepsilon$ and $(\boldsymbol{\gamma}_{i}^{N})_{\ell} \geq \varepsilon + 1/N$. Set $\boldsymbol{\gamma}_{i+1}^{N} \ldef \boldsymbol{\gamma}_{i}^{N} + \hat{e}_{k} - \hat{e}_{\ell} \in \mathcal{P}_{N}$. By construction, the extended lattice path $(\boldsymbol{\gamma}_{0}^{N}, \ldots, \boldsymbol{\gamma}_{i+1}^{N})$ remains vertex-disjoint and, in view of \eqref{eq:meso_rates}, we have $\widetilde{p}_{N}(\boldsymbol{\gamma}_{i}^{N}, \boldsymbol{\gamma}_{i+1}^{N}) > 0$. Moreover,
  \begin{align*}
    \widetilde{F}_{N}(\boldsymbol{\gamma}_{i}^{N})
    - \widetilde{F}_{N}(\boldsymbol{\gamma}_{i+1}^{N})
    \;=\;
    \frac{1}{N}\,
    \Bigl(
      f\bigl( (\boldsymbol{\gamma}_{i}^{N})_{k} + 1/N \bigr)
      - f\bigl( (\boldsymbol{\gamma}_{i}^{N})_{\ell} \bigr)
    \Bigr),
  \end{align*}
  where $f\colon (0, 1] \to \mathbb{R}$ is the function $x \mapsto f(x) = x - \beta^{-1} \ln x$. Since $f$ attains its global minimum at $\beta^{-1}$ and for any $N \geq \bar{N}$ one has $\varepsilon + 1/N \leq \me^{-\beta} < \beta^{-1}$, using also that \(\bar{N} \geq 1/(\me^{-\beta} - \varepsilon)\), it follows that
  \begin{align*}
    f\bigl( (\boldsymbol{\gamma}_{i}^{N})_{k} + 1/N \bigr)
    \;\geq\;
    f\bigl( \varepsilon + 1/N \bigr)
    \;=\;
    \max_{x \in [\varepsilon + 1/N, 1]} f(x)
    \;\geq\;
    f\bigl((\boldsymbol{\gamma}_{i}^{N})_{\ell}\bigr).
  \end{align*}
  Hence, $\widetilde{F}_{N}(\boldsymbol{\gamma}_{i}^{N}) \geq \widetilde{F}_{N}(\boldsymbol{\gamma}_{i+1}^{N})$. If $\boldsymbol{\gamma}_{i+1}^{N} \in \mathcal{P} \setminus \mathcal{P}^{\varepsilon}$, repeat the construction. Otherwise, set $k_{N} \ldef i+1$ and $\boldsymbol{z}^{*} \ldef \boldsymbol{\gamma}_{i+1}^{N}$.
  
  Consequently, by using that $\widetilde{F}_{N}(\boldsymbol{z}) - \widetilde{F}_{N}(\boldsymbol{\gamma}_{i}^{N}) \geq -(\sqrt{q} K_{\varepsilon, 1} + 2 K_{\varepsilon, 2}) / N$ for all $i \in \{0, \ldots, K_{N}\}$, it follows from Rayleigh's shortcut method,
  \begin{align}
    \frac{%
      \capacitym_{N}(\boldsymbol{z}, \boldsymbol{M}_{N})
    }{\widetilde{Q}_{N}(\boldsymbol{z})}
    \;\geq\;
    \Biggl(
      \sum_{i=0}^{K_{N}-1}
      \frac{
        \widetilde{Q}_{N}(\boldsymbol{z})
      }
      {%
        \widetilde{Q}_{N}(\boldsymbol{\gamma}^{N}_{i})
        \widetilde{p}_{N}(\boldsymbol{\gamma}^{N}_{i}, \boldsymbol{\gamma}^{N}_{i+1})
      }
    \Biggr)^{\!\!-1}
    \;\geq\;
    \me^{-\beta(\sqrt{q} K_{\varepsilon, 1} + 2 K_{\varepsilon, 2})}\,
    \frac{\me^{-\beta}}{q}\, N^{-q}.
  \end{align}
  This completes the proof of \eqref{lower_cap_final}.
\end{proof}
\begin{proof}[Proof of Theorem~\ref{thm:meta:cwp}]
  Recall that $\mathcal{M}_{N} \ldef \bigcup_{i \in \mathcal{I}_{\beta}} \mathcal{M}_{i, N}$, where, for each $i \in \mathcal{I}_{\beta}$, the subset $\mathcal{M}_{i, N} \subset \mathcal{S}_{N}$ (defined in \eqref{eq:metaset}) is the preimage under $L_{N}$ of $\boldsymbol{m}_{i, N}$. To verify that the Markov chain $(\widetilde{\Sigma}^{N}(t))_{t \geq 0}$ is $\varrho_{N}$-metastable with respect to the metastable set $\{\mathcal{M}_{i, N} : i \in \mathcal{I}_{\beta}\}$ for a suitably chosen decreasing sequence $(\varrho_{N})_{N \in \mathbb{N}} \subset (0, \infty)$, it suffices to estimate the numerator in \eqref{eq:rhometa} from above and denominator from below.

  First, by sub-additivity, we obtain for any $j \in \mathcal{I}_{\beta}$,
  \begin{align}
    &\widetilde{\Prob}^{N}_{\tilde{\mu}_{N} \rvert \mathcal{M}_{j,N}}\bigl[
      \tilde{\tau}^{N}_{\mathcal{M}_{N} \setminus \mathcal{M}_{j,N}} < \tilde{\tau}^{N}_{\mathcal{M}_{j,N}}
    \bigr]
    \nonumber\\[.5ex]
    &\mspace{36mu}\leq\;
    \sum_{\substack{i \in \mathcal{I}_{\beta} \\ i \neq j}}
    \widetilde{\Prob}^{N}_{\tilde{\mu}_{N} \rvert \mathcal{M}_{j,N}}\bigl[
      \tilde{\tau}^{N}_{\mathcal{M}_{i,N}} < \tilde{\tau}^{N}_{\mathcal{M}_{j,N}}
    \bigr]
    \overset{\eqref{eq:escape_prob}}{\;\leq\;}
    |\mathcal{I}_{\beta}|\,
    \max_{\substack{i, j \in \mathcal{I}_{\beta} \\ i \neq j}}
    \frac{%
      \widetilde{\capacity}_{N}(\mathcal{M}_{j, N}, \mathcal{M}_{i, N})
    }{\tilde{\mu}_{N}\bigl[\mathcal{M}_{j, N}\bigr]}.
  \end{align}
  Using the lumpability of $(\widetilde{\Sigma}^{N}(t))_{t \geq 0}$, we further obtain that for every distinct $i, j \in \mathcal{I}_{\beta}$,
  \begin{align}
    \frac{\widetilde{\capacity}_{N}(\mathcal{M}_{i, N}, \mathcal{M}_{j, N})}{\tilde{\mu}_{N}[\mathcal{M}_{j, N}]}
    \overset{\eqref{mesocap}}{\;=\;}
    \frac{\capacitym_{N}(\boldsymbol{m}_{i, N}, \boldsymbol{m}_{j, N})}{\widetilde{Q}_{N}(\boldsymbol{m}_{j, N})}.
  \end{align}
  Moreover, by construction, for every $i \in \mathcal{I}_{\beta}$ the sequence $(\boldsymbol{m}_{i, N})_{N \in \mathbb{N}}$ converges to the global minima, $\boldsymbol{m}_{i}$, of the limiting free energy $\widetilde{F}_{\beta, q}$. Since the global minima $\boldsymbol{m}_{i}$ and $\boldsymbol{m}_{j}$, as described in Proposition~\ref{prop:eland}, are separated by the essential gate $\mathcal{G}(\boldsymbol{m}_{i}, \boldsymbol{m}_{j})$ with communication height $c(\boldsymbol{m}_{i}, \boldsymbol{m}_{j})$, an application of \eqref{upper_cap_final} yields for all $N$ large enough
  \begin{align}\label{eq:upper_cap_prefinal}
    N^{-\ell_{1}}
    \;\leq\;
    \frac{%
      \capacitym_{N}(\boldsymbol{m}_{i, N}, \boldsymbol{m}_{j, N})
    }{\widetilde{Q}_{N}(\boldsymbol{m}_{j, N})}\,
    \me^{\beta N (c(\boldsymbol{m}_{i}, \boldsymbol{m}_{j}) - \widetilde{F}_{\beta, q}(\boldsymbol{m}_{j}))}
    \;\leq\;
    N^{\ell_{2}}.
  \end{align}
  Hence, by combining the estimates above, we obtain that, for every sufficiently large $N$
  \begin{align}\label{eq:def:meta:ub}
    \widetilde{\Prob}^{N}_{\tilde{\mu}_{N} \rvert \mathcal{M}_{j,N}}\bigl[
      \tilde{\tau}^{N}_{\mathcal{M}_{N} \setminus \mathcal{M}_{j,N}}
      < \tilde{\tau}^{N}_{\mathcal{M}_{j,N}}
    \bigr]
    \;\leq\;
    |\mathcal{I}_{\beta}|\, N^{\ell_{2}}
    \max_{\substack{i, j \in \mathcal{I}_{\beta} \\ i \neq j}}
    \me^{-\beta N (c(\boldsymbol{m}_{i}, \boldsymbol{m}_{j}) - \widetilde{F}_{\beta, q}(\boldsymbol{m}_{j}))}
  \end{align}

  Next, we establish a lower bound for the denominator in \eqref{eq:rhometa}. For any $\mathcal{X} \subset \mathcal{S}_{N} \setminus \mathcal{M}_{N}$ there exists $\{\boldsymbol{x}_{k} : k = 1, \ldots, K\} \subset \mathcal{P}_{N}$ such that
  \begin{align*}
    \mathcal{X} \cap L_{N}^{-1}(\boldsymbol{x}_{k}) \;\neq\; \emptyset,
    \quad \forall\, k \in \{1, \ldots, K\}
    \qquad \text{and} \qquad
    \mathcal{X} \;\subset\; \bigcup_{k=1}^{K} L_{N}^{-1}(\boldsymbol{x}_{k}).
  \end{align*}
  In the sequel, set $\boldsymbol{M}_{\!N} \ldef \{\boldsymbol{m}_{i} : i \in \mathcal{I}_{\beta}\}$, and write $\mathcal{X}_{k} \ldef L_{N}^{-1}(\boldsymbol{x}_{k})$ for any $k \in \{1, \ldots, K\}$. Then,
  \begin{align}\label{eq:def:meta:lb1}
    \widetilde{\Prob}^{N}_{\tilde{\mu}_{N} \rvert \mathcal{X}}\bigl[
      \tilde{\tau}^{N}_{\mathcal{M}_{N}} < \tilde{\tau}^{N}_{\mathcal{X}}
    \bigr]
    \overset{\eqref{eq:escape_prob}}{\;=\;}
    \frac{%
      \widetilde{\capacity}_{N}(\mathcal{X}, \mathcal{M}_{N})
    }{\tilde{\mu}_{N}[\mathcal{X}]}
    \overset{\!\eqref{eq:monotonic}\!}{\;\geq\;}
    \frac{1}{K} \sum_{k=1}^{K}
    \frac{
      \widetilde{\capacity}(\mathcal{X} \cap \mathcal{X}_{k}, \mathcal{M}_{N})
    }{\tilde{\mu}_{N}[\mathcal{X}]}.
  \end{align}
  Since for any $k \in \{1, \ldots, K\}$,
  \begin{align*}
    \widetilde{\capacity}(\mathcal{X} \cap \mathcal{X}_{k}, \mathcal{M}_{N})
    &\overset{\eqref{eq:cap_def}}{\;=\;}
    \mspace{-6mu}\sum_{\sigma \in \mathcal{X} \cap \mathcal{X}_{k}}\mspace{-6mu} \tilde{\mu}_{N}(\sigma)\,
    \widetilde{\Prob}_{\!\sigma}^{N}\bigl[
      \tilde{\tau}_{\mathcal{M}_{N}}^{N} < \tilde{\tau}_{\mathcal{X} \cap \mathcal{X}_{k}}
    \bigr]
    \;\geq\;
    \mspace{-6mu}\sum_{\sigma \in \mathcal{X} \cap \mathcal{X}_{k}}\mspace{-6mu} \tilde{\mu}_{N}(\sigma)\,
    \widetilde{\Prob}_{\!\sigma}^{N}\bigl[
      \tilde{\tau}_{\mathcal{M}_{N}}^{N} < \tilde{\tau}_{\mathcal{X}_{k}}
    \bigr]
  \end{align*}
  and, in view of \eqref{eq:lumpprop}, $\widetilde{\Prob}_{\sigma}^{N}[\mathcal{X} \cap \mathcal{X}_{k}] = \widetilde{\Prob}^{N}_{\tilde{\mu}_{N} \rvert \mathcal{X}_{k}}[\mathcal{X} \cap \mathcal{X}_{k}]$ for any $\sigma \in \mathcal{X}_{k}$, we obtain that
  \begin{align}\label{eq:def:meta:lb2}
    \widetilde{\capacity}(\mathcal{X} \cap \mathcal{X}_{k}, \mathcal{M}_{N})
    &\;\geq\;
    \tilde{\mu}_{N}\bigl[\mathcal{X} \cap \mathcal{X}_{k}\bigr]\,
    \widetilde{\Prob}^{N}_{\tilde{\mu}_{N} \rvert \mathcal{X}_{k}}\bigl[
      \tilde{\tau}^{N}_{\mathcal{M}_{N}} < \tilde{\tau}^{N}_{\mathcal{X}_{k}}
    \bigr]
    \nonumber\\[.5ex]
    &\;=\;
    \tilde{\mu}_{N}\bigl[\mathcal{X} \cap \mathcal{X}_{k}\bigr]\,
    \frac{%
      \widetilde{\capacity}_{N}(\mathcal{X}_{k}, \mathcal{M}_{\!N})
    }{\tilde{\mu}_{N}[\mathcal{X}_{k}]}
    \;=\;
    \tilde{\mu}_{N}\bigl[\mathcal{X} \cap \mathcal{X}_{k}\bigr]\,
    \frac{%
      \capacitym_{N}(\boldsymbol{x}_{k}, \boldsymbol{M}_{N})
    }{\widetilde{Q}_{N}(\boldsymbol{x}_{k})}.
  \end{align}
  Thus, by combining \eqref{eq:def:meta:lb1} and \eqref{eq:def:meta:lb2} and applying Proposition~\ref{prop:preem_cap_estimates}, we obtain, for any sufficiently large enough $N$,
  \begin{align}\label{eq:def:meta:lb}
    \widetilde{\Prob}^{N}_{\tilde{\mu}_{N} \rvert \mathcal{X}}\bigl[
      \tilde{\tau}^{N}_{\mathcal{M}_{N}} < \tilde{\tau}^{N}_{\mathcal{X}}
    \bigr]
    \;\geq\;
    \frac{1}{K}\,
    \min_{k \in \{1, \ldots, K\}}
    \frac{%
      \capacitym_{N}(\boldsymbol{x}_{k}, \boldsymbol{M}_{\!N})
    }{\widetilde{Q}_{N}(\boldsymbol{x}_{k})}
    \overset{\eqref{lower_cap_final}}{\;\geq\;}
    N^{-(\ell_{3}+q)},
  \end{align}
  where we used that $K \leq |\mathcal{P}_{N}| \leq N^{q}$. Finally, set $k_{1} \ldef \beta \min_{i \neq j \in \mathcal{I}_{\beta}} (c(\boldsymbol{m}_{i}, \boldsymbol{m}_{j}) - \widetilde{F}_{\beta, q}(\boldsymbol{m}_{j}))$. By combining \eqref{eq:def:meta:ub} and \eqref{eq:def:meta:lb}, we obtain that for any $k \in (0, k_{1})$ there exists $N_{0} \in \mathbb{N}$ such that for any $N \geq N_{0}$
  \begin{align*}
    |\mathcal{I}_{\beta}|\,
    \frac{
      \max_{j \in \mathcal{I}_{\beta}}
      \widetilde{\Prob}^{N}_{\tilde{\mu}_{N} \rvert \mathcal{M}_{j, N}}\bigl[
        \tilde{\tau}^{N}_{\mathcal{M}_{N} \setminus \mathcal{M}_{j, N}} < \tilde{\tau}^{N}_{\mathcal{M}_{j, N}}
      \bigr]
    }{
      \min_{\mathcal{X} \subset \mathcal{S}_{N} \setminus \mathcal{M}_{N}}
      \widetilde{\Prob}^{N}_{\tilde{\mu}_{N} \rvert \mathcal{X}}\bigl[
        \tilde{\tau}^{N}_{\mathcal{M}_{N}} < \tilde{\tau}^{N}_{\mathcal{X}}
      \bigr]
    }
    \;\leq\;
    |\mathcal{I}_{\beta}|^{2}\, N^{\ell_{2} + \ell_{3} + q}\, \me^{-\beta k_{1} N}
    \;\leq\;
    \me^{-k N},
  \end{align*}
  which concludes the proof.
\end{proof}
\begin{remark}\label{rem:optimal:k1:CWP}
  Notice that, for any $j \in \mathcal{I}_{\beta}$,
  \begin{align*}
    \widetilde{\Prob}^{N}_{\tilde{\mu}_{N} \rvert \mathcal{M}_{j, N}}\bigl[
      \tilde{\tau}^{N}_{\mathcal{M}_{N} \setminus \mathcal{M}_{j, N}} < \tilde{\tau}^{N}_{\mathcal{M}_{j, N}}
    \bigr]
    \overset{\eqref{eq:monotonic}}{\;\geq\;}
    \max_{\substack{i, j \in \mathcal{I}_{\beta} \\ i \neq j}}
    \widetilde{\Prob}^{N}_{\tilde{\mu}_{N} \rvert \mathcal{M}_{j, N}}\bigl[
      \tilde{\tau}^{N}_{\mathcal{M}_{i, N} \setminus \mathcal{M}_{j, N}} < \tilde{\tau}^{N}_{\mathcal{M}_{j, N}}
    \bigr].
  \end{align*}
  Combining this estimate with \eqref{eq:escape_prob}, \eqref{mesocap} and \eqref{eq:upper_cap_prefinal}, and recalling the definition of $k_{1}$, it follows that for any $k > k_{1}$ there exists $N_{0}' \in \mathbb{N}$ such that for any $N \geq N_{0}'$,
  \begin{align}\label{eq:optimal:k1:CWP}
    |\mathcal{I}_{\beta}|\,
    \frac{
      \max_{j \in \mathcal{I}_{\beta}}
      \widetilde{\Prob}^{N}_{\tilde{\mu}_{N} \rvert \mathcal{M}_{j, N}}\bigl[
        \tilde{\tau}^{N}_{\mathcal{M}_{N} \setminus \mathcal{M}_{j, N}} < \tilde{\tau}^{N}_{\mathcal{M}_{j, N}}
      \bigr]
    }{
      \min_{\mathcal{X} \subset \mathcal{S}_{N} \setminus \mathcal{M}_{N}}
      \widetilde{\Prob}^{N}_{\tilde{\mu}_{N} \rvert \mathcal{X}}\bigl[
        \tilde{\tau}^{N}_{\mathcal{M}_{N}} < \tilde{\tau}^{N}_{\mathcal{X}}
      \bigr]
    }
    \;\geq\;
    \frac{|\mathcal{I}_{\beta}|}{N^{\ell_{1}}}\,
    \me^{-k_{1} N}
    \;\geq\;
    \me^{-k N},
  \end{align}
  which shows that, on the exponential scale, the constant $k_{1}$ is optimal.
\end{remark}

\section{Metastability for the DCWP model}\label{sec:meta-DCWP}
In this section we analyse the metastable behaviour of the DCWP model by comparison with the CWP model studied in Section~\ref{app}. In particular, our aim is to prove Theorem~\ref{thm:meta:dcwp}. Its proof depends on preliminary comparisons of several quantities of interest in the two models; these comparisons appear as Lemma~\ref{coro:betapmapprox} and Lemma~\ref{lemma:xi_comp} in Section~\ref{sec:prel}, while the proof of Theorem~\ref{thm:meta:dcwp} is presented in Section~\ref{sec:proofmeta_dcwp}.

\subsection{Preliminary comparison between the CWP and DCWP models}\label{sec:prel}
In order to simplify notation, define the following random variable
\begin{align}\label{eq:delta}
  \Delta_{N}(\sigma) \;\ldef\; H_{N}(\sigma) - \widetilde{H}_{N}(\sigma).
\end{align}
Further, let $\phi\colon \mathbb{R} \to [0, \infty]$ be the log-moment generating function of the centred random variable $J_{12}-1$, that is,
\begin{align}\label{eq:lmgf}
  \phi(t)
  \;=\;
  \ln \mean\bigl[\me^{t(J_{ij}-1)}\bigr].
\end{align}
The next lemma gives an expression of the annealed Gibbs density in terms of $\phi$. The strategy is inspired by the arguments given by the proof of \cite[Lemma 4.2]{BdHMPS24}.
\begin{lemma}\label{coro:betapmapprox}
  Let $\beta \geq 0$ and recall that $v = \mathbb{V}[J_{12}]$.
  \begin{enumerate}[(i)]
  \item
    For any $2 \leq N \in \mathbb{N}$, such that $\beta/N \in \mathscr{D}$, with $\mathscr{D}$ as in Assumption \ref{ass:law:J},
    \begin{align}
      \mean\Bigl[\me^{\pm\beta \Delta_{N}(\sigma)}\Bigr]
      \;=\;
      \me^{-N \phi(\pm\beta/N) \widetilde{H}_{N}(\sigma)}
      \qquad \forall\, \sigma \in \mathcal{S}_{N}.
    \end{align}
    
  \item
    For $N \to \infty$,
    \begin{align}\label{eq:mean_e_delta}
      \me^{\frac{\beta^2v}{4q}(1+o(1))}
      \;\leq\;
      \min_{\sigma \in \mathcal{S}_{N}} \mean\Bigl[ \me^{-\beta \Delta_{N}(\sigma)} \Bigr]
      \;\leq\;
      \max_{\sigma \in \mathcal{S}_{N}} \mean\Bigl[ \me^{-\beta \Delta_{N}(\sigma)} \Bigr]
      \;\leq\;
      \me^{\frac{\beta^2 v}{4}(1+o(1))}.
    \end{align}
    
  \item
    For $N \to \infty$, 
    \begin{align}\label{eq:here_to_pick_proof}
      \max_{\substack{\sigma, \eta \in \mathcal{S}_{N} \\ \sigma \sim \eta}}
      \frac{%
        \mean\bigl[\me^{\pm \beta (H_{N}(\sigma) \vee H_{N}(\eta))}\bigr]
      }{%
        \me^{\pm \beta (\widetilde{H}_{N}(\sigma) \vee \widetilde{H}_{N}(\eta))}
      }
      \;\leq\;
      \me^{\frac{\beta^2v}{4}}
      \bigl( 1 + o(1)\bigr). 
    \end{align}
  \end{enumerate}
\end{lemma}
\begin{proof}
  Throughout the proof set $\widetilde{M}(\sigma) \ldef \sum_{0 \leq i < j \leq N} \indicator_{ \sigma_{i} = \sigma_{j}} = -N \widetilde{H}_{N}(\sigma)$ for any $\sigma \in \mathcal{S}_{N}$.

  \textit{(i)} Since the random variables $(J_{ij})_{1 \leq i < j \leq N}$ are assumed to be independent and identically distributed,
  \begin{align}
    \label{eq:mean_Delta_estimate}
    \mean\Bigl[ \me^{-\beta \Delta_{N}(\sigma)} \Bigr]
    \;=\;
    \mean\Bigl[ \me^{\frac{\beta(J_{ij}-1)}{N}} \Bigr]^{\widetilde{M}(\sigma)}
    \;=\;
    \me^{\widetilde{M}(\sigma)\phi(\beta/N)},
  \end{align}
  where $\phi$ is the log–moment generating function defined in \eqref{eq:lmgf}.
  
  \textit{(ii)} By a Taylor expansion, we have that $\phi(x) = v x^{2}/2 + o(x^{2})$ as $x \to 0$. Since, for any $N > q$, there are at most \(N^2/2\) unordered pairs and, by the pigeonhole principle, at least \(N(N-1)/(2q)\) pairs of equal spin, if follows that $N(N-1)/(2q) \leq \widetilde{M}(\sigma) \leq N^{2} / 2$ for any $\sigma \in \mathcal{S}_{N}$. Hence, the assertion follows from \eqref{eq:mean_Delta_estimate}.
  
  \textit{(iii)} Since the expectation of the minimum of two random variables is upper bounded by the minimum of the expectations, we have
  \begin{align}
    \mean\Bigl[
      \me^{-\beta (H_{N}(\sigma) \vee H_{N}(\eta))}
    \Bigr]
    \;\leq\;
    \mean\Bigl[ \me^{-\beta H_{N}(\sigma)}\Bigr]
    \wedge \mean\Bigl[\me^{-\beta H_{N}(\eta)}\Bigr].
  \end{align}
  Thus, in view of (ii), we obtain the desired upper bound \eqref{eq:here_to_pick_proof}.
  
  We now treat the the other sign. For this, since $\sigma$ and $\eta$ differ at most in one spin we can find $k \in \{1, \dots, N\}$, $\ell \in \{1, \dots, q\}$ such that $\eta = \sigma^{k,\ell}$, where
  \begin{align*}
    \sigma^{k, \ell}_{i}
    \;=\;
    \begin{cases}
      \sigma_i, &\quad k \neq i
      \\
      \ell, &\quad k = i.
    \end{cases}
  \end{align*}
  We further introduce the decomposition
  \begin{align*}
    S^{k}(\sigma)
    \;=\;
    \frac{-1}{N}\sum_{\substack{i < j \\ i, j \neq k}}^{N} J_{ij}\indicator_{\sigma_{i} = \sigma_{j}},
    \qquad
    D^{k}(\sigma)
    \;=\;
    \frac{-1}{N}\sum_{\substack{j \neq k}}^{N} J_{kj}\indicator_{\sigma_{k} = \sigma_{j}}.
  \end{align*}
  One can observe that $H_{N} = S^{k} + D^{k}$, $S^{k}(\sigma) = S^{k}(\sigma^{k, \ell})$, and that $S^{k}(\sigma)$ is independent of $D^{k}(\eta)$ for any $\sigma,\eta$. Then we have
  \begin{align*}
    \mean\Bigl[
      \me^{\beta (H_{N}(\sigma) \vee H_{N}(\eta))}
    \Bigr]
    &\;=\;
    \mean\Bigl[
      \me^{\beta H_{N}(\sigma)} \vee \me^{\beta H_N(\eta)}
    \Bigr]
    \;=\;
    \mean\Bigl[\me^{\beta S^{k}(\sigma)}\Bigr]\,
    \mean\Bigl[
      \me^{\beta D^{k}(\sigma)} \vee \me^{\beta D^{k}(\sigma^{k,\ell})}
    \Bigr].
  \end{align*}
  Also note that for any positive random variables $X, Y$,
  \begin{align}\label{bound:mean_max}
    \mean\bigl[ X \vee Y \bigr]
    \;\leq\;
    \Bigl(\mean[X] \vee \mean[Y] \Bigr)
    \biggl(
      1 +
      \biggl(
        \frac{\sqrt{\var[X]}}{\mean[X]}
        +
        \frac{\sqrt{\var[Y]}}{\mean[Y]}
      \biggr)
    \biggr).
  \end{align}
  This leads to the following estimate
  \begin{align*}
    \frac{
      \mean\bigl[
        \me^{\beta (H_N(\sigma) \vee H_N(\eta))}\bigr]
    }{
      \mean\bigl[\me^{\beta H_{N}(\sigma)}\bigr]
      \vee
      \mean\bigl[\me^{\beta H_{N}(\sigma^{k, \ell})}\bigr]
    }
    \;\leq\;
    1
    +
    \frac{
      \sqrt{\var\bigl[ \me^{\beta D^k(\sigma)} \bigr]}
    }{
      \mean\bigl[\me^{\beta D^k(\sigma)}\bigr]
    }
    +
    \frac{
      \sqrt{\var\bigl[ \me^{\beta D^k(\sigma^{k,\ell})} \bigr]}
    }{
      \mean\bigl[ \me^{\beta D^k(\sigma^{k,\ell})} \bigr]
    }
  \end{align*}
  In order to control the right hand side, we first observe that
  \begin{align*}
    \mean\Bigl[ (\me^{\beta D^k(\eta)})^2 \Bigr]
    \;=\;
    \me^{\frac{-2\beta M(\eta)}{N} + M(\eta)\phi\left(\frac{-2\beta}{N}\right)}
    \qquad \text{and} \qquad
    \mean\Bigl[ \me^{\beta D^k(\eta)} \Bigr]^2
    \;=\;
    \me^{\frac{-2\beta M(\eta)}{N} + 2 M(\eta) \phi\left(\frac{-\beta}{N}\right)},
  \end{align*}
  where $M(\eta) = \sum_{j \neq k}^{N} \indicator_{\eta_{k} = \eta_{j}}$ with $\eta \in \{ \sigma, \sigma^{k, \ell} \}$, and $\phi$ is defined in \eqref{eq:lmgf}. Therefore, using the convexity of $\phi$ and bounding $M$ by $N$, we have
  \begin{align*}
    \frac{\var\bigl[ \me^{\beta D^k(\eta)} \bigr]}{\mean\bigl[ \me^{ \beta D^k(\eta)} \bigr]^2}
    \;=\;
    \me^{M(\eta) \left(\phi\left(\frac{-2\beta}{N}\right) - 2\phi\left(\frac{-\beta}{N}\right)\right)} - 1
    \;\leq\;
    \frac{v\beta^2}{N}(1+o(1)).
  \end{align*}
  Since $H_{N}(\sigma) = \widetilde{H}_{N}(\sigma) + \Delta_{N}(\sigma)$, an application of \eqref{eq:mean_e_delta} yields the desired bound. 
\end{proof}
In the next lemma we obtain both estimates for $\Delta_N(\sigma)$ in the form of concentration inequalities and estimates for the probability of the following event
\begin{align}\label{bound:squaren}
  \Xi_{N}(a)
  \;=\;
  \biggl\{
    \max_{\sigma \in \mathcal{S}_{N}} \abs{\Delta_N(\sigma)} \leq a\sqrt{N}
  \biggr\}, \qquad a>0.
\end{align}
In particular, $ \Xi_{N}(a)^{c}$ turns out to be negligible in the limit as $N \to \infty$.
\begin{lemma}\label{lemma:xi_comp}
  \begin{enumerate}[(i)]
  \item
    Let $(t_{N})_{N \in \mathbb{N}} \subset [0, \infty)$ be a sequence such that $t_{N}/N \to 0$ as $N \to \infty$. Then, for any $\sigma \in \mathcal{S}_{N}$, 
    \begin{align}\label{eq:concdelta}
      \prob\bigl[ \abs{\Delta_{N}(\sigma)} \geq t_{N} \bigr]
      \;\leq\;
      2\exp\biggl( -\frac{t_{N}^{2}}{v} \bigl( 1 + o(1) \bigr) \biggr).
    \end{align}
    
  \item
    For any $a > 0$
    \begin{align}\label{eq:xi_comp_prob}
      \prob\bigl[ \Xi_{N}(a)^c \bigr]
      \;\leq\;
      2 \exp\biggl( N \ln q - \frac{a^2 N}{v} \bigl( 1+o(1) \bigr) \biggr).
    \end{align}
  \end{enumerate}
\end{lemma}
\begin{proof}
  \textit{(i)} First note that $\prob\bigl[\abs{\Delta_{N}(\sigma)} \geq t \bigr] = \prob\bigl[\Delta_{N}(\sigma) \geq t \bigr] + \prob\bigl[\Delta_{N}(\sigma) \leq - t\bigr]$, so by symmetry it suffices to bound the first term. Fix $\lambda \in \mathbb{R}$ and choose $N$ large enough that $-\lambda/N$ lies in the domain of $\phi$ defined in \eqref{eq:lmgf}. Then, an application of Markov's inequality and Lemma~\ref{coro:betapmapprox} give
  \begin{align*}
    \prob\bigl[ \Delta_{N}(\sigma) \geq t\bigr]
    \;\leq\;
    \frac{\mean\bigl[ \me^{\lambda \Delta_{N}(\sigma)} \bigr]}{\me^{\lambda t_{N}}} 
    \;\leq\;
    \me^{\frac{\lambda^2v}{4}(1+o(1)) -\lambda t_{N}}.
  \end{align*}
  Optimising the exponent by taking $\lambda = 2t_{N}/v (1+o(1))$ yields the claimed bound.
  
  \textit{(ii)} A union bound yields
  \begin{align*}
    \prob\biggl[
      \max_{\sigma \in \mathcal{S}_{N}}
      \abs{\Delta_{N}(\sigma)} > a\sqrt{N}
    \biggr]
    &\;\leq\; 
    \sum_{\sigma \in \mathcal{S}_{N}}
    \prob\Bigl[ \abs{\Delta_{N}(\sigma)} > a \sqrt{N} \Bigr]
    \;\leq\;
    2\, \me^{N \ln(q) - N a^{2}/v (1 + o(1))},
  \end{align*}
  where the final inequality uses \eqref{eq:concdelta} and $\abs{\mathcal{S}_{N}} = q^{N}$.    
\end{proof}

\subsection{Proof of Theorem~\ref{thm:meta:dcwp}} \label{sec:proofmeta_dcwp}
In the previous section we provided a comparison of the quenched Hamiltonian $H_N$ and the annealed Hamiltonian $\widetilde{H}_{N}$. In the following proof we will proceed by comparing the quadratic forms of the quenched and annealed models, introduced in Section \ref{subsec:pot_theo}. This proof follows along the lines of \cite[Theorem 2.10]{BdHMPS24} with minor modifications. However, we present it here for the convenience of the reader.

\begin{proof}[Proof of Theorem~\ref{thm:meta:dcwp}]
  First note that for any two adjacent configurations $\sigma, \eta \in \mathcal{S}_{N}$,
  \begin{align*}
    Z_{N} \mu_{N}(\sigma) \pi_{N}(\sigma, \eta)
    \;=\;
    \frac{1}{Nq}\me^{-\beta(H_N(\sigma)\vee H_N(\eta))}.
  \end{align*}
  Hence, the quadratic form in \eqref{def:D_form_E} can be written as
  \begin{align*}
    Z_{N} \mathcal{E}_{N}(f)
    \;=\;
    \frac{1}{2} \sum_{\substack{\sigma,\eta\in \mathcal S_{N} \\ \sigma \sim \eta}}
    \frac{1}{Nq}\, \me^{-\beta(H_{N}(\sigma) \vee H_{N}(\eta))}
    \bigl( f(\sigma)-f(\eta) \bigr)^2.
  \end{align*}
  Let $a > 0$ and recall the definition of $\Xi_{N}(a)$ given in \eqref{bound:squaren}. Then, on the event $\Xi_{N}(a)$,
  \begin{align}\label{eq:xi_control_C}
    \me^{-\beta a \sqrt{N}}\, \widetilde{Z}_{N} \widetilde{\mathcal{E}}_{N}(f)
    \;\leq\;
    Z_{N} \mathcal{E}_{N}(f)
    \;\leq\;
    \me^{\beta a \sqrt{N}}\, \widetilde{Z}_{N} \widetilde{\mathcal{E}}_{N}(f)
  \end{align}
  for any $f\colon \mathcal{S}_{N} \to \mathbb{R}$. Hence, the Dirichlet principle \eqref{eq:Dirichlet} implies that, for any non-empty, disjoint $\mathcal{X}, \mathcal{Y} \subset \mathcal{S}_{N}$,
  \begin{align}\label{eq:xi_control_cap}
    \me^{-\beta a \sqrt{N}}\,
    \widetilde{Z}_{N} \widetilde{\capacity}_{N}(\mathcal{X}, \mathcal{Y})
    \;\leq\;
    Z_{N} \capacity_{N}(\mathcal{X}, \mathcal{Y})
    \;\leq\;
    \me^{\beta a \sqrt{N}}\,
    \widetilde{Z}_{N} \widetilde{\capacity}_{N}(\mathcal{X}, \mathcal{Y}).
  \end{align}
  Likewise, on the event $\Xi_{N}(a)$, we have that, for any $\mathcal{X} \subset \mathcal{S}_{N}$,
  \begin{align}
    \me^{-\beta a \sqrt{N}}\,
    \widetilde{Z}_{N} \widetilde{\mu}_{N}[\mathcal{X}]
    \;\leq\;
    Z_{N} \mu_{N}[\mathcal{X}]
    \;\leq\;
    \me^{\beta a \sqrt{N}}\,
    \widetilde{Z}_{N} \widetilde{\mu}_{N}[\mathcal{X}].
  \end{align}
  Thus, in view of \eqref{eq:escape_prob}, we deduce that, on the event $\Xi_{N}(a)$,
  \begin{align}\label{eq:xi_control_prob}
    \me^{-2\beta a \sqrt{N}}
    \;\leq\;
    \frac{\Prob^{N}_{\mu_{N} \rvert \mathcal{X}}\bigl[ \tau^{N}_{\mathcal{X}} < \tau^{N}_{\mathcal{Y}}\bigr]}
    {
      \widetilde{\Prob}^{N}_{\widetilde{\mu}_{N} \rvert \mathcal{X}}\bigl[
        \widetilde{\tau}^{N}_{\mathcal{X}} < \widetilde{\tau}^{N}_{\mathcal{Y}}
      \bigr]
    }
    \;\leq\;
    \me^{2\beta a \sqrt{N}}.
  \end{align}
  Next, choose $a > \sqrt{v \ln q}$. Fix any $k < k_{1}$ (with $k_{1}$ as in Theorem~\ref{thm:meta:cwp}) and any $k' \in (k, k_{1})$. Taking $N$ sufficiently large so that $k' - 2a\beta/\sqrt{N} \geq k$,  Theorem~\ref{thm:meta:cwp} together with \eqref{eq:xi_control_prob} yield that, on the event $\Xi_{N}(a)$,
  \begin{align}
    \label{boundformeta}
    &|\mathcal{I}_{\beta}|\,
    \frac{
      \max_{i \in \mathcal{I}_{\beta}} \Prob^{N}_{\mu_{N} \rvert \mathcal{M}_{i, N}}\bigl[
        \tau^{N}_{\mathcal{M}_{N} \setminus \mathcal{M}_{i, N}} < \tau^{N}_{\mathcal{M}_{i, N}}
      \bigr]
    }{
      \min_{\mathcal{X} \subset \mathcal{S}_{N} \setminus \mathcal{M}_{N}} \Prob^{N}_{\mu_{N} \rvert \mathcal{X}}\bigl[
        \tau^{N}_{\mathcal{M}_{N}} < \tau^{N}_{\mathcal{X}}
      \bigr]
    }
    \nonumber\\[.5ex]
    &\mspace{36mu}\leq\;
    |\mathcal{I}_{\beta}|\,\me^{2\beta a \sqrt{N}}\,
    \frac{
      \max_{i \in \mathcal{I}_{\beta}}
      \widetilde{\Prob}^{N}_{\widetilde{\mu}_{N} \rvert \mathcal{M}_{i, N}}\bigl[
        {\tau}^{N}_{\mathcal{M}_{N} \setminus \mathcal{M}_{i, N}}
        < {\tau}^{N}_{cM_{i, N}}
      \bigr]
    }{
      \min_{\mathcal{X} \subset \mathcal{S}_{N} \setminus \mathcal{M}_{N}}
      \widetilde{\Prob}^{N}_{\widetilde{\mu}_{N} \rvert \mathcal{X}}\bigl[
        {\tau}^{N}_{\mathcal{M}_{N}} < {\tau}^{N}_{\mathcal{X}}
      \bigr]
    }
    \;\leq\;
    \me^{-k' N + 2\beta a \sqrt{N}}
    \;\leq\;
    \me^{-k N},
  \end{align}
  implying that $\Xi_{N}(a) \subset \Omega_{\textnormal{meta}}(N)$. On the other hand, due to the above choice of $a$, the right-hand side of \eqref{eq:xi_comp_prob} is summable in $N$. Therefore, an application of the Borell-Cantelli lemma finally yields
  \begin{align}
    \prob\biggl[ \limsup_{N \to \infty} \Omega_{\textnormal{meta}}(N)^c \biggr]
    \;\leq\;
    \prob\biggl[ \limsup_{N \to \infty} \Xi_{N}(a)^c \biggr]
    \;=\;
    0.
  \end{align}
\end{proof}
\begin{remark}\label{rem:tightk}
  Notice that, for every $k \geq k_{1}$, there exists a random variable $N_{1}$ such that, $\prob$-almost surely, $N_{1} < \infty$ and for all $N \geq N_{1}$
  \begin{align}\label{eq:optimal:k1:DCWP}
    |\mathcal{I}_{\beta}|\,
    \frac{
      \max_{i \in \mathcal{I}_{\beta}} \Prob^{N}_{\mu_{N} \rvert \mathcal{M}_{i, N}}\bigl[
        \tau^{N}_{\mathcal{M}_{N} \setminus \mathcal{M}_{i, N}} < \tau^{N}_{\mathcal{M}_{i, N}}
      \bigr]
    }{
      \min_{\mathcal{X} \subset \mathcal{S}_{N} \setminus \mathcal{M}_{N}} \Prob^{N}_{\mu_{N} \rvert \mathcal{X}}\bigl[
        \tau^{N}_{\mathcal{M}_{N}} < \tau^{N}_{\mathcal{X}}
      \bigr]
    }
    \;\geq\;
    \me^{-k N}.
  \end{align}
  Indeed, since for any fixed $a > \sqrt{v \ln q}$ the right-hand side of \eqref{eq:xi_comp_prob} is summable in $N$, the Borell-Cantelli lemma implies that there exists a random variable $\tilde{N}_{1}$ such that, $\prob$-almost surely, $\tilde{N}_{1} < \infty$ and for every $N \geq \tilde{N}_{1}$
  \begin{align*}
    \max_{i \in \mathcal{I}_{\beta}}
    \Prob^{N}_{\mu_{N} \rvert \mathcal{M}_{i, N}}\bigl[
      \tau^{N}_{\mathcal{M}_{N} \setminus \mathcal{M}_{i, N}} < \tau^{N}_{\mathcal{M}_{i, N}}
    \bigr]
    &\;\geq\;
    \me^{-2 \beta a \sqrt{N}}\,
    \max_{i \in \mathcal{I}_{\beta}}
    \widetilde{\Prob}^{N}_{\tilde{\mu}_{N} \rvert \mathcal{M}_{i, N}}\bigl[
      \tilde{\tau}^{N}_{\mathcal{M}_{N} \setminus \mathcal{M}_{i, N}}
      < \tilde{\tau}^{N}_{\mathcal{M}_{i, N}}
    \bigr].
  \end{align*}
  On the other hand, by combining again \eqref{eq:escape_prob}, \eqref{mesocap} and \eqref{eq:upper_cap_prefinal}, and recalling the definition of $k_{1}$, it follows that for any $k > k_{1}$ there exists $N_{0}' \in \mathbb{N}$ such that for any $N \geq N_{0}'$
  \begin{align*}
    \me^{-2 \beta a \sqrt{N}}\,
    \max_{i \in \mathcal{I}_{\beta}}
    \widetilde{\Prob}^{N}_{\tilde{\mu}_{N} \rvert \mathcal{M}_{i, N}}\bigl[
      \tilde{\tau}^{N}_{\mathcal{M}_{N} \setminus \mathcal{M}_{i, N}}
      < \tilde{\tau}^{N}_{\mathcal{M}_{i, N}}
    \bigr]
    \;\geq\;
    \frac{1}{N^{\ell_{1}}}\,
    \me^{-N k_{1} -2a\sqrt{N}\beta}
    \;\geq\;
    \me^{-N k}.
  \end{align*}
  Thus, by setting $N_{1} \ldef \widetilde{N}_{1} \vee N_{0}'$, the claim \eqref{eq:optimal:k1:DCWP} follows.
\end{remark}

\section{Capacity estimates for the DCWP model}\label{section:cap_mu}
In this section, we derive bounds for the capacity of the quenched model in relation to that of the annealed model. These estimates are regarded as general, as they do not require any assumptions regarding metastability and are applicable to arbitrary subsets of the configuration space. To establish bounds on the capacity, we adopt the same strategy as developed in \cite{BdHMPS24}. However, the adaptation is not straightforward due to the particularities of the DCWP model and the selection of (potentially) unbounded random variables.

\subsection{Concentration inequalities}\label{app:concent}
We begin with establishing a concentration inequality for functionals of independent random variables. By employing Chernoff-type bounds, this approach generalise McDiarmid's concentration inequality by replacing the bounded-difference assumption with a Lipschitz condition together with mild regularity on the law of the involved random variables. Furthermore, this method yields tighter estimates in the case of vanishing variance, that is, as $v \to 0$, in comparison to \cite[Proposition 2.1]{BMP21}.
%
\begin{theorem}\label{theo:general_conc}
  Let $(\Omega, \mathcal{F}, \prob)$ be a probability space and $n \in \mathbb{N}$. Consider a vector $X = (X_{1}, \dots, X_{n})$ of independent, $\mathbb{R}$-valued random variables on $(\Omega, \mathcal{F}, \prob)$ such that, for any $i \in \{1, \ldots, n\}$, the symmetrised cumulant generating function
  \begin{align}\label{def:varphi}
    \varphi_{i}(\lambda)
    \;\ldef\;
    \ln\mean\bigl[ \me^{\lambda X_{i}} \bigr]
    + \ln\mean\bigl[ \me^{-\lambda X_{i}} \bigr]
  \end{align}
  have domains $\mathcal{D}_i$ containing an open neighbourhood of 0. Further, let $f\colon \mathbb{R}^{n} \to \mathbb{R}$ be a measurable function and suppose that there exists $c_{1}, \ldots, c_{n} \in [0, \infty)$ such that
  \begin{align}\label{cond:lips}
    \bigl| f(x) - f(y) \bigr|
    \;\leq\;
    \sum_{i=1}^{n} c_{i}\, \bigl|x_{i} - y_{i}\bigr|.
  \end{align}
  Then, for $\lambda \in \bigcap_{i=1}^{n} c_{i}^{-1}\mathcal{D}_{i} \cap [0, \infty)$ and $t > 0$,
  \begin{align}
    \label{eq:concentration:ineq}
    \prob\bigl[f(X) - \mean\bigl[f(X)\bigr] > t \bigr]
    \;\leq\;
    \exp\biggl(
      -\lambda t + \sum_{i=1}^{n} \varphi_{i}(\lambda c_{i})
    \biggr).
  \end{align}
  If, additionally, the random variables $X_{1}, \ldots, X_{n}$ are identically distributioned then, for any $t > 0$,
  \begin{align}
    \prob\bigl[f(X) - \mean\bigl[f(X)\bigr] > t \bigr]
    \;\leq\;
    \exp\bigl(
      -n \varphi_{1}^{\ast}\bigl(t/(C n)\bigr)
    \bigr),
  \end{align}
  where $C \ldef \max\{c_{1}, \ldots, c_{n}\}$ and $\varphi_{1}^{\ast}$ denotes the Legendre tranform of $\varphi$.
\end{theorem}
\begin{proof}
  The proof is based on the method of martingale differences. For this purpose, set $\mathcal{F}_{0} \ldef \{\emptyset, \Omega\}$ and, for \(i\in\{1,\dots,n\}\), let $\mathcal{F}_{i} \ldef \sigma(X_{j} : j \leq i)$ be the $\sigma$-algebra generated by the random variables $X_{1}, \ldots, X_{i}$. Further, define $Z_{i} \ldef \mean\bigl[f(X) \mid \mathcal{F}_{i}\bigr]$ for any $i \in \{0, \ldots, n\}$. Then, the stochastic process $(Z_{i})_{i \in \{0, \ldots, n\}}$ is a martingale, and
  \begin{align}
    f(X) - \mean[F(X)] \;=\; \sum_{i=1}^{n} (Z_{i} - Z_{i-1}).
  \end{align}
  For any $i \in \{1, \ldots, n\}$ define the random vector $X^{(i)}$ by $X_{j}^{(i)} \ldef X_{j}$ for any $j \in \{1, \ldots, n\}$ with $j \neq i$ and $X_{i}^{(i)} \ldef X_{i}'$, where $X_{i}'$ is an independent copy of $X_{i}$. Using the independence of the random variables $X_{1}, \ldots, X_{n}$, we obtain that, $\prob$-a.s.,
  \begin{align*}
    \mean\bigl[ f(X^{(i)}) \mid \mathcal{F}_{i} \bigr]
    =
    \int_{\mathbb{R}^{n-i+1}} f(X_{1}, \ldots, X_{i-1}, y_{i}, \ldots, y_{n})\, \prod_{j=i}^{n} \mathrm{d}\bigl(\prob \circ X_{j}^{-1}\bigr)(y_{j})
    =
    \mean\bigl[ f(X) \mid \mathcal{F}_{i-1} \bigr]
  \end{align*}
  for any $i \in \{1, \ldots, n\}$. In particular, for any $\lambda \in \bigcap_{j=1}^{n} c_{j}^{-1}\mathcal{D}_{j} \cap [0, \infty)$ and $i \in \{1, \ldots, n\}$, we have that $\exp(\lambda (Z_{i} - Z_{i-1})) \in L^{1}(\prob)$ and, $\prob$-a.s.,
  \begin{align}
    \label{eq:martingale:difference}
    \mean\bigl[ \exp\bigl(\lambda (Z_{i} - Z_{i-1}) \bigr) \mid \mathcal{F}_{i-1} \bigr]
    &\;=\;
    \mean\bigl[
      \exp\bigl(\lambda \mean\bigl[ f(X) - f(X^{(i)}) \mid \mathcal{F}_{i} \bigr] \bigr)
      \mid \mathcal{F}_{i-1}
    \bigr]
    \nonumber\\[.5ex]
    &\;\leq\;
    \mean\bigl[ \exp\bigl(\lambda (f(X) - f(X^{(i)})\bigr) \mid \mathcal{F}_{i-1} \bigr]
    \nonumber\\[.5ex]
    &\;=\;
    \mean\bigl[ \cosh\bigl(\lambda (f(X) - f(X^{(i)})\bigr) \mid \mathcal{F}_{i-1} \bigr],
  \end{align}
  where in the final equality we used that $\prob \circ (X_{i}, X_{i+1}, \ldots, X_{n})^{-1} = \prob \circ (X_{i}', X_{i+1}, \ldots, X_{n})^{-1}$. Since $\cosh(t) \leq \cosh(s)$ for any $s > |t|$, if follows from \eqref{cond:lips} that, $\prob$-a.s.,
  \begin{align}
    \label{eq:mean:cosh}
    \mean\bigl[ \cosh\bigl(\lambda (f(X) - f(X^{(i)})\bigr) \mid \mathcal{F}_{i-1} \bigr]
    &\;\leq\;
    \mean\bigl[ \cosh\bigl(\lambda c_{i} |X_{i} - X_{i}'|\bigr) \mid \mathcal{F}_{i-1} \bigr]
    \nonumber\\[.5ex]
    &\;=\;
    \mean\bigl[ \cosh\bigl(\lambda c_{i} (X_{i} - X_{i}')\bigr) \bigr]
    \nonumber\\[.5ex]
    &\;=\;
    \mean\bigl[\exp(\lambda c_{i} X_{i})\bigr]\mean\bigl[\exp(-\lambda c_{i} X_{i})\bigr]
    \;=\;
    \exp\bigl( \varphi_{i}(\lambda c_{i})\bigr).
  \end{align}
  Thus, by combining the estimates \eqref{eq:martingale:difference} and \eqref{eq:mean:cosh}, it follows that, for any $i \in \{1, \ldots, n\}$,
  \begin{align}
    \label{eq:exponential:martingale:difference}
    \mean\bigl[
      \exp\bigl(\lambda (Z_{i} - Z_{i-1}) \bigr) \mid \mathcal{F}_{i-1}
    \bigr]
    \;=\;
    \exp\bigl( \varphi_{i}(\lambda c_{i})\bigr)
    \qquad \prob\text{-a.s.}
  \end{align}
  In particular, by applying \eqref{eq:exponential:martingale:difference} iteratively, we get
  \begin{align*}
    \mean\Biggl[
      \prod_{i=1}^{n} \me^{t \lambda (Z_{i} - Z_{i-1})}
    \Biggr]
    \;=\;
    \mean\Biggl[
      \prod_{i=1}^{n-1} \me^{t \lambda (Z_{i} - Z_{i-1})}
      \mean\Bigl[ \me^{t \lambda (Z_{n} - Z_{n-1})} \Big| \mathcal{F}_{n-1} \Bigr]
    \Biggr]
    \overset{\eqref{eq:exponential:martingale:difference}}{\;=\;}
    \prod_{i=1}^{n}
    \me^{\varphi_{i}(\lambda c_{i})}.
  \end{align*}
  Thus, by applying the exponential Markov inequality, we obtain for any $t > 0$
  \begin{align*}
    \prob\bigl[ f(X) - \mean[f(X)] > t \bigr]
    &\;=\;
    \prob\Biggl[ \sum_{i=1}^{n} (Z_{i} - Z_{i-1} ) > t \Biggr]
    \nonumber\\[.5ex]
    &\;\leq\;
    \me^{-t \lambda} \mean\Biggl[ \prod_{i=1}^{n} \me^{t \lambda (Z_{i} - Z_{i-1})} \Biggr]
    \;=\;
    \exp\biggl(-t \lambda + \sum_{i=1}^{n} \varphi_{i}(\lambda c_{i})\biggr),
  \end{align*}
  which conclude the proof of \eqref{eq:concentration:ineq}.
  
  If, additionally, the random variables $X_{1}, \ldots, X_{n}$ are identically distributed, setting $C \ldef \max\{c_{1}, \dots, c_{n}\}$, the right-hand side of \eqref{eq:concentration:ineq} simplifies to
  \begin{align*}
    \prob\bigl[ f(X) - \mean[f(X)] > t \bigr]
    \;\leq\;
    \exp\bigl( -\lambda t + n \varphi_{1}(\lambda C) \bigr).
  \end{align*}
  Since $\varphi_{1}$ is an even function and $t > 0$, 
  \begin{align*}
    \sup_{\lambda \leq 0} \bigl(t \lambda - \varphi_{1}(\lambda) \bigr)
    \;=\;
    \sup_{\lambda \geq 0} \bigl(-t \lambda - \varphi_{1}(-\lambda) \bigr)
    \;\leq\;
    \sup_{\lambda \geq 0} \bigl(t \lambda - \varphi_{1}(\lambda) \bigr),
  \end{align*}
  the assertion follows by optimising over all $\lambda > 0$.
\end{proof}
\begin{corollary}\label{lemma:general_conc_approx}
  Let $X = (X_{1}, \dots, X_{n})$ be i.i.d.\ random variables taking values in $\mathbb{R}$ with cumulant generating function defined in an open interval containing $0$. Let $f\colon \mathbb{R}^{n} \to \mathbb{R} $ satisfy the Lipschitz condition \eqref{cond:lips} with $c_{i} = c$ for all $i \in \{1, \dots, n\}$. Then, for any sequence $(t_{n})_{n \in \mathbb{N}} \subset [0, \infty)$ such that $t_{n} / n \to 0$ as $n \to \infty$,
  %
  \begin{align}
    \prob\bigl[ f(X) - \mean[f(X)] > t_{n} \bigr] 
    \;\leq\; 
    \exp\biggl( -\frac{t_{n}^{2}}{4 \var[X_{1}]\, c^{2} n} (1 + o(1)) \biggr).
  \end{align}
\end{corollary}
\begin{proof}
  First observe that $\varphi_{1}$ as defined in \eqref{def:varphi} is the cumulant-generating function of $X_{1}$ minus an independent copy of itself. Consequently, its Legendre transform inherits some useful properties. As $\varphi_{1}$ is smooth, strictly convex and satisfies $\varphi_{1}(0) = \varphi_{1}'(0) = 0$ and $\varphi_{1}''(0) = 2\var[X_{1}]$ it follows that $\varphi_{1}^{\ast}(0) = {\varphi_{1}^{\ast}}'(0) = 0$ and ${\varphi_{1}^{\ast}}''(0) = 1/(2 \var[X_{1}])$. Hence, for small $t > 0$,
  \begin{align}
    \varphi_{1}^{\ast}(t)
    \;=\;
    \frac{t^2}{4 \var[X_{1}]} + o(t^{2}),
  \end{align}
  and Theorem~\ref{theo:general_conc} then yields the claimed bound.
\end{proof}

\subsection{Concentration of the capacity}\label{sec:conccap}
\begin{lemma}\label{lemma:conclnzcap}
  Let $\beta > 0$ and $(t_{N})_{N \in \mathbb{N}} \subset [0, \infty)$ be a sequence such that $t_{N}/N \to 0$ as $N \to \infty$. Then, for any two non-empty, disjoint subsets $\mathcal{A}, \mathcal{B} \subset \mathcal{S}_{N}$,
  \begin{align}\label{eq:conccap}
    \prob\Bigl[
      \Bigl|
        \ln\bigl( Z_{N} \capacity_{N}(\mathcal{A}, \mathcal{B}) \bigr)
        - \mean\bigl[ \ln\bigl( Z_{N} \capacity_{N}(\mathcal{A},\mathcal{B}) \bigr) \bigr]
      \Bigr| > t_{N}
    \Bigr]
    \;\leq\;
    2\exp\biggl( -\frac{t_{N}^{2}}{2\beta^{2} v} (1+o(1)) \biggr).
  \end{align} 
\end{lemma}
\begin{proof}
  To emphasise the dependence on the random array $J = (J_{ij})_{1 \leq j < j \leq N}$, we henceforth write $H_N^J$, $Z^J_N$, $\mathcal{E}^J_N$, $\capacity^J_N(\mathcal{A},\mathcal{B})$. Further, let $J' = (J_{ij}')_{1 \leq i < j \leq N}$ be the triangular array of random variables which coincides with $J$ except that $J_{kl}'$ is an independent copy of $J_{kl}$ for some fixed pair $(k,l)$. To apply of Lemma~\ref{lemma:general_conc_approx} we first show that, for every $N$, the map $J \mapsto F_{N}(J) \ldef \ln(Z_{N}^{J} \capacity_{N}^{J}(\mathcal{A}, \mathcal{B}))$ satisfies the bounded difference condition \eqref{cond:lips}. By linearity,
  \begin{align}\label{eq:Lipschitz:H}
    \bigl| H_{N}^{J}(\sigma) - H_{N}^{J'}(\sigma) \bigr|
    \;\leq\;
    \frac{|J_{kl} - J'_{kl}|}{N}.
  \end{align}
  In particular, for every $\sigma,\eta \in \mathcal{S}_{N}$
  \begin{align*}
    H_{N}^{J'}(\sigma) \vee H_{N}^{J'}(\eta)
    - \frac{|J_{kl} - J'_{kl}|}{N}
    \;\leq\;
    H_{N}^{J}(\sigma) \vee H_{N}^{J}(\eta)
    \;\leq\;
    H_{N}^{J'}(\sigma) \vee H_{N}^{J'}(\eta)
    + \frac{|J_{kl} - J'_{kl}|}{N}.
  \end{align*}
  Hence, for any test function $h\colon \mathcal{S}_{N} \to \mathbb{R}$, it follows that
  \begin{align*}
    \ln\bigl( Z_{N}^{J'} \mathcal{E}_{N}^{J'}(h) \bigr)
    - \frac{\beta}{N} |J_{kl} - J_{kl}'|
    \;\leq\;
    \ln\bigl( Z_{N}^{J} \mathcal{E}_{N}^{J}(h) \bigr)
    \;\leq\;
    \ln\bigl( Z_{N}^{J'} \mathcal{E}_{N}^{J'}(h) \bigr)
    + \frac{\beta}{N} |J_{kl} - J_{kl}'|.
  \end{align*}
  In view of the Dirichlet principle \eqref{eq:Dirichlet} we further obtain that
  \begin{align*}
    \ln\bigl(Z_{N}^{J}\capacity_{N}^{J}(\mathcal{A}, \mathcal{B})\bigr)
    &\;\leq\;
    \ln\bigl(Z_{N}^{J}\mathcal{E}_{N}^{J}(h_{\mathcal{A}, \mathcal{B}}^{J'})\bigr)
    \;\leq\;
    \ln\bigl(Z_{N}^{J'}\capacity_{N}^{J'}(\mathcal{A}, \mathcal{B})\bigr)
    + \frac{\beta}{N} |J_{kl} - J_{kl}'|.
  \end{align*}
  By repeating the computation above with $J$ and $J'$ interchanged, we obtain
  \begin{align}
    \label{eq:interpolate:cap:kl}
    \bigl|
      \ln\bigl(Z_{N}^{J}\capacity_{N}^{J}(\mathcal{A}, \mathcal{B})\bigr)
      -
      \ln\bigl(Z_{N}^{J'}\capacity_{N}^{J'}(\mathcal{A}, \mathcal{B})\bigr)
    \bigr|
    \;\leq\;
    \frac{\beta}{N} |J_{kl} - J_{kl}'|
  \end{align}
  Finally, to verify that $F_{N}$ satisfies the condition \eqref{cond:lips}, let $J = (J_{ij})_{1 \leq i < j \leq N}$ and $J' = (J_{ij}')_{1 \leq i < j \leq N}$ be two independent triangular arrays. By interpolating between $J$ and $J'$ coordinatewise and using \eqref{eq:interpolate:cap:kl} yields that \eqref{cond:lips} holds with $n = N(N-1)/2$ and $c_{i} = \beta / N$ for every $i \in \{1, \ldots, N(N-1)/2\}$.
\end{proof}

\subsection{Annealed capacity estimates}\label{sec:ancap}
We proceed with annealed estimates that relate the expectations of the main quantities in the DCWP model to their counterparts in the CWP model.
\begin{proposition}\label{theo:meanzcap}
  Let $\beta > 0$ and, for each $N \in \mathbb{N}$, let $\mathcal{A}$ and $\mathcal{B}$ be two non-empty, disjoint subsets of $\mathcal{S}_{N}$.
  \begin{enumerate}[(i)]
  \item 
    Then, as $N \to \infty$,
    \begin{align}\label{eq:meanzcap1}
      \Bigl|
        \mean\bigl[
          \ln\bigl( Z_{N} \capacity_{N}(\mathcal{A}, \mathcal{B}) \bigr)
        \bigr]
        -
        \ln\bigl(
          \widetilde{Z}_{N} \widetilde{\capacity}_{N}(\mathcal{A}, \mathcal{B})
        \bigr)
      \Bigr|
      \;\leq\;
      \frac{\beta^{2}v}{4}\, \bigl(1 + o(1)\bigr).
    \end{align}
    
  \item
    Then, for any $k \geq 1$ as $N \to \infty$,
    \begin{align}\label{eq:meanzcap2}
      \me^{-\beta^{2}v/4 (1+o(1))}
      \;\leq\;
      \frac{
        \mean\Bigl[
          \bigl( Z_{N} \capacity_{N}(\mathcal{A}, \mathcal{B}) \bigr)^{\pm k}
        \Bigr]^{1/k}
      }{
        \bigl(
          \widetilde{Z}_{N} \widetilde{\capacity}_{N}(\mathcal{A}, \mathcal{B})
        \bigr)^{\pm 1}
      }
      \;\leq\;
      \me^{k\beta^{2}v/4 (1+o(1))}.
    \end{align}
  \end{enumerate}
\end{proposition}
\begin{proof}
  Recall from Section~\ref{subsec:pot_theo} the both definition of the Dirichlet form $\mathcal{E}_{N}(f)$ for functions $f \in \mathcal{H}_{\mathcal{A}, \mathcal{B}}^N$ and the quadratic form $\mathcal{D}_{N}(\varphi)$ for unit flows $\varphi \in \mathcal{U}_{\mathcal{A}, \mathcal{B}}^N$. In view of Lemma~\ref{coro:betapmapprox}(ii) we have that
  \begin{align}
    \label{eq:EB:forms:estimate}
    \mean\bigl[ Z_{N} \mathcal{E}_{N}(f) \bigr]
    &\;\leq\;
    \widetilde{Z}_{N} \widetilde{\mathcal{E}}_{N}(f)\, 
    \me^{\beta^{2} v/4} \bigl(1+o(1)\bigr)
    \mspace{72mu} \forall\, f \in \mathcal{H}_{\mathcal{A}, \mathcal{B}}^N,
    \\
    \mean\bigl[Z_{N}^{-1} \mathcal{D}_{N}(\varphi)\bigr]
    &\;\leq\;
    \widetilde{Z}_{N}^{-1} \widetilde{\mathcal{D}}_{N}(\varphi)\, 
    \me^{\beta^{2} v/4} \bigl(1+o(1)\bigr)
    \mspace{59mu} \forall\, \varphi \in \mathcal{U}_{\mathcal{A}, \mathcal{B}}^N.
  \end{align}
  The remaining part of the proof literally follows from \cite[Proposition 4.3]{BdHMPS24}.
\end{proof}

\section{Estimates of the harmonic sum}\label{sec:final}
In this section we control the numerator in Equation~\eqref{eq:golden_formula}, namely the $\ell_{1}(\mu)$-norm of the harmonic function, also called \emph{harmonic sum}. This requires first a preliminary estimate obtained in Proposition~\ref{prop:harm_estimate}, which significantly simplifies the harmonic sum and is used in Lemma~\ref{lemma:mean_harm_estimate} and Propositions~\ref{lemma:concharm_sum} and~\ref{prop:mean_ln_hsum}.

The main differences between the proof of Proposition~\ref{prop:harm_estimate} and that of \cite[Proposition 5.4]{BdHMPS24} are the treatment of multiple regimes and the removal of the non-degeneracy assumption \cite[Equation~(2.22)]{BdHMPS24}. That assumption fails in some regimes and is unnecessary here, since the measure of the metastable valleys is controlled by the explicit estimates of Lemma~\ref{lemma:partition_weight}. Although the proof of Proposition~\ref{prop:harm_estimate} proceeds by a regime-wise analysis, the remaining results in this section are independent on the inner structure of the metastable sets.

\subsection{Metastable partition and preliminary estimates}
We begin with the following definition, needed to state Proposition~\ref{prop:harm_estimate}. It is convenient to partition the state space into the neighbouring valleys associated with each relevant metastable set with respect to the free energy landscape $\widetilde{F}_{\beta, q}$, leaving only a remainder of negligible weight.
\begin{definition}[Metastable partition]\label{def:meta_valley}
  For $\beta > \beta_{1}(q)$ let $\{\mathcal{M}_{i,N} : i \in \mathcal{I}_{\beta}\}$ be the metastable sets as defined in \eqref{eq:metaset} and $\bigcup_{i \in \mathcal{I}_{\beta}} \mathcal{M}_{i, N}$. The metastable sets $\{\mathcal{M}_{i, N} : i \in \mathcal{I}_{\beta}\}$ give rise to a \emph{metastable partition} $\{\mathcal{S}_{i, N} : i \in \mathcal{I}_{\beta}\}$ (with respect to the CWP model) such that
  \begin{align}
    \mathcal{M}_{i, N} \;\subset\; \mathcal{S}_{i, N} \;\subset\; \mathcal{V}_{i, N},
    \qquad i \in \mathcal{I}_{\beta},
  \end{align}
  where, for any $i \in \mathcal{I}_{\beta}$, the \emph{local valley} $\mathcal{V}_{i, N}$ around the metastable set $\mathcal{M}_{i, N}$ is defined by
  \begin{align}
    \mathcal{V}_{i, N}
    \;\ldef\;
    \Bigl\{
      \sigma \in \mathcal{S}_{N} \!\setminus\! \mathcal{M}_{N} :
      h_{\mathcal{M}_{i, N}, \mathcal{M}_{N} \setminus \mathcal{M}_{i, N}}^{N}(\sigma)
      \geq
      \max_{j \neq i}
      h_{\mathcal{M}_{j, N}, \mathcal{M}_{N} \setminus \mathcal{M}_{j, N}}^{N}(\sigma)
    \Bigr\}.
  \end{align}
\end{definition}
\begin{remark}\label{remark:sym_construction}
  Note that a metastable partition of the state space $\mathcal{S}_{N}$ is not uniquely determined by the conditions above: there may exist subsets $\mathcal{X} \subset \mathcal{V}_{i, N} \cap \mathcal{V}_{j, N}$ lying in the intersection of two distinct local valleys $\mathcal{V}_{i, N}$ and $\mathcal{V}_{j,N}$. However, it is well known (cf. \cite[Lemma~3.2]{SS19} and \cite[Lemma~8.7]{BdH15}) that the Gibbs measure $\tilde{\mu}_{N}[X]$ of such a set is negligible compared with the Gibbs measures $\tilde{\mu}_{N}[\mathcal{M}_{i, N}]$ and $\tilde{\mu}_{N}[\mathcal{M}_{j, N}]$ of the corresponding metastable sets $\mathcal{M}_{i, N}$ and $\mathcal{M}_{j, N}$. In particular, we may choose the metastable partition $\{\mathcal{S}_{i, N} : i \in \mathcal{I}_{\beta}\}$ so that, for every $i \in \mathcal{I}_{\beta}$, $\mathcal{S}_{i, N} = L_{N}^{-1}(\boldsymbol{S}_{i, N})$ for some $\boldsymbol{S}_{i, N} \subset \mathcal{P}_{N}$.
\end{remark}
\begin{figure}
  \centering
  \begin{tikzpicture}[line cap=round,line join=round,>=triangle 45,x=0.5cm,y=0.5cm]
    \clip(-0.23,0.02) rectangle (10.78,4.49);
    \draw [->] (4.83,2.18) -- (6.06,2.16);
    \begin{scriptsize}
      \fill [color=fftttt] (1.02,1.44) circle (4.5pt);
      \fill [color=fftttt] (1.94,1.44) circle (4.5pt);
      \fill [color=eeeeee] (1.48,2.24) circle (4.5pt);
      \fill [color=ttttff] (2.4,2.24) circle (4.5pt);
      \fill [color=ttttff] (1.94,3.05) circle (4.5pt);
      \fill [color=ttffqq] (2.87,1.44) circle (4.5pt);
      \fill [color=eeeeee] (3.33,2.25) circle (4.5pt);
      \fill [color=ttttff] (2.87,3.05) circle (4.5pt);
      \fill [color=ttttff] (2.4,3.85) circle (4.5pt);
      \fill [color=ttffqq] (3.79,1.45) circle (4.5pt);
      \fill [color=fftttt] (1.48,0.64) circle (4.5pt);
      \fill [color=eeeeee] (2.41,0.64) circle (4.5pt);
      \fill [color=ttffqq] (3.33,0.64) circle (4.5pt);
      \fill [color=fftttt] (0.56,0.64) circle (4.5pt);
      \fill [color=ttffqq] (4.26,0.64) circle (4.5pt);
      \fill [color=fftttt] (6.98,1.39) circle (4.5pt);
      \fill [color=fftttt] (7.91,1.39) circle (4.5pt);
      \fill [color=fftttt] (7.44,2.19) circle (4.5pt);
      \fill [color=ttttff] (8.37,2.19) circle (4.5pt);
      \fill [color=ttttff] (7.91,2.99) circle (4.5pt);
      \fill [color=ttffqq] (8.83,1.39) circle (4.5pt);
      \fill [color=ttttff] (9.29,2.19) circle (4.5pt);
      \fill [color=ttttff] (8.83,3) circle (4.5pt);
      \fill [color=ttttff] (8.37,3.8) circle (4.5pt);
      \fill [color=ttffqq] (9.76,1.39) circle (4.5pt);
      \fill [color=fftttt] (7.45,0.59) circle (4.5pt);
      \fill [color=ttffqq] (8.37,0.59) circle (4.5pt);
      \fill [color=ttffqq] (9.3,0.59) circle (4.5pt);
      \fill [color=fftttt] (6.52,0.59) circle (4.5pt);
      \fill [color=ttffqq] (10.22,0.59) circle (4.5pt);
    \end{scriptsize}
  \end{tikzpicture}\\
  \begin{tikzpicture}[line cap=round,line join=round,>=triangle 45,x=0.5cm,y=0.5cm]
    \clip(-0.09,0.62) rectangle (10.42,4.77);
    \draw [->] (4.76,2.64) -- (5.99,2.62);
    \begin{scriptsize}
      \fill [color=fftttt] (1.02,1.44) circle (4.5pt);
      \fill [color=fftttt] (1.94,1.44) circle (4.5pt);
      \fill [color=fftttt] (1.48,2.24) circle (4.5pt);
      \fill [color=eeeeee] (2.4,2.24) circle (4.5pt);
      \fill [color=ttttff] (1.94,3.05) circle (4.5pt);
      \fill [color=ttffqq] (2.87,1.44) circle (4.5pt);
      \fill [color=ttffqq] (3.33,2.25) circle (4.5pt);
      \fill [color=ttttff] (2.87,3.05) circle (4.5pt);
      \fill [color=ttttff] (2.4,3.85) circle (4.5pt);
      \fill [color=ttffqq] (3.79,1.45) circle (4.5pt);
      \fill [color=fftttt] (6.98,1.39) circle (4.5pt);
      \fill [color=fftttt] (7.91,1.39) circle (4.5pt);
      \fill [color=fftttt] (7.44,2.19) circle (4.5pt);
      \fill [color=fftttt] (8.37,2.19) circle (4.5pt);
      \fill [color=ttttff] (7.91,2.99) circle (4.5pt);
      \fill [color=ttffqq] (8.83,1.39) circle (4.5pt);
      \fill [color=ttffqq] (9.29,2.19) circle (4.5pt);
      \fill [color=ttttff] (8.83,3) circle (4.5pt);
      \fill [color=ttttff] (8.37,3.8) circle (4.5pt);
      \fill [color=ttffqq] (9.76,1.39) circle (4.5pt);
    \end{scriptsize}
  \end{tikzpicture}
  \caption{%
    Graphical representation of a metastable partition of $\mathcal{P}_{N}$ (with $N \in \{5, 6\}$ and $q=3$) for the regimes in which $\boldsymbol{m}_{0}$ is not a local minimum. The red, blue and green points denotes $\boldsymbol{S}_{1, N}$, $\boldsymbol{S}_{2, N}$ and $\boldsymbol{S}_{3,N}$, respectively. Points shown in grey may be assigned to more than one set in the partition, as explained in Remark~\ref{remark:sym_construction}.
  }
\end{figure}

Before proceeding, we need to control the measure of each component within the meta\-stable partition. The following technical lemma establishes that this measure deviates from that of the metastable sets by at most a polynomial factor.
\begin{lemma}\label{lemma:partition_weight}
  Suppose that $\beta > \beta_{1}$. Further, $N_{0}(\beta)$ be as in Proposition~\ref{prop:preem_cap_estimates}-(ii). Then, there exists $\ell_{4} \equiv \ell_{4}(\beta, q) \in [q, \infty)$ such that, for all $N \geq N_{0}(\beta)$,
  \begin{align}\label{eq:annealed_part}
    \tilde{\mu}_{N}\bigl[ \mathcal{S}_{i, N} \bigr]
    \;\leq\;
    N^{\ell_{4}}\tilde{\mu}_{N}\bigl[ \mathcal{M}_{i, N} \bigr],
    \qquad \forall\, i \in \mathcal{I}_{\beta}.
  \end{align}
  In particular, for all $N \geq N_{0}(\beta)$, on the event $\Xi_N(a)$,
  \begin{align}\label{eq:quenched_part}
    \mu_{N}\bigl[ \mathcal{S}_{i,N} \bigr]
    \;\leq\;
    N^{\ell_{4}} \me^{2\beta a\sqrt{N}} \mu_{N}\bigl[ \mathcal{M}_{i, N} \bigr],
    \qquad \forall\, i \in \mathcal{I}_{\beta}.
  \end{align}
\end{lemma}
\begin{proof}
  Fix some arbitrary $i \in \mathcal{I}_{\beta}$. Then, the definition of the local valley implies that, for any $\sigma \in \mathcal{V}_{i, N}$,
  \begin{align*}
    1
    &\;=\;
    \sum_{j \in \mathcal{I}_{\beta}}
    \widetilde{\Prob}_{\!\sigma}^{N}\bigl[
      \widetilde{\Sigma}^{N}(\tilde{\tau}_{\mathcal{M}_{N}}^{N}) \in \mathcal{M}_{j, N}
    \bigr]
    \\
    &\;=\;
    \sum_{j \in \mathcal{I}_{\beta}}
    \widetilde{\Prob}_{\!\sigma}^{N}\bigl[\tilde{\tau}_{\mathcal{M}_{j, N}}^{N} < \tilde{\tau}_{\mathcal{M}_{N} \setminus \mathcal{M}_{j, N}}^{N}\bigr]
    \;\leq\;
    |\mathcal{I}_{\beta}|\,
    \widetilde{\Prob}_{\!\sigma}^{N}\bigl[\tilde{\tau}_{\mathcal{M}_{i, N}}^{N} < \tilde{\tau}_{\mathcal{M}_{N} \setminus \mathcal{M}_{i, N}}^{N}\bigr].
  \end{align*}
  Consequently, for every $\sigma \in \mathcal{S}_{i, N} \!\setminus\! \mathcal{M}_{i, N}$ we have $\widetilde{h}_{\mathcal{M}_{i, N}, \mathcal{M}_{N} \setminus \mathcal{M}_{i, N}}^{N}(\sigma) \geq 1/|\mathcal{I}_{\beta}|$. On the other hand, using the properties of the equilibrium potientials $\widetilde{h}_{\mathcal{M}_{i, N}, \mathcal{M}_{N} \setminus \mathcal{M}_{i, N}}^{N}$ and $\widetilde{h}_{\mathcal{S}_{i, N} \setminus \mathcal{M}_{i, N}, \mathcal{M}_{N}}^{N}$, it follows that
  \begin{align}\label{eq:annealed:part:est1}
    0
    &\;=\;
    \scpr{%
      \widetilde h_{\mathcal{M}_{i, N}, \mathcal{M}_{N} \setminus \mathcal{M}_{i, N}}^{N},
      -\widetilde{\mathcal{L}}_{N} \widetilde{h}_{\mathcal{S}_{i, N} \setminus \mathcal{M}_{i, N}, \mathcal{M}_{N}}^{N}
    }_{\tilde{\mu}_{N}}
    \nonumber\\[1.5ex]
    &\;\geq\;
    \frac{1}{|\mathcal{I}_{\beta}|}\,
    \widetilde\capacity_{N}\bigl(\mathcal{S}_{i, N} \setminus \mathcal{M}_{i, N}, \mathcal{M}_{N}\bigr)
    -
    \sum_{\sigma \in \mathcal{M}_{i, N}} \mspace{-6mu}
    \tilde{\mu}_{N}(\sigma)\,
    \widetilde{\Prob}_{\sigma}^{N}\bigl[
      \tilde{\tau}_{\mathcal{S}_{i, N} \!\setminus\! \mathcal{M}_{i,N}}^{N} < \tilde{\tau}_{\mathcal{M}_{N}}^{N}\bigr]
    \nonumber\\
    &\;\geq\;
    \frac{1}{|\mathcal{I}_{\beta}|}\,
    \widetilde{\capacity}_{N}\bigl(
      \mathcal{S}_{i, N} \!\setminus\! \mathcal{M}_{i, N}, \mathcal{M}_{N}
    \bigr)
    - \tilde{\mu}_{N}\bigl[\mathcal{M}_{i, N}\bigr].
  \end{align}
  Thus, we are left with bounding $\widetilde{\capacity}_{N}\bigl(\mathcal{S}_{i, N} \!\setminus\! \mathcal{M}_{i, N}, \mathcal{M}_{N} \bigr)$ from below. Repeating the argument of \eqref{eq:def:meta:lb1}, estimating $|\boldsymbol{S}_{i, N} \!\setminus\! \{\boldsymbol{m}_{i, N}\}| \leq N^{g}$, and invoking Proposition~\ref{prop:preem_cap_estimates}-(ii), we deduce that for any $N \geq N_{0}(\beta)$,
  \begin{align}\label{eq:annealed:part:est2}
    \widetilde{\capacity}_{N}\bigl(
      \mathcal{S}_{i, N} \!\setminus\! \mathcal{M}_{i, N}, \mathcal{M}_{N}
    \bigr)
    &\overset{\eqref{mesocap}}{\;=\;}
    \capacitym_{N}\bigl(
      \boldsymbol{S}_{i, N} \!\setminus\! \{\boldsymbol{m}_{i, N}\}, \boldsymbol{M}_{N}
    \bigr)
    \nonumber\\[1ex]
    &\;\geq\;
    N^{-\ell_{3}-q}\,
    \widetilde{Q}_{N}\bigl[
      \boldsymbol{S}_{i, N} \!\setminus\! \{\boldsymbol{m}_{i, N}\}
    \bigr]
    \;=\;
    N^{-\ell_{3}-q}\,
    \tilde{\mu}_{N}\bigl[
      \mathcal{S}_{i, N} \!\setminus\! \mathcal{M}_{i, N}
    \bigr]
  \end{align}
  Thus, by combining \eqref{eq:annealed:part:est1} and \eqref{eq:annealed:part:est2}, the assertion \eqref{eq:annealed_part} is immediate. Finally, by combining \eqref{eq:annealed_part} with \eqref{eq:xi_control_prob}, concludes the proof of \eqref{eq:quenched_part}.
\end{proof}
This result allows us to compare measures of different elements of the metastable partition. For $i, j \in \mathcal{I}_{\beta}$, by monotonicity of the measure and the previous lemma, we have
\begin{align}\label{eq:measure_comp}
  \frac{%
    \mu_{N}\bigl[\mathcal{S}_{i, N}\bigr]
  }{
    \mu_{N}\bigl[\mathcal{S}_{j, N}\bigr]
  }
  \;\leq\;
  \me^{2 \beta a \sqrt{N}} N^{\ell_{4}}\,
  \frac{%
    \widetilde{\mu}_{N}\bigl[\mathcal{M}_{i, N}\bigr]
  }{
    \widetilde{\mu}_{N}\bigl[\mathcal{M}_{j, N}\bigr]
  }
  \;=\;
  \me^{2\beta a \sqrt{N}} N^{\ell_{4}}\,
  \me^{%
    \beta N(\widetilde{F}_{\beta, q}(\boldsymbol{m}_{j})
    - \widetilde{F}_{\beta, q}(\boldsymbol{m}_{i})+O(N^{-1}))
  }. 
\end{align}
with an analogous lower bound.

In the next proposition, we estimate the $\ell_{1}(\mu_N)$-norm of the harmonic function, where $\mathcal{A}_{N},\mathcal{B}_{N}$ are the initial and final sets in each regime, as in Definition \ref{def:regimes}. The proof is inspired by \cite[Lemma~3.3]{SS19} and \cite[Proposition~5.2]{BdHMPS24}, and we include detailed explanations only where it deviates from their methods. The main modifications are due to the presence of multiple regimes in the DCWP model and the choice of unbounded random variables $J_{ij}$.
\begin{proposition}\label{prop:harm_estimate}
  For any $\beta > \beta_{1}$ let $\mathcal{A}_{N}$ and $\mathcal{B}_{N}$ be chosen as in Definition~\ref{def:regimes}, and $a > \sqrt{v \ln q}$. Then, there exists a constant $C_{1} \in (0, k_{1})$ such that, on the event $\Xi_{N}(a)$, as $N \to \infty$
  \begin{align}
    \label{eq:harm:estimate}
    \norm{h_{\mathcal{A}_{N}, \mathcal{B}_{N}}^{N}}_{\mu_{N}}
    \;=\;
    \mu_{N}\bigl[ \mathcal{S}_{\mathcal{A}, N}\bigr]\, \bigl(1 + O(\me^{-C_{1} N})\bigr),
  \end{align}
  where $\mathcal{S}_{\mathcal{A}, N}$ denotes the union of the elements of the metastable partition $\{\mathcal{S}_{i, N} : i \in \mathcal{I}_{\beta}\}$ corresponding to the metastable sets in $\mathcal{A}_N$, that is,
  \begin{align}\label{eq:union_part}
    \mathcal{S}_{\mathcal{A}, N}
    \;=\;
    \bigcup_{i : \mathcal{M}_{i, N} \subset \mathcal{A}_{N}} \mathcal{S}_{i, N}.
  \end{align}
\end{proposition}
\begin{proof}
  The proof comprises two steps. In the first step, we provide a regime independent upper bound for the average of the equilibrium potential between different metastable sets with respect to the Gibbs measure conditioned on the element of the metastable partition corresponding to the hitting set. This is done by comparison with the CWP model. In the second step, we control the harmonic sum by decomposing it over the elements of the metastable partition. This decomposition is regime dependent, and depends on the relative weights of the metastable partition under $\mu_{N}$.
  \medskip 
  
  \textit{Step 1.} 
  Let $a$ be as stated and $k \in (0, k_{1})$. We claim that for any two metastable sets $\mathcal{M}_{i, N},\mathcal{M}_{j, N}$, any $\varepsilon \in (0, 1]$ and large enough $N$, on the event $\Xi_{N}(a)$, 
  \begin{align} \label{eq:harm:claim:1}
    \norm{h_{\mathcal{M}_{j, N}, \mathcal{M}_{i, N}}^{N}}_{\mu_{N} \rvert  \mathcal{S}_{i, N}}
    \;\leq\;
    \varepsilon +  \me^{-k N}\, \ln(1/\varepsilon) 
    \min\biggl\{
      1, \frac{\mu_{N}[\mathcal{S}_{j, N}]}{\mu_{N}[\mathcal{S}_{i, N}]}
    \biggr\}.
  \end{align}
  Indeed, this literally follows from \cite[Proposition 5.2, Step 1.]{BdHMPS24} except for the following modification: By using \eqref{eq:xi_control_prob}, we have that, for any non-empty $ \mathcal X\subset \mathcal{S}_{k,N} \setminus \mathcal{M}_{k,N}$, on the event  $\Xi_N(a)$,
  \begin{align} \label{eq:muX}
    \mu_{N}[\mathcal{X}]
    &\;=\;
    \frac{%
      \capacity_{N}(\mathcal{X}, \mathcal{M}_{i, N})
    }{%
      \Prob_{\mu_{N} \rvert \mathcal{X}}^{N}\bigl[
        \tau^{N}_{\mathcal{M}_{i, N}} < \tau^{N}_{\mathcal{X}}
      \bigr]
    }
    \nonumber\\
    &\overset{\mspace{-15mu}\eqref{eq:xi_control_prob}\mspace{-15mu}}{\;\leq\;}
    \me^{2\beta a\sqrt{N}}\,
    \frac{%
      \capacity_{N}(\mathcal{X}, \mathcal{M}_{i, N})
    }{%
      \widetilde{\Prob}_{\tilde{\mu}_{N} \rvert \mathcal{X}}^{N}\bigl[
        \tilde{\tau}_{\mathcal{M}_{i, N}}^{N} < \tilde{\tau}_{\mathcal{X}}^{N}
      \bigr]
    }
    \nonumber\\
    &\;\leq\;
    \me^{2 \beta a\sqrt{N}}\, \me^{-k N}
    \biggl(
      \max_{\ell \in \mathcal{I}_{\beta}} 
      \widetilde{\Prob}_{\tilde{\mu}_{N} \rvert \mathcal{M}_{\ell, N}}^{N}\bigl[
        \tilde{\tau}_{\mathcal{M}_{N} \setminus \mathcal{M}_{\ell, N}}^{N} 
        < \tilde{\tau}_{\mathcal{M}_{\ell, N}}^{N}
      \bigr]
    \biggr)^{\!\!-1}
    \capacity_{N}(\mathcal{X}, \mathcal{M}_{i, N})
    \nonumber\\[.5ex]
    &\overset{\mspace{-15mu}\eqref{eq:xi_control_prob}\mspace{-15mu}}{\;\leq\;}
    \me^{4 \beta a\sqrt{N}}\,
    \me^{-k N}
    \biggl(
      \max_{\ell \in \mathcal{I}_{\beta}} 
      \Prob_{\mu_{N} \rvert \mathcal{M}_{\ell, N}}^{N}\bigl[
        \tau_{\mathcal{M}_{N} \setminus \mathcal{M}_{\ell, N}}^{N} 
        < \tau_{\mathcal{M}_{\ell, N}}^{N}
      \bigr]
    \biggr)^{\!\!-1}
    \capacity_{N}(\mathcal{X}, \mathcal{M}_{i, N}),
  \end{align}
  where we applied Theorem~\ref{thm:meta:cwp} and \cite[Lemma~3.1]{SS19} to obtain the second inequality. Note that \cite[Lemma~3.1]{SS19} cannot be applied directly in the first line because the sets $\{\mathcal{S}_{i,N} i \in \mathcal{I}_{\beta}\}$ form a metastable partition for the \emph{CWP model}, not the \emph{DCWP} model; this necessitates the estimates in the second and third lines.
  \medskip
  
  \textit{Step 2.}
  In view of \eqref{eq:harm:claim:1}, the proof of \eqref{eq:harm:estimate} follows the same lines as that of \cite[Theorem~1.7]{SS19}; however, modifications are required for each regime in Definition~\ref{def:regimes}.
  
  \textit{First metastable regime:} Let $\beta_{1} < \beta \leq \beta_2$, $\mathcal{A}_{N} = \bigcup_{\ell = 1}^{q}{\mathcal{M}}_{\ell, N}$, $ \mathcal{B}_{N} = \mathcal{M}_{0, N},$ and write
  \begin{align}\label{eq:hABmu_dec_r1}
    \norm{h_{\mathcal{A}_{N}, \mathcal{B}_{N}}^{N}}_{\mu_{N}}
    \;=\;
    \mu_{N}\bigl[\mathcal{S}_{0, N}\bigr]\,
    \norm{h_{\mathcal{A}_{N}, \mathcal{B}_N}^{N}}_{\mu_{N} \vert  \mathcal{S}_{0, N}}
    + \sum_{j=1}^{q} \mu_{N}\bigl[\mathcal{S}_{j, N}\bigr]\, \norm{h_{\mathcal{A}_{N}, \mathcal{B}_{N}}^{N}}_{\mu_{N} \vert \mathcal{S}_{j, N}}.
  \end{align}
  In order to prove a lower bound, we neglect the first term appearing on the right-hand side of \eqref{eq:hABmu_dec_r1}, while bounding each term in the remaining sum from below by
  \begin{align}\label{eq:r1:est1}
    \norm{h_{\mathcal{A}_{N}, \mathcal{B}_{N}}^{N}}_{\mu_{N} \vert  \mathcal{S}_{j, N}}
    \;=\;
    1 - 
    \norm{h_{\mathcal{B}_{N}, \mathcal{A}_{N}}^{N}}_{\mu_{N} \vert  \mathcal{S}_{j, N}}
    \;\geq\;
    1 - \norm{h_{\mathcal{M}_{0, N}, \mathcal{M}_{j, N}}^{N}}_{\mu_N \vert  \mathcal{S}_{j, N}},
  \end{align}
  where we used that, for all $\sigma \in \mathcal{S}_{N} \setminus (\mathcal{A}_{N} \cup \mathcal{B}_{N})$,
  \begin{align*}
    h_{\mathcal{B}_{N}, \mathcal{A}_N}^{N}(\sigma)
    &\;=\;
    \Prob_{\!\sigma}^{N}\Bigl[
      \tau_{\mathcal{M}_{0, N}}^{N}
      < \tau_{\bigcup_{\ell=1}^{q} \mathcal{M}_{\ell, N}}^{N}
    \Bigr]
    \;\leq\;
    \Prob_{\!\sigma}^{N}\bigl[
      \tau_{\mathcal{M}_{0, N}}^{N} < \tau_{\mathcal{M}_{j, N}}^{N}
    \bigr]
    \;=\;
    h_{\mathcal{M}_{0, N}, \mathcal{M}_{j, N}}^{N}(\sigma).
  \end{align*}
  By applying \eqref{eq:harm:claim:1} with $\varepsilon = \me^{-k N}$, we further get that, on the event $\Xi_{N}(a)$,
  \begin{align}\label{eq:r1:est2}
    \norm{h_{\mathcal{M}_{0, N}, \mathcal{M}_{j, N}}^{N}}_{\mu_N \vert  \mathcal{S}_{j, N}}
    \;\leq\;
    \varepsilon + \me^{-k N} \ln(1/\varepsilon)
    \;=\;
    \me^{-k N} (1 + k N)
  \end{align}
  for any $j \in \{1, \ldots, q\}$. Thus, by combining \eqref{eq:r1:est1} and \eqref{eq:r1:est2} with \eqref{eq:hABmu_dec_r1} and using the additivity of the Gibbs measure $\mu_{N}$, we obtain on the event $\Xi_{N}(a)$,
  \begin{align}\label{eq:harm:sum:lb1}
    \norm{h_{\mathcal{A}_{N}, \mathcal{B}_{N}}^{N}}_{\mu_{N}}
    \;\geq\;
    \mu_{N}\bigl[\mathcal{S}_{\mathcal{A}, N}\bigr]\,
    \bigl(
      1 -  \me^{-k N} (1 + k N)
    \bigr).
  \end{align}
  To get the upper bound, by recalling that $h_{\mathcal{A}_{N} \mathcal{B}_{N}}^{N} \leq 1$, we immediately have that the last sum appearing on the right-hand side of \eqref{eq:hABmu_dec_r1} is bounded from above by $\mu_{N}\bigl[\mathcal{S}_{\mathcal{A}, N}\bigr]$. Since for any $\sigma \in \mathcal{S}_{N}$,
  \begin{align*}
    \Prob_{\!\sigma}^{N}\Bigl[
      \tau_{\bigcup_{\ell=1}^{q} \mathcal{M}_{\ell, N}}^{N}
      < \tau_{\mathcal{M}_{0, N}}^{N}
    \Bigr]
    \;=\;
    \sum_{\ell=1}^{q}\,
    \Prob_{\!\sigma}^{N}\bigl[
      \tau_{\mathcal{M}_{\ell, N}}^{N}
      < \tau_{\mathcal{M}_{N} \setminus \mathcal{M}_{\ell, N}}^{N}
    \bigr]
    \;\leq\;
    \sum_{\ell=1}^{q}\,
    \Prob_{\!\sigma}^{N}\bigl[
      \tau_{\mathcal{M}_{\ell, N}}^{N} < \tau_{\mathcal{M}_{0, N}}^{N}
    \bigr]
  \end{align*}
  an application of Minkowski's inequality and \eqref{eq:harm:claim:1} yields that, on the event $\Xi_{N}(a)$,
  \begin{align*}
    \norm{h_{\mathcal{A}_{N}, \mathcal{B}_{N}}^{N}}_{\mu_{N} \vert \mathcal{S}_{0, N}}
    &\;\leq\;
    \sum_{\ell=1}^{q}\;
    \norm{h_{\mathcal{M}_{\ell, N}, \mathcal{M}_{0, N}}^{N}}_{\mu_{N} \vert \mathcal{S}_{0, N}}
    \;\leq\;
    q\Biggl(
      \varepsilon + \me^{-k N} \ln(1/\varepsilon)\,
      \frac{%
        \mu_{N}\bigl[\mathcal{S}_{\mathcal{A}, N}\bigr]
      }{\mu_{N}\bigl[\mathcal{S}_{0, N}\bigr]}
    \Biggr).
  \end{align*}
  By combining the above estimates and choosing $\varepsilon = \me^{-k N} \mu_{N}\bigl[\mathcal{S}_{\mathcal{A}, N}\bigr] / \mu_{N}\bigl[\mathcal{S}_{0, N}\bigr]$, we therefore obtain that, on the event $\Xi_{N}(a)$,
  \begin{align}
    \norm{h_{\mathcal{A}_N, \mathcal{B}_N}^N}_{\mu_N}
    \;\leq\;
    \mu_{N}\bigl[\mathcal{S}_{\mathcal{A}, N}\bigr]\,
    \Biggl(
      1 + q \me^{-k N}
      \Biggl(
        1 + k N + \ln \frac{\mu_{N}\bigl[\mathcal{S}_{0, N}\bigr]}{\mu_{N}\bigl[\mathcal{S}_{\mathcal{A}, N}\bigr]}
      \Biggr)
    \Biggr)
  \end{align}
  We conclude by noting that, due to \eqref{eq:measure_comp}, there exists a deterministic constant $c \in (0, \infty)$ such that $\ln\bigl(\mu_{N}[\mathcal{S}_{0, N}]/\mu_{N}[\mathcal{S}_{\mathcal{A}, N}]\bigr) \leq c N$.
  \smallskip
  
  \textit{Second metastable regime:} Let $\beta_{2} < \beta < \beta_{4}$, $ \mathcal{A}_{N} = \mathcal{M}_{0, N}$, $ \mathcal{B}_{N} = \cup_{\ell=1}^{q} \mathcal{M}_{\ell, N}$, and write
  \begin{align}\label{eq:hABmu_dec}
    \norm{h_{\mathcal{A}_{N}, \mathcal{B}_{N}}^{N}}_{\mu_{N}}
    \;=\;
    \mu_{N}\bigl[\mathcal{S}_{0, N}\bigr]\,
    \Biggl(
      \norm{h_{\mathcal{A}_{N}, \mathcal{B}_N}^N}_{\mu_N \vert  \mathcal{S}_{0,N}}
      +
      \sum_{j=1}^{q}\,
      \frac{%
        \mu_{N}\bigl[\mathcal{S}_{j, N}\bigr]
      }{\mu_{N}\bigl[\mathcal{S}_{0, N}\bigr]}\,
      \norm{h_{\mathcal{A}_{N}, \mathcal{B}_{N}}^{N}}_{\mu_{N} \vert  \mathcal{S}_{j, N}}
    \Biggr).
  \end{align}
  In order to prove a lower bound, we neglect the final term, while the first term is bounded from below by
  \begin{align}\label{eq:h:lb:arg}
    \norm{h_{\mathcal{A}_{N}, \mathcal{B}_{N}}^{N}}_{\mu_{N} \vert \mathcal{S}_{0, N}}
    \;=\;
    1 - \norm{h_{\mathcal{B}_{N}, \mathcal{A}_{N}}^{N}}_{\mu_{N} \vert  \mathcal{S}_{0, N}}
    \;\geq\;
    1 - \sum_{\ell=1}^{q}\;
    \norm{h_{\mathcal{M}_{\ell, N}, \mathcal{M}_{0, N}}^{N}}_{\mu_{N} \vert  \mathcal{S}_{0, N}}.
  \end{align}
  By applying again \eqref{eq:harm:claim:1} with $\varepsilon = \me^{-kN}$, we obtain that, on the event $\Xi_{N}(a)$,
  \begin{align*}
    \norm{h_{\mathcal{A}_{N}, \mathcal{B}_{N}}^{N}}_{\mu_{N} \vert  \mathcal{S}_{0, N}}
    \;\geq\;
    1 - q\bigl( \varepsilon + \me^{-k N} \ln(1/\varepsilon)\bigr)
    \;\geq\;
    1 - q\me^{-k N} (1 + k N).
  \end{align*}
  Thus, on the event $\Xi_{N}(a)$,
  \begin{align}\label{eq:harm:sum:lb2}
    \norm{h_{\mathcal{A}_{N}, \mathcal{B}_{N}}^{N}}_{\mu_{N}}
    \;\geq\;
    \mu_{N}\bigl[\mathcal{S}_{0, N}\bigr]
    \bigl(
      1 - q \me^{-k N} ( 1 + k N )
    \bigr).
  \end{align}
  For the upper bound, we bound the first term inside the bracket on the right-hand of \eqref{eq:hABmu_dec} from above by $1$. On the other hand, for any $j \in \{1, \ldots, q\}$, an application of \eqref{eq:harm:claim:1} with $\varepsilon = \me^{-k N} \mu_{N}\bigl[\mathcal{S}_{0, N}\bigr] / \mu_{N}\bigl[\mathcal{S}_{j, N}\bigr]$ yields, on the event $\Xi_{N}(a)$,
  \begin{align}\label{eq:SjS0reg2}
    \norm{h_{\mathcal{A}_{N}, \mathcal{B}_{N}}^{N}}_{\mu_N \vert \mathcal{S}_{j, N}}
    &\;\leq\;
    \norm{h_{\mathcal{M}_{0, N}, \mathcal{M}_{j, N}}^N}_{\mu_N \vert  \mathcal{S}_{j, N}}
    \nonumber\\[.5ex]
    &\;\leq\;
    \frac{\mu_{N}\bigl[\mathcal{S}_{0, N}\bigr]}{\mu_{N}\bigl[\mathcal{S}_{j,N}\bigr]}\, \me^{-k N}\,
    \Biggl(
      1 + k N + \ln \frac{\mu_{N}\bigl[\mathcal{S}_{j,N}\bigr]}{\mu_{N}\bigl[\mathcal{S}_{0, N}\bigr]}
    \Biggr).
  \end{align}
  Hence, we obtain that, on the event $\Xi_{N}(a)$,
  \begin{align}
    \norm{h_{\mathcal{A}_{N}, \mathcal{B}_{N}}^{N}}_{\mu_{N}}
    \;\leq\;
    \mu_{N}\bigl[\mathcal{S}_{0, N}\bigr]\,
    \Biggl(
      1 + q \me^{-k N}
      \Biggl(
        1 + k N + \max_{j \in \{1, \ldots, q\}}
        \ln \frac{\mu_{N}\bigl[\mathcal{S}_{j,N}\bigr]}{\mu_{N}\bigl[\mathcal{S}_{0, N}\bigr]}
      \Biggr)
    \Biggr).
  \end{align}
  Using \eqref{eq:measure_comp} again, there exists a deterministic constant $c \in (0, \infty)$ such that, for every $j \in \{1, \ldots, q\}$, $\ln\bigl(\mu_{N}[\mathcal{S}_{j, N}] / \mu_{N}[\mathcal{S}_{0, N}]\bigr)$ is bounded from above by $c N$.
  \smallskip
  
  \textit{Tunneling regime:} Let $\beta_{2} < \beta < \beta_{4}$, $\mathcal{A}_{N} = \mathcal{M}_{1, N}$, $\mathcal{B}_{N} = \cup_{\ell=2}^{q} \mathcal{M}_{\ell, N}$, and write
  \begin{align}\label{eq:hABmu_reg3}
    \norm{h_{\mathcal{A}_{N}, \mathcal{B}_{N}}^{N}}_{\mu_{N}}
    \;=\;
    \mu_{N}\bigl[\mathcal{S}_{1, N}\bigr]
    \Biggl(
      \norm{h_{\mathcal{A}_{N}, \mathcal{B}_{N}}^{N}}_{\mu_{N} \vert  \mathcal{S}_{1, N}}
      +
      \sum_{j=2}^{q}
      \frac{
        \mu_{N}\bigl[\mathcal{S}_{j, N}\bigr]
      }{\mu_{N}\bigl[\mathcal{S}_{1, N}\bigr]}\,
      \norm{h_{\mathcal{A}_{N}, \mathcal{B}_{N}}^{N}}_{\mu_{N} \vert \mathcal{S}_{j, N}}
      \nonumber\\
      +\,
      \frac{
        \mu_{N}\bigl[\mathcal{S}_{0, N}\bigr]
      }{\mu_{N}\bigl[\mathcal{S}_{1, N}\bigr]}\,
      \norm{h_{\mathcal{A}_{N}, \mathcal{B}_{N}}^{N}}_{\mu_{N} \vert \mathcal{S}_{0, N}}
    \Biggr).
  \end{align}
  The lower bound is obtained via a similar procedure by dropping the second and third terms and using a similar argument as in \eqref{eq:h:lb:arg} to control the first term using \eqref{eq:harm:claim:1}. For an upper bound, we bound the harmonic sums appearing both in the first and third term by $1$, while dealing with the second term as in \eqref{eq:SjS0reg2}. The remaining ratio $\mu_{N}[\mathcal{S}_{0, N}] / \mu_{N}[\mathcal{S}_{1, N}]$ appearing in the final term is controlled via \eqref{eq:measure_comp}. 
  
  Finally, when taking $\beta \geq \beta_{4}$ the proof is analogous with the third term not appearing as $\mathcal{M}_{0, N}$ is not a metastable set. 
\end{proof}
In the next lemma we provide an annealed version of Proposition~\ref{prop:harm_estimate}. This will be used to prove both concentration inequalities and annealed estimates for the harmonic sum in Propositions~\ref{lemma:concharm_sum} and \ref{prop:mean_ln_hsum}.
\begin{lemma}
  \label{lemma:mean_harm_estimate}
  For any $\beta > \beta_{1}$ let $\mathcal{A}_{N}$ and $\mathcal{B}_{N}$ be chosen as in Definition~\ref{def:regimes}. Fix $C_{2} \in (0, C_{1})$, where $C_{1}$ is as in Proposition~\ref{prop:harm_estimate}. Then, as $N \to \infty$,
  \begin{align}
    \label{eq:mean:harm:estimate}
    \mean\Bigl[
      \ln\bigl(
        Z_{N} \norm{h_{\mathcal{A}_{N}, \mathcal{B}_{N}}^{N}}_{\mu_{N}}
      \bigr)
    \Bigr]
    \;=\;
    \mean\Bigl[
      \ln\bigl(
        Z_{N} \mu_{N}\bigl[ \mathcal{S}_{\mathcal{A}, N} \bigr]
      \bigr)
    \Bigr]
    + O\bigl(\me^{-C_{2} N}\bigr).
  \end{align}
\end{lemma}
\begin{proof}
  Let $\Xi_{N}(a)$ be the event defined in \eqref{bound:squaren} and $a > \sqrt{v(C_{1} + \ln q)}$. Then, by applying Proposition~\ref{prop:harm_estimate}, we obtain, for any $N$ large enough,
  \begin{align*}
    &\mean\Bigl[
      \ln\bigl(
        Z_{N} \norm{h_{\mathcal{A}_N, \mathcal{B}_N}^{N}}_{\mu_{N}}
      \bigr)\,
      \indicator_{\Xi_{N}(a)}
    \Bigr]
    \\
    &\mspace{36mu}=\;
    \mean\Bigl[
      \ln\bigl(
        Z_{N} \mu_{N}\bigl[ \mathcal{S}_{\mathcal{A}, N} \bigr]
      \bigr)\,
      \indicator_{\Xi_{N}(a)}
    \Bigr]
    \,+\,
    \ln \bigl(1 + O(\me^{-C_{1} N})\bigr) \prob\bigl[\Xi_{N}(a)\bigr]
    \\
    &\mspace{36mu}=\;
    \mean\Bigl[
      \ln\bigl(
        Z_{N} \mu_{N}\bigl[ \mathcal{S}_{\mathcal{A}, N} \bigr]
      \bigr)\,
      \indicator_{\Xi_{N}(a)}
    \Bigr]
    \,+\,
    O(\me^{-C_{1} N}).
  \end{align*}
  Hence,
  \begin{align}\label{eq:annealed:term3}
    &\mean\Bigl[
      \ln\bigl(
        Z_{N} \norm{h_{\mathcal{A}_N, \mathcal{B}_N}^{N}}_{\mu_{N}}
      \bigr)
    \Bigr]
    \nonumber\\[.5ex]
    &\mspace{36mu}=\;
    \mean\Bigl[
      \ln\bigl(
        Z_{N} \mu_{N}\bigl[ \mathcal{S}_{\mathcal{A}, N} \bigr]
      \bigr)
    \Bigr]
    +
    O(\me^{-C_{1} N})
    +
    \mean\biggl[
      \ln\biggl(
        \frac{
          \norm{h_{\mathcal{A}_N, \mathcal{B}_N}^{N}}_{\mu_{N}}
        }
        {
          \mu_{N}\bigl[ \mathcal{S}_{\mathcal{A}, N} \bigr]
        }
      \biggr)
      \indicator_{\Xi_{N}(a)^{c}}
    \biggr].
  \end{align}
  Thus, we are left with bounding the final term appearing on the right-hand side of \eqref{eq:annealed:term3}. For this purpose, first note that, by the properties of the equilibrium potential, $h_{\mathcal{A}_{N}, \mathcal{B}_{N}}^{N}$, we have that
  \begin{align*}
    \mu_{N}\bigl[\mathcal{A}_{N}\bigr]
    \;\leq\;
    \frac{\norm{h_{\mathcal{A}_N, \mathcal{B}_N}^{N}}_{\mu_{N}}}
    {\mu_{N}\bigl[ \mathcal{S}_{\mathcal{A}, N} \bigr]}
    \;\leq\;
    \frac{1}{\mu_{N}\bigl[\mathcal{A}_{N}\bigr]}
    .
  \end{align*}
  Hence,
  \begin{align*}
    \biggl|
      \mean\biggl[
        \ln\biggl(
          \frac{
            \norm{h_{\mathcal{A}_N, \mathcal{B}_N}^{N}}_{\mu_{N}}
          }
          {
            \mu_{N}\bigl[ \mathcal{S}_{\mathcal{A}, N} \bigr]
          }
        \biggr)
        \indicator_{\Xi_{N}(a)^{c}}
      \biggr]
    \biggr|
    &\;\leq\;
    \mean\biggl[
      \ln\biggl(
        \frac{1}{\mu_{N}\bigl[\mathcal{A}_{N}\bigr]}
      \biggr)
      \indicator_{\Xi_{N}(a)^{c}}
    \biggr]
    \\
    &\;\leq\;
    \ln \mean\biggl[
      \frac{1}{\mu_{N}\bigl[\mathcal{A}_{N}\bigr]}
      \,\bigg|\,
      \Xi_{N}(a)^{c}
    \biggr]\,
    \prob\bigl[\Xi_{N}(a)^{c}\bigr]
    \\
    &\;\leq\;
    \biggl(
      \ln
      \mean\biggl[
        \frac{1}{\mu_{N}\bigl[\mathcal{A}_{N}\bigr]}
      \biggr]\,
      - \ln \prob\bigl[\Xi_{N}(a)^{c}\bigr]
    \biggr)\, \prob\bigl[\Xi_{N}(a)^{c}\bigr].
  \end{align*}
  Since, for any $\xi \in \mathcal{S}_{N}$, $H_{N}(\xi) = \widetilde{H}_{N}(\xi) + \Delta_{N}(\xi)$ and $-N/2 \leq \widetilde{H}_{N}(\xi) \leq 0$, by applying the Cauchy-Schwarz inequality together with Lemma~\ref{coro:betapmapprox} which gives $\mean\bigl[\exp(\pm 2 \beta \Delta_{N}(\xi))\bigr] = \exp(\beta^{2} v (1 + o(1)))$ as $N$ tends to infinity, we obtain, for any $\sigma \in \mathcal{A}_{N}$,
  \begin{align*}
    \mean\biggl[
      \frac{1}{\mu_{N}\bigl[\mathcal{A}_{N}\bigr]}
    \biggr]
    &\;\leq\;
    \sum_{\eta \in \mathcal{S}_{N}}
    \mean\Bigl[
      \me^{-\beta H_{N}(\eta)}\, \me^{\beta H_{N}(\sigma)}
    \Bigr]
    \\
    &\;\leq\;
    \me^{\beta N / 2}\, \sum_{\eta \in \mathcal{S}_{N}}
    \mean\Bigl[\me^{-2\beta \Delta_{N}(\eta)}\Bigr]^{1/2}\,
    \mean\Bigl[\me^{2\beta \Delta_{N}(\sigma)}\Bigr]^{1/2}
    \;=\;
    \me^{N(\beta / 2 + \ln q)}\, \me^{\beta^{2} v (1 + o(1))}.
  \end{align*}
  Hence, using \eqref{eq:xi_control_prob} and the fact that $x \mapsto -x \ln x$ is increasing on $[0, \me^{-1})$, we obtain that for any sufficiently large $N$
  \begin{align}\label{eq:annealed:term3:est}
    &\biggl|
      \mean\biggl[
        \ln\biggl(
          \frac{
            \norm{h_{\mathcal{A}_N, \mathcal{B}_N}^{N}}_{\mu_{N}}
          }
          {
            \mu_{N}\bigl[ \mathcal{S}_{\mathcal{A}, N} \bigr]
          }
        \biggr)
        \indicator_{\Xi_{N}(a)^{c}}
      \biggr]
    \biggr|
    \nonumber\\[.5ex]
    &\mspace{36mu}\leq\;
    \biggl(
      \beta^{2} v \bigl(1 + o(1)\bigr)
      + N \biggl(
        \frac{\beta}{2} + 2 \ln q - \frac{a^{2}}{v} \bigl(1 + o(1)\bigr)
      \biggr)
    \biggr)\,
    \me^{-N (a^{2} / v (1 + o(1)) - \ln q)}.
  \end{align}
  Thus, by combining \eqref{eq:annealed:term3} and \eqref{eq:annealed:term3:est}, the assertion follows.
\end{proof}

\subsection{Concentration estimates}
We are now ready to state some concentration inequalities for the logarithm of the harmonic sum.
\begin{proposition}\label{lemma:concharm_sum}
  For any $\beta > \beta_{1}$ let $\mathcal{A}_{N}$ and $\mathcal{B}_{N}$ be chosen as in Definition~\ref{def:regimes}, and let $C_{2}$ be as in Lemma~\ref{lemma:mean_harm_estimate}. Then, there exists a constant $c \in (0, \infty)$ such that for any $t > 0$, as $N \to \infty$,
  \begin{align}\label{eq:concentration:hsum:ln}
    &\prob\Bigl[
      \bigl|
        \ln\bigl(
          Z_{N} \norm{h^{N}_{\mathcal{A}_{N}, \mathcal{B}_{N}}}_{\mu_{N}}
        \bigr)
        -
        \mean\bigl[
          \ln\bigl(
            Z_{N} \norm{h^{N}_{\mathcal{A}_{N}, \mathcal{B}_{N}}}_{\mu_{N}}
          \bigr)
        \bigr]
      \bigr|
      >
      t
    \Bigr]
    \nonumber\\[.5ex]
    &\mspace{36mu}\leq\;
    2\exp\biggl(
      -\frac{(t-c\, \me^{-C_{2} N})^{2}}{2 \beta^{2} v}
      (1+o(1))
    \biggr)
    +
    2 \exp\bigl(-C_{2} N\bigr).
  \end{align}
\end{proposition}
\begin{proof}
  Let $a > \sqrt{v(C_{2} + \ln q)}$. By Proposition~\ref{prop:harm_estimate} and Lemma~\ref{lemma:mean_harm_estimate} there exists a $c \in (0, \infty)$ such that for all $N$ sufficiently large, on the event $\Xi_{N}(a)$,
  \begin{align*}
    &\Bigl|
      \ln\bigl(
        Z_{N} \norm{h^{N}_{\mathcal{A}_{N}, \mathcal{B}_{N}}}_{\mu_{N}}
      \bigr)
      -
      \mean\Bigl[
        \ln\bigl(
          Z_{N} \norm{h^{N}_{\mathcal{A}_{N}, \mathcal{B}_{N}}}_{\mu_{N}}
        \bigr)
      \Bigr]
    \Bigr|
    \\[.5ex]
    &\mspace{36mu}\leq\;
    \Bigl|
      \ln\bigl(
        Z_{N} \mu_{N}\bigl[\mathcal{S}_{\mathcal{A}_{N}}\bigr]
      \bigr)
      -
      \mean\Bigl[
        \ln\bigl(
          Z_{N} \mu_{N}\bigl[\mathcal{S}_{\mathcal{A}, N}\bigr]
        \bigr)
      \Bigr]
    \Bigr|
    + c\, \me^{-C_{2} N}.
  \end{align*}
  Hence, for any $t > 0$ and $N$ sufficiently large,
  \begin{align}\label{eq:concentration:sum:h:1}
    &\prob\Bigl[
      \bigl|
        \ln\bigl(
          Z_{N} \norm{h^{N}_{\mathcal{A}_{N}, \mathcal{B}_{N}}}_{\mu_{N}}
        \bigr)
        -
        \mean\bigl[
          \ln\bigl(
            Z_{N} \norm{h^{N}_{\mathcal{A}_{N}, \mathcal{B}_{N}}}_{\mu_{N}}
          \bigr)
        \bigr]
      \bigr|
      >
      t
    \Bigr]
    \nonumber\\[.5ex]
    &\mspace{36mu}\leq\;
    \prob\Bigl[
      \bigl|
        \ln\bigl(
          Z_{N} \norm{h^{N}_{\mathcal{A}_{N}, \mathcal{B}_{N}}}_{\mu_{N}}
        \bigr)
        -
        \mean\bigl[
          \ln\bigl(
            Z_{N} \norm{h^{N}_{\mathcal{A}_{N}, \mathcal{B}_{N}}}_{\mu_{N}}
          \bigr)
        \bigr]
      \bigr|
      >
      t,\,
      \Xi_{N}(a)
    \Bigr]
    +
    \prob\bigl[\Xi_{N}(a)^{c}\bigr]
    \nonumber\\[.5ex]
    &\mspace{36mu}\leq\;
    \prob\Bigl[
      \bigl|
        \ln\bigl(
          Z_{N} \mu_{N}\bigl[\mathcal{S}_{\mathcal{A}_{N}}\bigr]
        \bigr)
        -
        \mean\bigl[
          \ln\bigl(
            Z_{N} \mu_{N}\bigl[\mathcal{S}_{\mathcal{A}_{N}}\bigr]
          \bigr)
        \bigr]
      \bigr|
      >
      t - c\, \me^{-C_{2} N}
    \Bigr]
    +
    \prob\bigl[\Xi_{N}(a)^{c}\bigr].
  \end{align}
  By Lemma~\ref{lemma:xi_comp}-(ii) and our choice of $a$, the second term on the right-hand side of \eqref{eq:concentration:sum:h:1} is, for all sufficiently large $N$, bounded by $\me^{-C_{2} N}$. Hence, it remains to bound the first term. To that end we employ an argument analogous to the one used in the proof of Lemma~\ref{lemma:conclnzcap}. For any integer $N \geq 2$, consider the map
  \begin{align*}
    (J_{ij})_{1\leq i<j\leq N}
    \;\longmapsto\;
    \ln\bigl(Z_{N}^{J}\,\mu_{N}^{J}[\mathcal{S}_{\mathcal{A},N}]\bigr),
  \end{align*}
  and, as before, write $H_{N}^{J}$, $Z_{N}^{J}$ and $\mu_{N}^{J}$ to emphasise the dependence on the random array $J = (J_{ij})_{1 \leq i < j \leq N}$. Since $\mathcal{S}_{\mathcal{A}, N}$ is defined in terms of the CWP model and is therefore independent of $J$, \eqref{eq:Lipschitz:H} immediately implies that
  \begin{align*}
    \bigl|
      Z_{N}^{J} \mu_{N}^{J}\bigl[\mathcal{S}_{\mathcal{A}.N}\bigr]
      - Z_{N}^{J'} \mu_{N}^{J'}\bigl[\mathcal{S}_{\mathcal{A}.N}\bigr]
    \bigr|
    \;\leq\;
    \frac{\beta}{N} \sum_{1 \leq i < j \leq N} \bigl| J_{ij} - J_{ij}'\bigr|.
  \end{align*}
  Hence, by applying Corollary~\ref{lemma:general_conc_approx} with $n = N(N-1)/2$, we get for any $t > 0$,
  \begin{align}\label{eq:concentration:sum:h:2}
    \prob\Bigl[
      \bigl|
        \ln\bigl(
          Z_{N} \mu_{N}\bigl[\mathcal{S}_{\mathcal{A}_{N}}\bigr]
        \bigr)
        -
        \mean\bigl[
          \ln\bigl(
            Z_{N} \mu_{N}\bigl[\mathcal{S}_{\mathcal{A}_{N}}\bigr]
          \bigr)
        \bigr]
      \bigr|
      >
      t
    \Bigr]
    \;\leq\;
    2\, \me^{-t^{2}/(2 \beta^{2} v) (1+o(1))}.
  \end{align}
  By combining \eqref{eq:concentration:sum:h:1} and \eqref{eq:concentration:sum:h:2}, we assertion follows.
\end{proof}

\subsection{Annealed estimates}
We first state two lemmas and then prove the main proposition with the annealed estimates for the norm of the harmonic function.
\begin{lemma}\label{lemma:annealed_measure}
  Let $\mathcal{A} \subset \mathcal{S}_{N}$, then as $N \to \infty$
  \begin{align}
    0
    \;\leq\;
    \mean\Bigl[
      \ln\bigl(
        Z_{N} \mu_{N}[\mathcal{A}]
      \bigr)
    \Bigr]
    -
    \ln\bigl(
      \widetilde{Z}_{N} \tilde{\mu}_{N}[\mathcal{A}]
    \bigr)
    \;\leq\;
    \frac{\beta^{2}v}{4}\, (1 + o(1)).
  \end{align}
\end{lemma}
\begin{proof}
  The assertion follows immediately from \cite[Eq. (5.13)-(5.14)]{BdHMPS24}.
\end{proof}
\begin{lemma}\label{lemma:moments_measure}
  Let $\mathcal{A} \subset \mathcal{S}_{N}$, then for any $k \in \mathbb{N}$, as $N \to \infty$,
  \begin{align}
    \widetilde{Z}_{N} \tilde{\mu}_{N}[\mathcal{A}]\, \me^{\beta^{2} v/(4q)(1+o(1))}
    \;\leq\;
    \mean\Bigl[
      \bigl(
        Z_{N} \mu_{N}[\mathcal{A}]
      \bigr)^{k}
    \Bigr]^{1/k}
    \;\leq\;
    \widetilde{Z}_{N} \tilde{\mu}_{N}[\mathcal{A}]\, \me^{k \beta^{2} v/4 (1+o(1))}.
  \end{align}
\end{lemma}
\begin{proof}
  The proof of the upper bound is obtained from \cite[Eq.(5.19)]{BdHMPS24} and Lemma~\ref{coro:betapmapprox}. The lower bound follows by Jensen's inequality and Lemma \ref{lemma:annealed_measure}.
\end{proof}
\begin{proposition}\label{prop:mean_ln_hsum}
  For any $\beta > \beta_{1}$ let $\mathcal{A}_{N}$ and $\mathcal{B}_{N}$ be chosen as in Definition~\ref{def:regimes}.
  \begin{enumerate}[(i)]
  \item
    Let $C_{2}$ be as in Lemma~\ref{lemma:mean_harm_estimate}. Then there exists a constant $c \in (0, \infty)$ such that, for all sufficiently large $N$,
    \begin{align}\label{eq:annealed:h:est}
      \mspace{-36mu}
      -c\, \me^{-C_{2} N}
      \;\leq\;
      \mean\Bigl[
        \ln\bigl(
          Z_{N} \norm{h^{N}_{\mathcal{A}_{N}, \mathcal{B}_{N}}}_{\mu_{N}}
        \bigr)
      \Bigr]
      -
      \ln\bigl(
        \widetilde{Z}_{N} \norm{\tilde{h}^{N}_{\mathcal{A}_{N}, \mathcal{B}_{N}}}_{\tilde{\mu}_{N}}
      \bigr)
      \;\leq\;
      \frac{\beta^{2} v}{4}\, \bigl( 1 + o(1) \bigr).
    \end{align}
    
  \item
    For any $k \geq 1$ and all sufficiently large $N$, we have
    \begin{align}
      \mspace{-36mu}
      \me^{\beta^{2}v/(4q)} \bigl(1 + o(1) \bigr)
      \;\leq\;
      \frac{
        \mean\Bigl[
          \bigl(
            Z_{N} \norm{h^{N}_{\mathcal{A}_{N}, \mathcal{B}_{N}}}_{\mu_{N}}
          \bigr)^{k}
        \Bigr]^{1/k}
      }{
        \widetilde{Z}_{N}
        \norm{\tilde{h}^{N}_{\mathcal{A}_{N}, \mathcal{B}_{N}}}_{\tilde{\mu}_{N}}
      }
      \;\leq\;
      \me^{k \beta^{2} v/4}\, \bigl( 1 + o(1) \bigr).
    \end{align}
  \end{enumerate}
\end{proposition}
\begin{proof}
  \textit{(i)} The assertion \eqref{eq:annealed:h:est} is an immediate consequence of Lemma~\ref{lemma:mean_harm_estimate}, applied to both the CWP and DCWP models, and Lemma \ref{lemma:annealed_measure}.

  \textit{(ii)} Fix $k \geq 1$ and choose $a > \sqrt{v (C_{1} + k \beta + (1 + 2k) \ln q)}$. To bound for the $k$-th moment of the harmonic sum from above, apply Minkowski's inequality and use the fact that $\norm{h_{\mathcal{A}_{N}, \mathcal{B}_{N}}^{N}}_{\mu_{N}} \leq 1$. This yields
  \begin{align}\label{eq:annealed:minkowski:split}
    &\mean\Bigl[
      \bigl(
        Z_{N} \norm{h^{N}_{\mathcal{A}_{N}, \mathcal{B}_{N}}}_{\mu_{N}}
      \bigr)^{k}
    \Bigr]^{1/k}
    \mspace{-6mu}
    \;\leq\;
    \mean\Bigl[
      \bigl(
        Z_{N} \norm{h^{N}_{\mathcal{A}_{N}, \mathcal{B}_{N}}}_{\mu_{N}}
      \bigr)^{k}\,
      \indicator_{\Xi_{N}(a)^{c}}
    \Bigr]^{1/k}
    \!+\,
    \mean\Bigl[
      \bigl(Z_{N}\, \indicator_{\Xi_{N}(a)^{c}} \bigr)^{k}
    \Bigr]^{1/k}.
  \end{align}
  For the first term appearing on the right-hand side of \eqref{eq:annealed:minkowski:split}, Proposition~\ref{prop:harm_estimate} gives, for all sufficiently large $N$,
  \begin{align}
    \mean\Bigl[
      \bigl(
        Z_{N} \norm{h^{N}_{\mathcal{A}_{N}, \mathcal{B}_{N}}}_{\mu_{N}}
      \bigr)^{k}\,
      \indicator_{\Xi_{N}(a)}
    \Bigr]^{1/k}
    \mspace{-6mu}
    \;=\;
    \mean\Bigl[
      \bigl(
        Z_{N} \mu_{N}\bigl[\mathcal{S}_{\mathcal{A}, N}\bigr]
      \bigr)^{k}\,
      \indicator_{\Xi_{N}(a)}
    \Bigr]^{1/k}\,
    \bigl(1 + O(\me^{-C_{1} N})\bigr).
  \end{align}
  By combining Lemma~\ref{lemma:moments_measure} and Proposition~\ref{prop:harm_estimate} (applied to the CWP model), yields, for all sufficiently large $N$,
  \begin{align}
    \mean\Bigl[
      \bigl(
        Z_{N} \mu_{N}\bigl[\mathcal{S}_{\mathcal{A}, N}\bigr]
      \bigr)^{k}\,
      \indicator_{\Xi_{N}(a)}
    \Bigr]^{1/k}
    &\;\leq\;
    \widetilde{Z}_{N} \tilde{\mu}_{N}\bigl[\mathcal{S}_{\mathcal{A}, N}\bigr]\,
    \me^{k \beta^{2} v/4 (1+o(1))}
    \nonumber\\[.5ex]
    &\;\leq\;
    \widetilde{Z}_{N}
    \norm{\tilde{h}^{N}_{\mathcal{A}_{N}, \mathcal{B}_{N}}}_{\tilde{\mu}_{N}}\,
    \me^{k \beta^{2} v/4 (1+o(1))}\,
    \bigl(1 + O(\me^{-C_{1} N})\bigr).
  \end{align}
  To bound the second term appearing on the right-hand side of \eqref{eq:annealed:minkowski:split}, we use both the Minkowski and Cauchy-Schwarz inequality. This gives
  \begin{align*}
    \mean\Bigl[
      \bigl(Z_{N}\, \indicator_{\Xi_{N}(a)^{c}}\bigr)^{k}
    \Bigr]^{1/k}
    \mspace{-6mu}
    &\;\leq\;
    \sum_{\sigma \in \mathcal{S}_{N}} \me^{-\beta \widetilde{H}_{N}(\sigma)}
    \mean\Bigl[
      \me^{-k \beta \Delta_{N}(\sigma)}\, \indicator_{\Xi_{N}(a)^{c}}
    \Bigr]^{1/k}
    \nonumber\\[.5ex]
    &\;\leq\;
    \sum_{\sigma \in \mathcal{S}_{N}} \me^{-\beta \widetilde{H}_{N}(\sigma)}
    \mean\Bigl[ \me^{-2k \beta \Delta_{N}(\sigma)} \Bigr]^{1/(2k)}
    \prob\bigl[ \Xi_{N}(a)^{c} \bigr]^{1/(2k)}
    \nonumber\\[.5ex]
    &\;\leq\;
    \widetilde{Z}_{N}\,
    \max_{\sigma \in \mathcal{S}_{N}}
    \mean\Bigl[ \me^{-2k \beta \Delta_{N}(\sigma)} \Bigr]^{1/(2k)}
    \prob\bigl[ \Xi_{N}(a)^{c} \bigr]^{1/(2k)}.
  \end{align*}
  By using the fact that $\widetilde{Z}_{N} \leq \exp(N(\beta/2 + \ln q))$ and Lemma~\ref{lemma:xi_comp} (apply (i) with $\beta$ replaced by $2 \beta k$ and (ii)), we obtain that, for all sufficiently large $N$,
  \begin{align}\label{eq:annealed:kth:est:tail}
    \mean\Bigl[
      \bigl(Z_{N}\, \indicator_{\Xi_{N}(a)^{c}}\bigr)^{k}
    \Bigr]^{1/k}
    \mspace{-6mu}
    \;\leq\;
    \me^{k \beta^{2} v/2 (1 + o(1))}\, \me^{N(\beta/2 + \ln q)}\, \me^{-N(a^{2}/v (1 + o(1)) - \ln q)/(2k)}.
  \end{align}
  Thus, by combining the above estimates and using the fact that $\widetilde{Z}_{N} \norm{\tilde{h}_{\mathcal{A}_{N}, \mathcal{B}_{N}}^{N}}_{\tilde{\mu}_{N}} \geq 1$, we conclude that, by our choice of $a$, there exists a $C_{3} \in (0, C_{1})$ such that, for all sufficiently large $N$,
  \begin{align}
    \mean\Bigl[
      \bigl(
        Z_{N} \norm{h^{N}_{\mathcal{A}_{N}, \mathcal{B}_{N}}}_{\mu_{N}}
      \bigr)^{k}
    \Bigr]^{1/k}
    \;\leq\;
    \widetilde{Z}_{N} \norm{\tilde{h}_{\mathcal{A}_{N}, \mathcal{B}_{N}}^{N}}_{\tilde{\mu}_{N}}\,
    \me^{k\beta^{2} v/4 (1+o(1))}\,
    \bigl(1 + \me^{-C_{3} N}\bigr).
  \end{align}
  It remains to prove a lower bound on the $k$-th moment of the harmonic sum. By Minkowski's inequality and Proposition~\ref{prop:harm_estimate} we get that
  \begin{align}\label{eq:annealed:minkowski:lb}
    &\mean\Bigl[
      \bigl(
        Z_{N} \norm{h^{N}_{\mathcal{A}_{N}, \mathcal{B}_{N}}}_{\mu_{N}}
      \bigr)^{k}
    \Bigr]^{1/k}
    \nonumber\\
    &\mspace{36mu}\geq\;
    \mean\Bigl[
      \bigl(
        Z_{N} \norm{h^{N}_{\mathcal{A}_{N}, \mathcal{B}_{N}}}_{\mu_{N}}
      \bigr)^{k}\,
      \indicator_{\Xi_{N}(a)}
    \Bigr]^{1/k}
    \nonumber\\
    &\mspace{36mu}\geq\;
    \biggl(
      \mean\Bigl[
        \bigl(
          Z_{N} \mu_{N}\bigl[\mathcal{S}_{\mathcal{A}, N}\bigr]
        \bigr)^{k}
      \Bigr]^{1/k}
      \!-
      \mean\Bigl[
        \bigl(
          Z_{N}\, \indicator_{\Xi_{N}(a)^{c}}
        \bigr)^{k}
      \Bigr]^{1/k}
    \biggr)\,
    \bigl(1 + O(\me^{-C_{1} N})\bigr).
  \end{align}
  By using Lemma~\ref{lemma:moments_measure} together with Proposition~\ref{prop:harm_estimate} (applied to the CWP model), we obtain that, for all sufficiently large $N$,
  \begin{align}\label{eq:annealed:mean:lb}
    \mean\Bigl[
      \bigl(
        Z_{N} \mu_{N}\bigl[\mathcal{S}_{\mathcal{A}, N}\bigr]
      \bigr)^{k}
    \Bigr]^{1/k}
    &\;\geq\;
    \widetilde{Z}_{N} \tilde{\mu}_{N}\bigl[\mathcal{S}_{\mathcal{A}, N}\bigr]\,
    \me^{\beta^{2} v/(4q) (1+o(1))}
    \nonumber\\[.5ex]
    &\;\geq\;
    \widetilde{Z}_{N}
    \norm{\tilde{h}^{N}_{\mathcal{A}_{N}, \mathcal{B}_{N}}}_{\tilde{\mu}_{N}}\,
    \me^{\beta^{2} v/(4q) (1+o(1))}\,
    \bigl(1 + O(\me^{-C_{1} N})\bigr).
  \end{align}
  Hence, by combining \eqref{eq:annealed:kth:est:tail} and \eqref{eq:annealed:mean:lb} with \eqref{eq:annealed:minkowski:lb} and using that $\widetilde{Z}_{N} \norm{\tilde{h}_{\mathcal{A}_{N}, \mathcal{B}_{N}}^{N}}_{\tilde{\mu}_{N}} \geq 1$, we finally get that, for all sufficiently large $N$,
  \begin{align}
    \mean\Bigl[
      \bigl(
        Z_{N} \norm{h^{N}_{\mathcal{A}_{N}, \mathcal{B}_{N}}}_{\mu_{N}}
      \bigr)^{k}
    \Bigr]^{1/k}
    \;\geq\;
    \widetilde{Z}_{N} \norm{\tilde{h}_{\mathcal{A}_{N}, \mathcal{B}_{N}}^{N}}_{\tilde{\mu}_{N}}\,
    \me^{\beta^{2} v/(4q) (1+o(1))}\,
    \bigl(1 - \me^{-C_{3} N}\bigr).
  \end{align}
  This completes the proof.
\end{proof}

Having completed the analysis of both the numerator (Section \ref{sec:final}) and the denominator (Section \ref{section:cap_mu}) of equation \eqref{eq:golden_formula}, we now proceed to prove Theorem \ref{theo:ratio_conc} and Theorem \ref{theo:ratio_mom}. The proof relies on Lemmas \ref{theo:conc_hit_time} and \ref{theo:ann_log_hitting}, providing a concentration inequality and an annealed estimate for the mean hitting time. 
\begin{lemma}\label{theo:conc_hit_time}
  For any $\beta > \beta_{1}$ let $\mathcal{A}_{N}$ and $\mathcal{B}_{N}$ be chosen as in Definition~\ref{def:regimes}. Then for any $t > 0$,
  %
  \begin{align}\label{eq:concentration:mean:ln}
    &
    \limsup_{N \to \infty}
    \prob\Bigl[
      \Bigl|
        \ln \Mean^{N}_{\nu_{\mathcal{A}_{N}, \mathcal{B}_{N}}}\bigl[\tau_{\mathcal{B}_{N}}^{N}\bigr]
        -
        \mean\Bigl[
          \ln \Mean^{N}_{\nu_{\mathcal{A}_{N}, \mathcal{B}_{N}}}\bigl[\tau_{\mathcal{B}_{N}}^{N}\bigr]
        \Bigr]
      \Bigr|
      > t
    \Bigr]
    \;\leq\;
    4\, \me^{-t^{2}/(8 \beta^{2} v)}.
  \end{align}
\end{lemma}
\begin{proof}
  In view of \eqref{eq:golden_formula}, we obtain for any $t>0$,
  \begin{align*}
    &\prob\Bigl[
      \Bigl|
        \ln \Mean^{N}_{\nu_{\mathcal{A}_{N}, \mathcal{B}_{N}}}\bigl[\tau_{\mathcal{B}_{N}}^{N}\bigr]
        -
        \mean\Bigl[
          \ln \Mean^{N}_{\nu_{\mathcal{A}_{N}, \mathcal{B}_{N}}}\bigl[\tau_{\mathcal{B}_{N}}^{N}\bigr]
        \Bigr]
      \Bigr|
      > t
    \Bigr]
    \nonumber\\[.5ex]
    &\mspace{36mu}\leq\;
    \prob\Bigl[
      \Bigl|
        \ln \bigl(Z_{N} \norm{h_{\mathcal{A}_{N}, \mathcal{B}_{N}}^{N}}_{\mu_{N}}\bigr)
        -
        \mean\Bigl[
          \ln \bigl(Z_{N} \norm{h_{\mathcal{A}_{N}, \mathcal{B}_{N}}^{N}}_{\mu_{N}}\bigr)
        \Bigr]
      \Bigr|
      > t/2
    \Bigr]
    \\
    &\mspace{72mu}+
    \prob\Bigl[
      \Bigl|
        \ln \bigl(Z_{N} \capacity_{N}(\mathcal{A}_{N}, \mathcal{B}_{N})\bigr)
        -
        \mean\Bigl[
          \ln \bigl(Z_{N} \capacity_{N}(\mathcal{A}_{N}, \mathcal{B}_{N})\bigr)
        \Bigr]
      \Bigr|
      > t/2
    \Bigr].
  \end{align*}
  Hence, \eqref{eq:concentration:mean:ln} follows immediate from \eqref{eq:concentration:hsum:ln} and \eqref{eq:conccap}. 
\end{proof}
\begin{lemma}\label{theo:ann_log_hitting}
  For any $\beta > \beta_{1}$ let $\mathcal{A}_{N}$ and $\mathcal{B}_{N}$ be chosen as in Definition~\ref{def:regimes}. Then for any sufficiently large $N$,
  \begin{align}\label{eq:annealed_log_hitting}
    - \frac{\beta^{2} v}{4}\, \bigl(1 + o(1)\bigr)
    \;\leq\;
    \mean\Bigl[
      \ln \Mean^{N}_{\nu_{\mathcal{A}_{N}, \mathcal{B}_{N}}}\bigl[ \tau_{\mathcal{B}_N}^N \bigr]
    \Bigr]
    -
    \ln \widetilde{\Mean}^{N}_{\tilde{\nu}_{\mathcal{A}_{N},\mathcal{B}_{N}}}\bigl[ \tilde{\tau}_{\mathcal{B}_{N}}^{N} \bigr]
    \;\leq\;
    \frac{\beta^{2}v}{2}\, \bigl(1 + o(1)\bigr).
  \end{align}
\end{lemma}
\begin{proof}
  In view of \eqref{eq:golden_formula}, the assertion follows immediately from Proposition~\ref{prop:mean_ln_hsum} and Proposition~\ref{theo:meanzcap}-(i).
\end{proof}
Although the proofs of Theorem~\ref{theo:ratio_conc} and \ref{theo:ratio_mom} follow along the lines of \cite[Theorem~2.12]{BdHMPS24}, we provide details for the reader's convenience.
\begin{proof}[Proof of Theorem~\ref{theo:ratio_conc}]
  First, note that Lemma~\ref{theo:ann_log_hitting} implies that for any $\varepsilon > 0$ and for all sufficiently large $N$,
  \begin{align*}
    \frac{
      \Mean_{\nu_{\mathcal{A}_{N}, \mathcal{B}_{N}}}^{N}\bigl[
        \tau_{\mathcal{B}_{N}}^{N}
      \bigr]\,
      \me^{-\beta^{2} v / 4 - \varepsilon}
    }{
      \exp\Bigl(
        \mean\bigl[
          \ln \Mean_{\nu_{\mathcal{A}_{N}, \mathcal{B}_{N}}}^{N}\bigl[
            \tau_{\mathcal{B}_{N}}^{N}
          \bigr]
        \bigr]
      \Bigr)
    }
    \;\leq\;
    \frac{
      \Mean_{\nu_{\mathcal{A}_{N}, \mathcal{B}_{N}}}^{N}\bigl[
        \tau_{\mathcal{B}_{N}}^{N}
      \bigr]\,
    }{
      \widetilde{\Mean}_{\tilde{\nu}_{\mathcal{A}_{N}, \mathcal{B}_{N}}}^{N}\bigl[
        \tilde{\tau}_{\mathcal{B}_{N}}^{N}
      \bigr]\,
    }
    \;\leq\;
    \frac{
      \Mean_{\nu_{\mathcal{A}_{N}, \mathcal{B}_{N}}}^{N}\bigl[
        \tau_{\mathcal{B}_{N}}^{N}
      \bigr]\,
      \me^{+\beta^{2} v / 2 + \varepsilon}
    }{
      \exp\Bigl(
        \mean\bigl[
          \ln \Mean_{\nu_{\mathcal{A}_{N}, \mathcal{B}_{N}}}^{N}\bigl[
            \tau_{\mathcal{B}_{N}}^{N}
          \bigr]
        \bigr]
      \Bigr)
    }.
  \end{align*}
  Hence for any $t > 0$ and any $\varepsilon \in (0, t)$, 
  \begin{align*}
    &\liminf_{N \to \infty}
    \prob\Biggl[
      \me^{-t - \beta^{2} v/4}
      \leq
      \frac{%
        \Mean^{N}_{\nu_{\mathcal{A}_{N}, \mathcal{B}_{N}}}\bigl[ \tau^{N}_{\mathcal{B}_N} \bigr]
      }{%
        \widetilde{\Mean}^{N}_{\tilde{\nu}_{\mathcal{A}_{N}, \mathcal{B}_{N}}}\bigl[\tilde{\tau}^N_{\mathcal{B}_N}\bigr]
      }
      \leq
      \me^{t + \beta^{2} v/2}
    \Biggr]
    \\
    &\mspace{36mu}\geq\;
    1 -
    \limsup_{N \to \infty}
    \prob\Bigl[
      \Bigl|
        \ln \Mean^{N}_{\nu_{\mathcal{A}_{N}, \mathcal{B}_{N}}}\bigl[\tau_{\mathcal{B}_{N}}^{N}\bigr]
        -
        \mean\Bigl[
          \ln \Mean^{N}_{\nu_{\mathcal{A}_{N}, \mathcal{B}_{N}}}\bigl[\tau_{\mathcal{B}_{N}}^{N}\bigr]
        \Bigr]
      \Bigr|
      > t - \varepsilon
    \Bigr]
    \\[1ex]
    &\mspace{31mu}\overset{%
      \mspace{-6mu}\eqref{eq:concentration:mean:ln}\mspace{-6mu}
    }{\;\geq\;}
    1 - 4 \me^{-(t-\varepsilon)^{2}/(8 \beta^{2} v)}.
  \end{align*}
  Since $\varepsilon \in (0, t)$ was arbitrary, letting $\varepsilon \downarrow 0$ gives the claimed result.
\end{proof}
\begin{proof}[Proof of Theorem~\ref{theo:ratio_mom}]
  Fix $k \geq 1$. Using the representation \eqref{eq:golden_formula} of the mean hitting time, the Cauchy-Schwarz inequality, Proposition~\ref{prop:mean_ln_hsum}-(ii) and Proposition~\ref{theo:meanzcap}-(ii) we obtain for all sufficiently large $N$,
  \begin{align}\label{eq:theo:moment:ub}
    \mean\Bigl[
      \Mean_{\nu_{\mathcal{A}_{N}, \mathcal{B}_{N}}^{N}}\bigl[
        \tau_{\mathcal{B}_{N}}^{N}
      \bigr]^{k}
    \Bigr]^{1/k}
    &\;\leq\;
    \mean\Bigl[
      \bigl(
        Z_{N} \norm{h^{N}_{\mathcal{A}_{N}, \mathcal{B}_{N}}}_{\mu_{N}}
      \bigr)^{2k}
    \Bigr]^{1/(2k)}
    \mean\Bigl[
      \bigl(
        Z_{N} \capacity_{N}(\mathcal{A}_{N}, \mathcal{B}_{N})
      \bigr)^{-2k}
    \Bigr]^{1/(2k)}
    \nonumber\\[.5ex]
    &\;\leq\;
    \widetilde{\Mean}_{\tilde{\nu}_{\mathcal{A}, \mathcal{B}_{N}}}^{N}\bigl[
      \tilde{\tau}_{\mathcal{B}_{N}}^{N}
    \bigr]\,
    \me^{k \beta^{2} v}\, \bigl( 1 + o(1) \bigr).
  \end{align}
  For the lower bound, Jensen's inequality together with Lemma~\ref{theo:ann_log_hitting} gives, for all sufficient large $N$,
  \begin{align}\label{eq:theo:moment:lb}
    \mean\Bigl[
      \Mean_{\nu_{\mathcal{A}_{N}, \mathcal{B}_{N}}^{N}}\bigl[
        \tau_{\mathcal{B}_{N}}^{N}
      \bigr]^{k}
    \Bigr]^{1/k}
    \;\geq\;
    \exp\Bigl(
      \mean\Bigl[
        \ln \Mean_{\nu_{\mathcal{A}_{N}, \mathcal{B}_{N}}^{N}}\bigl[
          \tau_{\mathcal{B}_{N}}^{N}
        \bigr]
      \Bigr]
    \Bigr)
    \overset{\eqref{eq:annealed_log_hitting}}{\;\geq\;}
    \widetilde{\Mean}_{\tilde{\nu}_{\mathcal{A}, \mathcal{B}_{N}}}^{N}\bigl[
      \tilde{\tau}_{\mathcal{B}_{N}}^{N}
    \bigr]\,
    \me^{-\beta^{2} v/4 (1 + o(1))}.
  \end{align}
  Combining the two bounds yields the assertion \eqref{eq:theo:ratio_mom}.
\end{proof}
%
%
%
%




\end{document}